 \documentclass[11pt]{amsart}
 \usepackage{amscd}
 \usepackage{amsxtra}
 \usepackage{amsthm}
 \usepackage{extarrows}
 \usepackage{latexsym}
 \usepackage{mathrsfs} 
 \usepackage{pst-all,multido,ifthen}
 \usepackage{amsmath}
 \usepackage{amssymb}
 \usepackage{amsfonts}
 \usepackage{stmaryrd}
 \usepackage[all]{xy}
 \usepackage{tikz}
 \usepackage{tikz-cd}
 \usepackage[english]{babel}
 \usepackage{mathtools}
 \usepackage{upgreek}
 \usepackage{hyperref}

\newtheorem{theorem}{Theorem}[section]
\newtheorem{definition}[theorem]{Definition}
\newtheorem{lemma}[theorem]{Lemma}
\newtheorem{fact}[theorem]{Fact}
\newtheorem{proposition}[theorem]{Proposition}
\newtheorem{corollary}[theorem]{Corollary}

\newtheorem{ex}[theorem]{Example}
\newtheorem{remark}[theorem]{Remark}

\def\Aut{\textup{Aut}}

\def\char{\textup{char}}
\def\Chi{\textup{Chi}}

\def\crys{\textup{crys}}
\def\circul{\textup{circ}}

\def\depth{\textup{depth}}
\def\End{\textup{End}}
\def\soc{\textup{soc}}

\def\ss{\textup{ss}} 
\def\Ext{\textup{Ext}}

\def\Frac{\textup{Frac}}
\def\Frob{\textup{Frob}}

\def\Hom{\textup{Hom}}
\def\Im{\textup{Im}}
\def\Isom{\textup{Isom}}

\def\Ker{\textup{Ker}}
\def\l{\textup{l}}

\def\r{\textup{r}}
\def\q{\textup{q}}

\def\rad{\textup{rad}}

\def\Spec{\textup{Spec}}

 \usepackage{scalerel}
 \usepackage{stackengine,wasysym}
 \stackMath
 \usepackage{tikz}

 \usepackage{pgffor}
 \makeatletter
 \newcommand{\medoplus}{\mathbin{\mathpalette\make@small\oplus}}
 \newcommand{\medotimes}{\mathbin{\mathpalette\make@small\otimes}}
 
 \newcommand{\make@small}[2]{%
 \vcenter{\hbox{%
 \scalebox{1.6}{$\m@th#1#2$}%
 }}%
 }
 \makeatother

\begin{document}
\title[On lifting representations and actions]{On lifting representations and actions on curves of the metacyclic groups $C_{p^s}\rtimes C_m$}
\author{Huy Dang and Adrian Vasiu}
\maketitle
 
\noindent
{\bf ABSTRACT.} For a prime $p$, a pair $(s,m)\in\mathbb N^2$ with $m$ relatively prime to $p$, a homomorphism $\chi:C_m\rightarrow\Aut(C_{p^s})$, and an algebraically closed field $k$ of characteristic $p$, we consider the semidirect product $G=C_{p^s}\rtimes_{\chi} C_m$, the $p$-Sylow subgroup $H\cong C_{p^s}$ of $G$, and a $k[G]$-module $V$. Let $R$ be a complete discrete valuation ring of residue field $k$ and mixed characteristic $(0,p)$ that contains a primitive $p^s$-th root of unity. If $\chi$ is injective, we present two necessary and sufficient criteria for lifting $V$ to an $R[G]$-module $\mathcal V$ which is a free $R$-module: (i) when no extra requirement is made on $\mathcal V$ and (ii) when we require $\mathcal V^H=\{0\}$. The criteria correct several results in the literature and we use them to prove that, if $\chi$ is injective and $G$ acts faithfully on a connected smooth projective curve $X$ over $k$, then, under mild hypotheses satisfied if $X\rightarrow X/G$ is a Harbater--Katz--Gabber cover, the $k[G]$-module $H^0(X,\Omega_X)$ has a lift $\mathcal V$ to $R$ with $\mathcal V^H=\{0\}$. Also, with $B:=R[\frac{1}{p}]$, we prove the following obstruction when $p$ is odd, $m$ is even, $\chi$ is injective, and $X/C_{p^s}\cong\mathbb P^1_k$: if no such lift $\mathcal V$ exists with the $B[H]$-module $\mathcal V\otimes_R B$ defined over $\mathbb Q$, then the action of $G$ on $X$ does not lift to $R$.
 
\bigskip\noindent
{\bf KEY WORDS:} cohomology, curve, differential form, field, genus, group, module, ramification, representation, and ring.
 
\bigskip\noindent
{\bf MSC 2020: 05E10, 13F35, 14E22, 14F30, 14F40, 14G17, 14H30, 14H37, 20C10, 20C15, and 20C20.} 

\section{Introduction}\label{S1}

For $l\in\mathbb N$ and $n\in\mathbb Z$, let $[n]_l:=n+l\mathbb Z\in\mathbb Z/l\mathbb Z$ and the {\it interval of integers} $\llbracket n,l\rrbracket:=\{i\in\mathbb Z|n\le i\le l\}$; so $\llbracket n,l\rrbracket=\emptyset$ if $n>l$. Let 
$$C_l:=\mathbb Z/l\mathbb Z=\{[i]_l|i\in\llbracket0,l-1\rrbracket\}$$ 
be the cyclic additive group of order $l$ with the `standard' generator $[1]_l$. For an $l$-th root of unity $\zeta$ in some ring and $n\in\mathbb Z$, let $\zeta^{[n]_l}:=\zeta^n$. The multiplicative group of units of a ring $S$ is denoted by $S^{\times}$.

Let $p$ be a prime. When we view $\mathbb Z/p\mathbb Z$ as a field, we denote it by $\mathbb F_p$. Let $s\in\mathbb N$. For $p>2$ we have an isomorphism $\Aut(C_{p^s})\cong\mathbb F_p^{\times}\times C_{p^{s-1}}$. For $p=2$, $\Aut(C_{2^s})$ is isomorphic to $\mathbb F_2^{\times}\times C_2\times C_{2^{s-2}}$ if $s\ge 3$, to $\mathbb F_2^{\times}\times C_2$ if $s=2$, and to $\mathbb F_2^{\times}$ if $s=1$. Thus the epimorphism $\Aut(C_{p^s})\rightarrow\mathbb F_p^{\times}$ that maps an automorphism $a$ to the reduction of $a([1]_{p^s})$ modulo $p$ has a unique section. Hence for $m\in\mathbb N$ relatively prime to $p$, each character $\chi:C_m\rightarrow\mathbb F_p^{\times}$ lifts uniquely to a character $C_m\rightarrow\Aut(C_{p^s})$ denoted also by $\chi$.

We consider the semidirect product 
$$G:=C_{p^s}\rtimes_{\chi} C_m$$ 
defined by (the lift $\chi:C_m\rightarrow\Aut(C_{p^s})$ of) $\chi:C_m\rightarrow\mathbb F_p^{\times}$. As $G$ is non-abelian (equivalently, non-cyclic) if $\chi$ is non-trivial, we use $\cdot$ for the binary operation on $G$. Let $H$ be the unique $p$-Sylow subgroup of $G$. As $H$ is cyclic of order $p^s$, let $\tau$ be a generator of $H$. We use $\tau$ to identify $H=C_{p^s}$, $\tau^n=[n]_{p^s}$ for each $n\in\mathbb Z$, and $\Aut(H)=\Aut(C_{p^s})$. For $(a,x)\in\Aut(H)\times H$ let $x^{a}:=a(x)$. Let $\sigma\in G$ be an element of order $m$; we use it to identify $\Sigma=G/H=C_m$ and $\sigma^n=[n]_m$ for each $n\in\mathbb Z$, where $\Sigma$ is the subgroup of $G$ generated by $\sigma$. We have an identity $\sigma\tau\sigma^{-1}=\tau^{\chi(\sigma)}$, where $\chi(\sigma):=\chi(\sigma H)$.

Let $k$ be an algebraically closed field of characteristic $p$. Let $R_0$ be the ring of $p$-typical Witt vectors with coefficients in $k$; so it is the unique complete discrete valuation ring of mixed characteristic $(0,p)$, residue field $k$, and uniformizer $p$. Let $B_0:=R_0[\frac{1}{p}]$ be the field of fractions of $R_0$. For $l\in\mathbb N$, let $B_l$ be the finite field extension of $B_0$ obtained by adjoining a primitive $p^l$-th root of unity $\zeta_{p^l}$ and let $R_l$ be the normalization of $R_0$ in $B_l$. As
$$[B_l:B_0]=p^l-p^{l-1},$$ 
$R_l$ is a free $R_0$-module of rank $p^l-p^{l-1}$. Let $R$ be the normalization of $R_0$ in a finite field extension $B$ of $B_0$; so $B=R[\frac{1}{p}]$ and $R$ is a complete discrete valuation ring of mixed characteristic $(0,p)$. Let $\upsilon$ be a uniformizer of $R$; we have $k=R_0/R_0p=R/R\upsilon$. In practice, $R$ often contains $R_s$.

For a commutative ring $S$ with identity element, each $S[G]$-module $M$ is assumed to be a left finitely generated $S[G]$-module. However, we view $S[G]$ as an $(S[G],S[H])$-bimodule and hence for each $S[H]$-module $N$, the tensor product $S[G]\otimes_{S[H]} N$ is an $S[G]$-module such that we have an isomorphism of $S[H]$-modules
\begin{equation}\label{EQ1}
S[G]\otimes_{S[H]} N=\oplus_{i=0}^{m-1} S[H]\sigma^{i}\otimes_{S[H]} N\cong\oplus_{i=0}^{m-1} S[H]\otimes_{\chi(\sigma^i),S[H]} N.
\end{equation} 

Let $\soc_{S[G]}(M)$ be the socle of $M$. For $l\in\mathbb N\cup\{0\}$, let $lM$ be the $S[G]$-module which is the direct sum of $l$ copies of $M$. For a subset $A$ of $G$, let $M^A:=\{v\in M|g(v)=v\;\textup{for each}\; g\in A\}$. If $A=\{g\}$ has one element, then we often denote $M^{\{g\}}$ simply by $M^g$. 

\begin{definition}\label{D1}
We say that a $k[G]$-module $V$ lifts to $R$ if there exists an $R[G]$-module $\mathcal V$, called a lift of $V$ to $R$, whose reduction modulo $\upsilon$ is $V$ and which as an $R$-module is free (i.e., the multiplication by $\upsilon$ on $\mathcal V$ is injective).
\end{definition}

{\bf Motivation.} This paper grew out of the intent to apply the classification results of \cite{KT1} and \cite{KT2} on $k[G]$-modules that lift to $R_s$. As it turned out that these classifications miss many $k[G]$-modules that lift to $R_s$, our refreshed motivation, the new classification, and applications to curves are intermingled with references to  \cite{KT1} and \cite{KT2} as follows.

The interest in lifting $k[G]$-modules stems out from the following two problems on an arbitrary finite group $\Gamma$: (i) classify all faithful global actions of $\Gamma$ on connected smooth projective curves $X$ over $k$ that satisfy certain axioms and that lift to characteristic $0$ (i.e., to some $R$), and (ii) classify all faithful local actions of $\Gamma$ on $k[[y]]$ that lift to characteristic $0$ (i.e., to some $R[[y]]$). Here $y$ is an indeterminate. 

If $\Gamma$ has a unique $p$-Sylow subgroup, then these two problems are interrelated as to every faithful local action $\phi:\Gamma\rightarrow \Aut_k(k[[y]])$, the Harbater--Katz--Gabber Theorem (see \cite{K}, pp.\ 70--71 and Main Thm.\ 1.4.1; for the case when $\Gamma$ is a $p$-group, see also \cite{H}, Cor.\ 2.4) associates to it a connected smooth projective curve $X$ over $k$ equipped with a faithful global action $\rho_{\phi}:\Gamma\rightarrow\Aut(X)$ such that $X/\Gamma\cong \mathbb P^1_k$ and the morphism $X\rightarrow X/\Gamma$ is ramified only above one or two points of $X/\Gamma$: above the first one it is totally and wildly ramified, and above the second one, if it exists, it is tamely ramified; $X$ is nowadays called the Harbater--Katz--Gabber (HKG) cover of $\mathbb P^1_k$ associated to $\phi$ (e.g., see \cite{P}, Sect.\ 4B, pp.\ 308--309).

Obus Conjecture (see \cite{O3}, Conj.\ 1.9) predicts that if $\Gamma=G$, then for each $\phi$ the KGB obstruction defined in \cite{CGH2}, Sect.\ 1, p.\ 539 is the only obstruction to the existence of a lift of $\phi$ to characteristic $0$. 

By definition, the KGB obstruction of $\phi:\Gamma\rightarrow\Aut_k(k[[y]])$ vanishes if there exists a smooth projective curve $\mathfrak X$ over a field $\mathfrak B$ of characteristic $0$ equipped with a faithful global action $\Gamma\rightarrow\Aut(\mathfrak X)$ such that for every subgroup $\Delta$ of $\Gamma$, the quotient curves $\mathfrak X/\Delta$ and $X/\Delta$ have the same genus.

One says that $\Gamma$ is a local Oort group for $k$ if every faithful local action $\phi$ of $\Gamma$ over $k$ lifts to characteristic $0$ (see \cite{CGH1}, Sect.\ 2, p.\ 853). The finite groups $\Gamma$ for which the KGB obstruction vanishes for every faithful local action $\phi:\Gamma\rightarrow\Aut_k(k[[y]])$ are classified in \cite{CGH2}, Thm.\ 1.2 and, as a correction, in \cite{CGH3}: they are the cyclic groups $C_l$ with $p$ arbitrary, the dihedral groups $D_{p^s}$ of order $2p^s$ with $p$ odd, and the group $A_4$ and the generalized quaternion groups $Q_{2^l}$ with $l\ge 4$ and presentation 
$$\langle g,h\mid g^{2^{l-1}}=1,\; h^2=g^{2^{l-2}},\; hgh^{-1}=g^{-1}\rangle$$ 
with $p=2$. If $p=2$, then the group $A_4$ is a local Oort group for $k$ by \cite{O2}, Thm.\ 1.2 but for each $l\ge 4$ the group $Q_{2^l}$ is not a local Oort group for $k$ by \cite{BrW}, Prop.\ 4.7 and Thm.\ 4.8. Oort's Conjecture (e.g., see \cite{P}, Sect.\ 1) predicted that the cyclic groups $C_{p^s}$ are local Oort groups for $k$; the fact that it is true follows from Pop's proof of the General Oort Conjecture (see \cite{P}, Thm.\ 1.1), which uses the cyclic lifting theorem of Obus--Wewers \cite{OW}, Thm.\ 1.4 and Rmk.\ 1.5(i). So Obus Conjecture is a natural continuation of Oort's Conjecture which for the particular case of Dihedral groups $D_{p^s}$ could complete the classification of all local Oort groups over $k$. 

Several positive cases are known for dihedral groups. For $p$ odd, $D_p$ is a local Oort group over $k$ by \cite{BoW}, Thm.\ 1.2. For $p=3$, $D_9$ is a local Oort group over $k$ by \cite{O3}, Thm.\ 8.6. For $p$ odd, every faithful local action of $D_{p^2}$ on $k[[y]]$ lifts to characteristic $0$ if the first upper jump is congruent to $1$ modulo $p$ by \cite{O3}, Thm.\ 8.7. Each faithful local action of $G$ on $k[[y]]$ with the $s$-tuple of upper jumps congruent to $\bigl(m-1,p(m-1),\ldots,p^{s-1}(m-1)\bigr)$ modulo $pm$ lifts to characteristic $0$ by \cite{O3}, Thm.\ 8.8. For $p=5$ (resp.\ $p=3$), $D_{25}$ (resp.\ $D_{27}$) is a local Oort group over $k$ by \cite{DDKOT}, Thm.\ 1.3.

As the Harbater--Katz--Gabber Theorem has an analog over each spectrum $\Spec R$ (see \cite{K}, Prop.\ 1.4.2), a lift of $\phi$ itself to characteristic $0$ forces the KGB obstruction to vanish for $\phi$ (see also Theorems \ref{C6} and \ref{C7}). The converse of this is highly non-trivial and it is predicted to hold by \cite{O3}, Conj.\ 1.9 when $\Gamma$ is $G$.

In \cite{KT2}, Sect.\ 5.1, taking $(p,s,m)=(5,3,2)$ and $\chi$ injective, a faithful local action of $G\cong D_{125}$ is considered whose three lower jumps are $l_1=9$, $l_2=189$, and $l_3=4689$ and whose associated HKG cover $X$ has genus $11656$; so $V_X:=H^0(X,\Omega_X)$ is a $k$-vector space of dimension $11656$ on which $G$ acts by functoriality. In \cite{KT2}, Sect.\ 5.1 it is claimed that for the $k[G]$-module $V_X$ there exists no lift $\mathcal V$ of it to any $R$ with $\mathcal V^H=\{0\}$. Hence \cite{KT2} concluded that Obus Conjecture is not true, that the KGB obstruction is not the only one, and that a new type of obstruction exists: whether the $k[G]$-modules of differential forms on curves have or do not have lifts to some $R$ with no non-zero elements fixed by $H$.

Example \ref{EX14} shows that in fact the $k[G]$-module $H^0(X,\Omega_X)$ does lift to $R_3$. More precisely, we have the following general theorem which is only a reformulation of Theorem \ref{T6} that does not use the extra language and notation introduced in what follows.

\begin{theorem}\label{T1}
Suppose that $\chi$ is injective. Let $X$ be a connected smooth projective curve over $k$ equipped with a faithful action of $G$ such that either the quotient morphism $X\rightarrow X/G$ is an HKG cover or $p$ is odd, $m$ is even, and the quotient curve $X/H$ is rational. Then for every $R$ containing $R_s$, there exists a lift $\mathcal V$ of the $k[G]$-module $H^0(X,\Omega_X)$ to $R$ such that $\mathcal V^H=\{0\}$.
\end{theorem}

It follows that Obus Conjecture is not only not disproved but the new obstruction envisioned in \cite{KT2} always vanishes in its context. Moreover, on the positive side we add that Theorem \ref{T1} and several other results of ours could be viewed as additional evidence for the likelihood of Obus Conjecture to be true in many cases of general nature and interest, on the negative side we add that we introduce a new obstruction, and on the neutral side we add that we do not know when the new obstruction vanishes and in particular there exists the possibility (supported by Remark \ref{R6}(3)) that it vanishes in all cases in which the KGB obstruction vanishes. As such, our obstruction is not stated as a criterion but either as an abstract corollary (see Corollary \ref{C5}) or as a classification result (see Theorem \ref{C6}(3.d)).

Returning to the classification of $k[G]$-modules that lift to $R$, we first introduce the invariants that control it and its applications to curves.

{\bf On $Q(G)$.} We fix a primitive $m$-th root of unity $\overline{\zeta}_m$ in $k$ and we lift it to a primitive $m$-th root of unity $\zeta_m$ in $R_0$. For a divisor $l$ of $m$, let $\overline{\zeta}_l:=\overline{\zeta}_m^{\frac{m}{l}}\in k$ and $\zeta_l:=\zeta_m^{\frac{m}{l}}\in R_0$; they are primitive $l$-th roots of unity. 

Let $Z(G)$ be the center of $G$. Let $Q(G):=\Ker(\chi)$ be the kernel of the character $\chi:\Sigma=C_m\rightarrow\mathbb F_p^{\times}$. We have $Q(G)=Z(G)\cap\Sigma$. More precisely, if $\chi$ is non-trivial then $Q(G)=Z(G)$ and if $\chi$ is trivial then $G=C_{p^sm}$ is abelian and $Q(G)=\Sigma$.

Let $q$ be the order of $Q(G)$ and let $d:=\frac{m}{q}$; so $\sigma^d$ is a generator of $Q(G)$ and $d\mid p-1$. We consider the unique element
$$\epsilon\in qC_m$$ 
such that $\chi(\sigma)=\overline{\zeta}_m^{\epsilon}\in\mathbb F_p^{\times}$. If $q=1$, then we have $\epsilon\in (\mathbb Z/m\mathbb Z)^{\times}$.

{\bf Isomorphism classes.} The set 
$$\mathbb J_{m,p^s}:=C_m\times\llbracket1,p^s\rrbracket=\mathbb Z/m\mathbb Z\times\llbracket1,p^s\rrbracket=\Sigma\times\llbracket1,p^s\rrbracket$$
parameterizes isomorphism classes of indecomposable $k[G]$-modules. For each pair $(a,b)\in\mathbb J_{m,p^s}$ let $U_{a,b}$ be the indecomposable $k[G]$-module that corresponds to it: $U_{a,b}$ is uniserial with $\dim_k(U_{a,b})=b$, its socle 
\begin{equation}\label{EQ2}
\soc_{k[G]}(U_{a,b})=\soc_{k[H]}(U_{a,b})=\Ker(\tau-1:U_{a,b}\rightarrow U_{a,b})
\end{equation}
is a vector space over $k$ of dimension $1$, and $\sigma$ acts on $\soc_{k[G]}(U_{a,b})$ as the scalar multiplication by $\overline{\zeta}_m^a$. Thus $\soc_{k[G]}(U_{a,b})\cong U_{a,1}$. The ordered factors of the unique composition series of the $k[G]$-module $U_{a,b}$ are the first $b$ terms of the infinite periodic sequence $\bigl(U(a,i)\bigr)_{i\in\mathbb N}$ of period $d$ defined by
\begin{equation}\label{EQ3}
U(a,ld+r):=U_{a+(1-r)\epsilon}
\end{equation}
for each $(l,r)\in(\mathbb N\cup\{0\})\times\llbracket0,d-1\rrbracket$.

For all the results of the previous paragraph see \cite{BCK}, Rmk.\ 3.4 and \cite{A}. In particular, see \cite{A}, Ch.\ II, Sect.\ 5, Lem.\ 8 for Equation (\ref{EQ2}) and see \cite{A}, Ch.\ I, Sect.\ 3, p.\ 42 and 43 for the property that indecomposable $k[G]$-modules are uniserial. This uniserial property gives that the isomorphism classes of simple $k[G]$-modules are given by $U_{a,1}$ with $a\in C_m$. Based on this and the uniserial property, to check the statement on the ordered factors of the unique composition series of $U_{a,b}$ (i.e., that Equation (\ref{EQ3}) works), one can assume that $b=2$ and this case is straightforward. See \cite{A}, Ch.\ II, Sect.\ 5, Thm.\ 3 for the property that there exist precisely $m$ isomorphism classes of indecomposable projective $k[G]$-modules: they are given by $U_{a,p^s}$ with $\break a\in C_m$. For an indecomposable $k[G]$-module $U$, the $k[G]$-module $U/\rad(U)$, i.e., the quotient of $U$ by its radical, is simple by the uniserial property. Based on the last two sentences, from \cite{A}, Ch.\ II, Sect.\ 5, Lem.\ 5 one gets that each indecomposable $k[G]$-module is a quotient of an indecomposable projective $k[G]$-module and thus their isomorphism classes are as mentioned.

{\bf Invariants.} Let $T_a$ be $U_{a,1}$ but viewed only as a $k[\Sigma]$-module.

Let $V$ be a $k[G]$-module. We consider an isomorphism
$$V\rightarrow \oplus_{(a,b)\in\mathbb J_{m,p^s}} \mu_V(a,b)U_{a,b}$$
of $k[G]$-modules, where 
$$\mu_V:\mathbb J_{m,p^s}\rightarrow\mathbb N\cup\{0\}$$
is the {\it multiplicity function} of $V$ by the Krull--Schmidt Theorem (see \cite{A}, Ch.\ II, Sect.\ 4, Thm.\ 3), and an isomorphism
$$V\rightarrow \oplus_{a\in C_m} \ell_V(a)T_{a}$$
of $k[\Sigma]$-modules, where 
$$\ell_V:C_m\rightarrow\mathbb N\cup\{0\}$$
is the {\it character function} of $V$. As the group $H$ acts trivially on the semisimplification $V^{\ss}$ of $V$, the $k[G]$-module $V^{\ss}$ is actually a $k[\Sigma]$-module isomorphic to $V$. So to know the character function $\ell_V$ is the same as to know the isomorphism class of the $k[G]$-module (equivalently, $k[\Sigma]$-module) $V^{\ss}$. In particular, we have an identity
\begin{equation}\label{EQ5}
\dim_k(V)=\sum_{a\in C_m} \ell_V(a).
\end{equation}

Let
$$\eta_V:=\min(\ell_V(a)|a\in C_m).$$

The number of indecomposable factors of $V$, equivalently the dimension over $k$ of $\soc_{k[G]}(V)$, is 
$$\imath_V:=\sum_{(a,b)\in\mathbb J_{m,p^s}} \mu_V(a,b),$$
and the number of indecomposable factors of $V$ that are projective $k[G]$-modules is 
$$\jmath_V:=\sum_{a\in C_m} \mu_V(a,p^s).$$

For a lift $\mathcal V$ of $V$ to $R$, let 
$$r_{\mathcal V}\in\mathbb N\cup\{0\}$$ 
be the rank of the free $R$-module ${\mathcal V}^H$. For instance, if $\mathcal U$ is a lift of a $k[H]$-module, then from Equation (\ref{EQ1}) we get that 
\begin{equation}\label{EQ6}
r_{R[G]\otimes_{R[H]}\mathcal U}=mr_{\mathcal U}. 
\end{equation}

\begin{ex}\normalfont\label{EX1}
Suppose that $G$ is a Dihedral group; so $m=d=2$ and $p$ is odd. Let $i\in\{0,1\}$. If $b\in\llbracket1,p^s\rrbracket$ is even, then the dual of $U_{[i]_2,b}$ is $U_{[1-i]_2,b}$ and the two eigenspaces of the action of $\sigma$ on $U_{[i]_2,b}$ have the same dimension $\frac{b}{2}$ over $k$. If $b\in\llbracket1,p^s\rrbracket$ is odd, then $U_{[i]_2,b}$ is self-dual and the eigenspace of the eigenvalue $1$ (resp.\ $-1$) of the action of $\sigma$ on $U_{[i]_2,b}$ has dimension $\frac{b+1-2i}{2}$ (resp.\ $\frac{b-1+2i}{2}$) over $k$. Thus
$$\ell_V([0]_2)-\ell_V([1]_2)=\sum_{b\in \llbracket1,p^s\rrbracket\cap (2\mathbb N-1)} [\mu_V([0]_2,b)-\mu_V([1]_2,b)].$$
\end{ex}

For $S\in\{k,R,B\}$, each $S[G]$-module $M$ is a direct sum 
\begin{equation}\label{EQ8}
M=\oplus_{j\in\llbracket1,q\rrbracket} M_j
\end{equation} 
of $S[G]$-modules, where every $M_j$ is the largest $S$-submodule of $M$ on which $\sigma^d\in Q(G)$ acts as the scalar multiplication by $\overline{\zeta}_q^j=\overline{\zeta}_m^{jd}$ if $S=k$ and by $\zeta_q^j=\zeta_m^{jd}$ if $S\in\{R,B\}$. For each pair $(i,j)\in\llbracket1,m\rrbracket\times\llbracket1,q\rrbracket$ we consider the $S[G]$-module 
\begin{equation}\label{EQ9}
M_j(i)
\end{equation}
which as an $S[H]$-module is $M_j$ and on which the new action of $\sigma$ is the initial action of $\sigma$ composed with the multiplication by $\overline{\zeta}_m^{-i}$ if $S=k$ and by $\zeta_m^{-i}$ if $S\in\{R,B\}$.\footnote{The convention that involves exponent $-i$ is suggested by Equation (\ref{EQ3}) and by the standard convention on Tate twists in crystalline cohomology.} If $[i]_q=[j]_q$, then $M_j(i)$ is an $S[G]$-module fixed by $Q(G)$ and hence it is also an $S[G/Q(G)]$-module which depends on $i$ in general. It is convenient to work with the $S[G/Q(G)]$-module 
\begin{equation}\label{EQ10}
M_j':=M_j(j)
\end{equation}
which is uniquely determined by $M_j$.

We have the following basic lifting principle. 

\begin{lemma}\label{L1}
Let $M$ be a $k[G]$-module and let $M=\oplus_{j\in\llbracket1,q\rrbracket} M_j$ be as in Equation (\ref{EQ8}) applied to $S=k$. Then $M$ has a lift $\mathcal M$ to $R$ (resp.\ it has a lift $\mathcal M$ to $R$ with $r_{\mathcal M}=0$) if and only if  for each $j\in\llbracket1,q\rrbracket$, the $k[G/Q(G)]$-module $M_j'$ has a lift $\mathcal M'_j$ to $R$ (resp.\ has a lift $\mathcal M'_j$ to $R$ with $r_{\mathcal M'_j}=0$).
\end{lemma}

\begin{proof} The `only if' part holds as we can take $\mathcal M_j':=\mathcal M_j(j)$ for every $j\in\llbracket1,q\rrbracket$ (see Equation (\ref{EQ10})); note that we have a natural identification $(\mathcal M_j')^H=\mathcal M_j^H$ of $R$-modules and hence an identity $r_{\mathcal M_j'}=r_{\mathcal M_j}$. 

The `if' part holds as the $R[G]$-module $\mathcal M:=\oplus_{j\in\llbracket1,q\rrbracket} \mathcal M_j'(m-j)$ is a lift of $M$; note that we have a natural identification $\mathcal M^H=\oplus_{j\in\llbracket1,q\rrbracket} (\mathcal M_j')^H$ of $R$-modules and hence an identity $r_{\mathcal M}=\sum_{j\in\llbracket1,q\rrbracket} r_{\mathcal M_j'}$.\end{proof}

\smallskip
Based on Lemma \ref{L1}, the problem of lifting a $k[G]$-module $V$ to $R$ is the same as the problem of lifting $q$ suitable $k[G/Q(G)]$-modules to $R$, and hence by replacing $G$ by $G/Q(G)$ we can assume that $\chi$ is injective.

Our classifications of liftable $k[G]$-modules are as follows.

\begin{theorem}\label{T2}
Suppose that $\chi$ is injective. Let $V$ be a $k[G]$-module. Then the following properties hold. 

\medskip
{\bf (1)} If the $k[G]$-module $V$ has a lift $\mathcal V$ to $R$, then we have 
$$\max\bigl(\jmath_V,\sum_{a\in\mathbb Z/m\mathbb Z} [\ell_V(a)-\eta_V]\bigl)=\max\bigl(\jmath_V,\dim_k(V)-m\eta_V\bigr)\le r_{\mathcal V}\le\imath_V.$$
Moreover, if the character function $\ell_V$ is constant, then $m\mid r_{\mathcal V}$.

\smallskip
{\bf (2)} Suppose that $R$ contains $R_s$. The $k[G]$-module $V$ has a lift $\mathcal V$ to $R$ with $r_{\mathcal V}=0$ if and only if $\ell_V$ is a constant function and $\jmath_V=0$.

\smallskip
{\bf (3)} Suppose that $R$ contains $R_s$. Let 
$$V_0:=\oplus_{(a,b)\in C_m\times\llbracket1,p^s-1\rrbracket} \mu_V(a,b)U_{a,b}.$$ 
Then the following three statements are equivalent.

\medskip\noindent
{\bf (3.a)} The $k[G]$-module $V$ has a lift to $R$.

\smallskip\noindent
{\bf (3.b)} The $k[G]$-module $V_0$ has a lift to $R$.

\smallskip\noindent
{\bf (3.c)} For each $a\in C_m$ we have an inequality 
\begin{equation}\label{EQ11}
\ell_{V_0}(a)\le \eta_{V_0}+\ell_{\soc_{k[H]}(V_0)}(a).
\end{equation}
\end{theorem}

We exemplify Theorem \ref{T2}(2) and (3) when the lift does not exist and when the lift exists. First, we show that if we only assume that the character function $\ell_V$ is constant, in general $V$ has no lift to any $R$. 

\begin{ex}\normalfont\label{EX2}
Suppose that $\chi$ is injective and $m\ge 3$. We consider the direct sum $V:=U_{[0]_m,p^s}\oplus U_{-\epsilon,m-1}$. So $\ell_V$ is constant of constant value $\eta_V=1+\frac{p^s-1}{m}$, $\imath_V=2$, $V_0=U_{-\epsilon,m-1}$, $\eta_{V_0}=0$, $\ell_{\soc_{k[H]}(V_0)}(a)=0$ if $a\in C_m\setminus\{-\epsilon\}$, and $\ell_{V_0}(a)=1$ if $a\in C_m\setminus\{0\}$. As $m\ge 3$, there exists an element $c\in C_m\setminus\{0,-\epsilon\}$. As $\ell_{V_0}(c)=1>0=\eta_{V_0}+\ell_{\soc_{k[H]}(V_0)}(c)$, Inequality (\ref{EQ11}) fails for $a=c$. Thus $V$ does not lift to any $R$ by Theorem \ref{T2}(3).
\end{ex}

\begin{ex}\normalfont\label{EX3}
Suppose that $R$ contains $R_s$ and there exists a direct sum decomposition $V=\oplus_{j\in J} V_j$ into $k[G]$-modules indexed by a set $J$ such that for each $j\in J$ one of the following disjoint conditions holds. 

\medskip
{\bf (1)} We have $\jmath_{V_j}=0$ and $\ell_{V_j}$ is constant.

\smallskip
{\bf (2)} There exists $a_j\in C_m$ with $V_j\cong U_{a_j,p^s}$.

\smallskip
 {\bf (3)}  We have $\jmath_{V_j}=0$ and the image of $\ell_{V_j}$ is $\{l_j,l_j+1\}$ for some $l_j\in\mathbb N\cup\{0\}$ for which the set $\{a\in C_m|\ell_{V_j}(a)=l_j\}$ has $m-1$ elements; so $m\ge 2$.

\medskip
To check that $V$ has a lift to $R$ we can assume that $J=\{j\}$ has only one element. If condition (1) holds, then $V=V_0$ and $V$ has a lift to $R$ by Theorem \ref{T2}(2). If condition (2) holds, then $V_0=\{0\}$ and $V$ has a lift to $R$ by Theorem \ref{T2}(3). We assume now that condition (3) holds. Thus $V_0=V$ and we only use $V$, we have $\eta_V=l_j$, and there exists a unique $c\in C_m$ such that $\ell_V(c)>l_j$; so $\ell_V(c)=l_j+1$. Clearly, Inequality (\ref{EQ11}) holds for each $a\in C_m\setminus\{c\}$. If $\ell_{\soc_{k[H]}(V)}(c)=0$, then $\ell_{V}(c+\epsilon)\ge\ell_V(c)> l_j$ and hence, as $c+\epsilon$ and $c$ are distinct elements of $C_m$, we reach a contradiction to the uniqueness of $c$. So $\ell_{\soc_{k[H]}(V)}(c)\ge 1$.  Inequality (\ref{EQ11}) also holds for $a=c$ as 
$$\ell_{V}(c)=l_j+1=\eta_{V}+1\le \eta_{V}+\ell_{\soc_{k[H]}(V)}(c).$$ 
Thus $V$ has a lift to $R$ by Theorem \ref{T2}(3).  
\end{ex}

{\bf Circular $k[G]$-modules.} We say that $V$ is {\it circular} if the $\mu_V(a,b)$s depend only on $b$, i.e., if the multiplicity function $\mu_V$ factors through the second projection $\mathbb J_{m,p^s}\rightarrow \llbracket1,p^s\rrbracket$. For $b\in\llbracket1,p^s\rrbracket$, the $k[G]$-module 
$$V_b^{\circul}:=\oplus_{a\in C_m} U_{a,b}$$ 
is circular and each circular $k[G]$-module is isomorphic to a direct sum $\oplus_{b=1}^{p^s} \mu_V([0]_m,b) V_b^{\circul}$. We check that there exists a $k[G]$-linear isomorphism
\begin{equation}\label{EQ12}
V_b^{\circul}\cong k[G]\otimes_{k[H]} U_{[0]_m,b}.
\end{equation}

As $k[H]$-modules, we have $V_b^{\circul}\cong \bigl(k[H]/(\tau-1)^b)\bigr)^m$ and from Equation (\ref{EQ1}) we get that $k[G]\otimes_{k[H]} U_{[0]_m,b}\cong\bigl(k[H]/(\tau-1)^b)\bigr)^m$. From this and the paragraph of Equation (\ref{EQ2}) we get that the $k[G]$-modules $V_b^{\circul}$ and $k[G]\otimes_{k[H]} U_{[0]_m,b}$ are isomorphic if and only if the $k[\Sigma]$-modules $\soc_{k[H]}(V_b^{\circul})$ and $\soc_{k[H]}(k[G]\otimes_{k[H]} U_{[0]_m,b})$ are isomorphic. We have $k[\Sigma]$-linear isomorphisms $\soc_{k[H]}(V_b^{\circul})=\oplus_{a\in C_m} T_a$ and (see Equation (\ref{EQ1}))
$$\soc_{k[H]}(k[G]\otimes_{k[H]} U_{[0]_m,b})=(\oplus_{i=0}^{m-1} k\sigma^i)\otimes_k \soc_{k[H]}(U_{[0]_m,b}).$$
As the action of $\sigma$ on $\oplus_{i=0}^{m-1} k\sigma^i$ is a $k$-linear isomorphism defined by the multiplication by $\sigma$, its characteristic polynomial is $x^m-1=\prod_{l=1}^m (x-\overline{\zeta}_m^l)$. It follows that the two socles are isomorphic as $k[\Sigma]$-modules to $\oplus_{a\in C_m} T_a$ and hence Isomorphism (\ref{EQ12}) exists.

{\bf On prior works.} A weaker version of Theorem \ref{T2}(1) is proved in \cite{KT1}, Prop.\ 14. A version of  Theorem \ref{T2}(3) is claimed in \cite{KT1}, Thm.\ 1. A version of Theorem \ref{T2}(3) is claimed in \cite{KT2}, Thm.\ 5 which is an extrapolation of \cite{KT1}, Thm.\ 1 and Prop.\ 14 with \cite{KT2}, Crit.\ 6 as a main output. Unfortunately, these two theorems are incorrect from many points of view. For instance, they are false if $\chi$ is not injective as, based on \cite{KT1}, Prop.\ 14, one only gets congruences modulo the divisor $d$ of $m$ and not modulo $m$ (note that our usual assumption ``$\chi$ is injective'' is not made in \cite{KT1} and \cite{KT2}).

In fact, \cite{KT1}, Thm.\ 1 is rarely true (presumably it is only true if $G=H$, a case which was well-known before, e.g., see \cite{KT1}, Rmk.\ 19 or Theorem \ref{T3}(2)), and even the last sentence of \cite{KT2}, Thm.\ 5, which is crucial to \cite{KT2}, Crit.\ 6, is not true. We include an example (see Example \ref{EX5}) that disproves both for all groups $G$ when $p$ is odd, $(p,s)\neq (3,1)$, and either $m=d\ge 3$ or $d\ge 2$, based on the following example which is a particular case of Theorem \ref{T2}(2) that can be proved directly as follows.

\begin{ex}\normalfont\label{EX4}
Suppose that $d=m$ (equivalently, $\chi$ is injective or $q=1$) and $R$ contains $R_s$. For $b\in\llbracket1,p^s-1\rrbracket$, for the circular $k[G]$-module $V_b^{\circul}$ the character function $\ell_{V_b^{\circul}}$ is constant with $\eta_{V_b^{\circul}}=b$ and $\jmath_{V_b^{\circul}}=0$. Thus $V_b^{\circul}$ has a lift $\mathcal V_b^{\circul}$ to $R$ with $r_{\mathcal V_b^{\circul}}=0$ by Theorem \ref{T2}(2); despite notation, the isomorphism class of $\mathcal V_b^{\circul}$ does not depend in general only on $b$ even if $R$ is fixed. Based on Isomorphism (\ref{EQ12}), we can take $\mathcal V_b^{\circul}$ such that
$$\mathcal V_b^{\circul}\cong R[G]\otimes_{R[H]} \mathcal U_{[0]_m,b},$$
where $\mathcal U_{[0]_m,b}$ is a lift of the $k[H]$-module $U_{[0]_m,b}$ to $R$ with $r_{\mathcal U_{[0]_m,b}}=0$ (e.g., see \cite{KT1}, Rmk.\ 19 or Theorem \ref{T3}(2)). Thus $r_{\mathcal V_b^{\circul}}=0$ by Equation (\ref{EQ6}).
\end{ex}

\begin{ex}\normalfont\label{EX5}
Suppose that $p$ is odd, $(p,s)\neq (3,1)$, and either $m=q\ge 3$ or $d\ge 2$. We show that the assumption that \cite{KT1}, Thm.\ 1 is true for the group $G$ leads to a contradiction. We refer to Conditions a. to c. of \cite{KT1}, Thm.\ 1. If $m=q\ge 3$, so $G=\Sigma\times H$, then $U_{[0]_m,2}$ lifts to an $R_s[G]$-module on which $\Sigma$ acts trivially (e.g., see \cite{KT1}, Rmk.\ 19) and, as $m\ge 3$, $2\not\equiv 0\;\textup{or}\;1\pmod{m}$, for the only partition of the set $\{([0]_m,2)\}$, Condition b. does not hold, a contradiction. If $d\ge 2$, then, to ease the notation, by replacing $G$ with $G/Q(G)$ we can assume that $q=1$ and $d=m$; so $m$ divides $p-1$. Let $i\in\llbracket1,\frac{p^s-3}{2}\rrbracket$. The $k[G]$-modules $V_{p^s-i-1}^{\circul}$ and $V_{p^s-i}^{\circul}$ have lifts $\mathcal V_{p^s-i-1}^{\circul}$ and $\mathcal V_{p^s-i}^{\circul}$ (respectively) with $r_{\mathcal V_{p^s-i-1}^{\circul}}=0$ and $r_{\mathcal V_{p^s-i}^{\circul}}=0$ by Example \ref{EX4}. The $R_s[G]$-module $\mathcal V:=\mathcal V_{p^s-i-1}^{\circul}\oplus\mathcal V_{p^s-i}^{\circul}$ lifts 
$$V:=V_{p^s-i-1}^{\circul}\oplus V_{p^s-i}^{\circul}=\oplus_{(a,b)\in J_0} U_{a,b}$$
and we have $r_{\mathcal V}=0$, where $ J_0:=C_m\times\{p^s-i-1,p^s-i\}$. Our assumption implies that the set $ J_0$ admits a partition in such a way that Conditions a. to c. hold. As $2(p^s-i-1)>p^s$, Condition a. forces the partition of $ J_0$ to be in $2m$ sets of cardinality $1$. As $p^s-2$ is odd and $p^s-1$ is even, Condition b. does not hold, a contradiction. 

If either $m=d=2$ and $s=1$ or $s\ge 2$, by taking either $i=1$ or such that $m$ divides $i+2$ and by considering the subset $J_1:=\mathbb Z/m\mathbb Z\times\{p^s-i-1\}$ of $ J_0$, we have $p^s-i-1\equiv 1 \pmod{m}$ and thus the last sentence of \cite{KT2}, Thm.\ 5, which is not part of \cite{KT1}, Thm.\ 1, is also not true.\end{ex}

We disproved \cite{KT1}, Thm.\ 1 and \cite{KT2}, Thm.\ 5 as `if and only if' statements by disproving their `only if' parts. However, their `if' parts are true by Theorem \ref{T2}(2) and (3) despite lengthy computational proofs that are hard to follow and with many inexactitudes. 

Briefly, the paper is structured as follows. Sections \ref{S2} to \ref{S5} contain results on various classes of modules that are required to prove Theorem \ref{T2} in Sections \ref{S6} to \ref{S8}. Section \ref{S9} contains complements on the upper and lower jumps of a cyclic extension of $k((z))$ of Galois group $C_{p^s}$. Sections \ref{S10} to \ref{S13} contain applications to curves which are additionally exemplified in Section \ref{S14}. In particular, if $\chi$ is injective and $m$ is even (so $p$ is odd), then Theorem \ref{C6}(3.d) classifies the $B[H]$-modules $H^0(\mathcal X,\Omega_{\mathcal X})\otimes_R B$ for homomorphisms $\varrho_{\phi}:G\rightarrow\Aut(\mathcal X)$ that lift to $R$ faithful global actions $\rho_{\phi}:G\rightarrow\Aut(X)$ for which the quotient morphisms $X\rightarrow X/G$ are HKG covers associated to faithful local actions $\phi:G\rightarrow\Aut_k(k[[y]])$. 
 
\section{Preliminaries on modules}\label{S2}

We review notation and basic properties of various modules.

Let $S$ be a commutative ring with identity element. Recall that $S^{\times}$ is the multiplicative group of units of $S$. 

For two $S$-modules $N$ and $N'$, let $\Hom_{S}(N,N')$ be the $S$-module of $S$-linear maps from $N$ to $N'$, let $\End_{S}(N):=\Hom_{S}(N,N)$, let $\Isom_{S}(N,N')$ be the set of $S$-linear isomorphisms from $N$ to $N'$, let $\Aut_{S}(N)$ be the group of $S$-linear automorphisms of $N$ (i.e., of units of the $S$-algebra $\End_{S}(N)$), and let $1_N\in\Aut_{S}(N)$ be the identity automorphism of $N$ (the identity element of $\Aut_{S}(N)$). For $f\in\Hom_{S}(N,N')$, let $\Ker(f)$ be its kernel. 

Recall that $N$ is called cyclic if it is generated by one element, i.e., if there exists an ideal $I$ of $S$ such that $N\cong S/I$. 

For two $S[G]$-modules $M_1$ and $M_2$, let $\Hom_{S[G]}(M_1,M_2)$ be the $S$-module of $S[G]$-linear maps from $M_1$ to $M_2$. For each $S[H]$-module $M$ we have an identity of $S$-modules
$$\Hom_{S[G]}(S[G]\otimes_{S[H]} M,M_1)=\Hom_{S[H]}(M,M_1).$$

The dual $M_1^{\vee}$ of an $S[G]$-module $M_1$ is the $S$-module $\Hom_S(M_1,S)$ with the action given by the rule $(g\cdot f)(v):=f(g^{-1}v)$ for $(g,f,v)\in G\times M_1^{\vee}\times M_1$. If $M_1$ is a free $S$-module, then the $S$-module $M_1^{\vee}$ is free and isomorphic to $M_1$.

For a set $A$, let $\mathcal P(A)$ be the power set of $A$. If $t\in\mathbb N$, let $\mathcal P_t(A)$ be the subset of $\mathcal P(A)$ formed by subsets of $A$ of cardinality $t$.

Let $\Xi_s$ be the set of $p^s$-roots of unity in each $R$ (resp.\ $B$) that contains $R_s$ (resp.\ $B_s$) and let $\Xi_s^{\ast}:=\Xi_s\setminus\{1\}$. The action $(\mathbb Z/p^s\mathbb Z)^{\times}\times\mathcal P(\Xi_s)\rightarrow\mathcal P(\Xi_s)$ that maps a pair $(u,\Lambda)$ with $u\in (\mathbb Z/p^s\mathbb Z)^{\times}$ and $\Lambda\subset\Xi_s$ to the subset 
$$u\cdot\Lambda:=\{\lambda^u|\lambda\in\Lambda\}$$ 
restricts to an action of $(\mathbb Z/p^s\mathbb Z)^{\times}$ on each $\mathcal P_t(\Xi_s)$.

Let $S\in\{k,R,B\}$ and $x$ an indeterminate. We identify 
$$S[H]=S[x]/(x^{p^s}-1)$$ 
in such a way that $\tau\in H\subset S[H]$ is identified with $x+(x^{p^s}-1)$. For $S=k$ we get the identification $k[H]=k[x]/(x-1)^{p^s}$. 

For $t\in\mathbb Z\setminus p\mathbb Z$ let
$$\mathfrak a_{[t]_{p^s}}:S[H]\rightarrow S[H]$$ 
be the $k$-algebra automorphism defined by the automorphism $H\rightarrow H$ that maps $\tau$ to $\tau^t$, i.e., defined by the rule $x+(x^{p^s}-1)\mapsto x^t+(x^{p^s}-1)$. Its inverse is $\mathfrak a_{[t']_{p^s}}$, where $t'\in\mathbb Z\setminus p\mathbb Z$ is such that $p^s\mid tt'-1$. 

For every $S[H]$-module $N$ defined by a $\mathbb Z$-algebra homomorphism $\break\mathfrak n:S[H]\rightarrow\End_{\mathbb Z}(N)$, let $N^{(t)}$ be the $S[H]$-module defined by the $\mathbb Z$-algebra homomorphism $\mathfrak{n}\circ \mathfrak{a}_{[t]_{p^s}}:S[H]\rightarrow\End_{\mathbb Z}(N)$. In simpler words, $N^{(t)}$ is the $S$-module $N$ on which $\tau$ (or x) acts as $\tau^t$; we have
\begin{equation}\label{EQ15}
N^{(t)}\cong S[H]\otimes_{ \mathfrak a_{[t']_{p^s}},S[H]} N.
\end{equation}
From Isomorphism (\ref{EQ15}) we get an identity $\bigl(N^{(t)}\bigr)^{(t')}=N$ of $S[H]$-modules.

With the notation of the prior paragraph, from Equation (\ref{EQ1}) we get that we have an $S[H]$-linear isomorphism
$$S[G]\otimes_{S[H]} N\cong\oplus_{i=0}^{m-1} N^{(\epsilon^i)}.$$

As $k[x]$ is a principal ideal domain, each $k[H]$-module is self-dual and a direct sum of cyclic $k[H]$-modules and each cyclic $k[H]$-module is isomorphic to $k[x]/(x-1)^b$ for some $b\in\llbracket0,p^s\rrbracket$.

For $(b,c)\in\llbracket0,p^s\rrbracket^2$ with $b\le c$, we have a $k[H]$-linear identification
\begin{equation}\label{EQ17}
k[x]/(x-1)^b=\Hom_{k[H]}\bigl(k[x]/(x-1)^c,k[x]/(x-1)^b\bigr)
\end{equation}
with $\alpha\in k[x]/(x-1)^b$ identified to $f_{\alpha}\in \Hom_{k[H]}(k[x]/(x-1)^c,k[x]/(x-1)^b)$ that satisfies $f_{\alpha}\bigl(1+(x-1)^c\bigr)=\alpha$. Thus, by duality we also have a $k[H]$-linear identification
$k[x]/(x-1)^b=\Hom_{k[H]}(k[x]/(x-1)^b,k[x]/(x-1)^c)$. For $b=c$ we get a $k[H]$-algebra identification
\begin{equation}\label{EQ18}
k[x]/(x-1)^b=\End_{k[H]}\bigl(k[x]/(x-1)^b\bigr).
\end{equation}

\section{On $B_0[H]$-modules}\label{S3}

The classifications of $B_0[H]$-modules and of $B[H]$-modules with $B$ containing $B_s$ are recalled in order to introduce notation that helps in stating the results in all the subsequent sections. For the sake of generality, we work over an arbitrary field of characteristic $0$.

\begin{definition}\label{D2}
Let $\mathcal B$ be a field of characteristic $0$. Let $\mathcal C$ be a $\mathcal B[G]$-module.

\medskip
{\bf (1)} We say that $\mathcal C$ is rational if it is definable over $\mathbb Q$, i.e., there exists a $\mathbb Q[G]$-module $\mathcal C_{\mathbb Q}$ whose extension $\mathcal C_{\mathbb Q}\otimes_{\mathbb Q} \mathcal B$ to $\mathcal B$ is isomorphic to $\mathcal C$.

\smallskip
{\bf (2)} Suppose that $\mathcal B$ contains a primitive $m$-th root of unity $\zeta_m$ (e.g., $\mathcal B$ contains $B_0$). We say that $\mathcal C$ is almost rational if it is definable over $\mathbb Q(\zeta_m)$, i.e., there exists a $\mathbb Q(\zeta_m)[G]$-module $\mathcal C_{\mathbb Q(\zeta_m)}$ whose extension $\mathcal C_{\mathbb Q(\zeta_m)}\otimes_{\mathbb Q(\zeta_m)} \mathcal B$ to $\mathcal B$ is isomorphic to $\mathcal C$.

\smallskip
{\bf (3)} A lift $\mathcal V$ of a $k[G]$-module $V$ to $R$ is called rational (resp.\ almost rational) if the $B[G]$-module $\mathcal V\otimes_R B$ is rational (resp.\ almost rational).
\end{definition}

As for $m=1$ we have $G=H$ and the notions of rational and almost rational coincide, in what follows we also apply Definition \ref{D2}(1) to $G$ replaced by $H$.

The isomorphism classes of simple $\mathbb Q[H]$-modules are indexed by irreducible factors of the monic separable polynomial $x^{p^s}-1\in \mathbb Q[x]$. For $i\in\llbracket1,s\rrbracket$, the cyclotomic polynomial $\Psi_i(x):=\frac{x^{p^i}-1}{x^{p^{i-1}}-1}\in \mathbb Q[x]$ is irreducible of degree $p^{i-1}(p-1)$. Let $\Psi_0(x):=x-1\in \mathbb Q[x]$. As we have a product decomposition $x^{p^s}-1=\prod_{i=0}^s \Psi_i(x)$, we get that the isomorphism classes of simple $\mathbb Q$-modules are indexed by $i\in\llbracket0,s\rrbracket$ and given by
$$\mathcal B_{i,\mathbb Q}:=\mathbb Q[x]/\bigl(\Psi_i(x)\bigr).$$ 
So for each $\mathbb Q[H]$-module $\mathcal C_{\mathbb Q}$ there exists a unique function
$$\varkappa_{\mathcal C_{\mathbb Q}}:\llbracket0,s\rrbracket\rightarrow\mathbb N\cup\{0\}$$
such that 
$$\mathcal C_{\mathbb Q}\cong\oplus_{i=0}^s \varkappa_{\mathcal C_{\mathbb Q}}(i)\mathcal B_{i,_{\mathbb Q}}$$
and the monic characteristic polynomial of the $\mathbb Q$-linear map $\tau:\mathcal C_{\mathbb Q}\rightarrow \mathcal C_{\mathbb Q}$ is 
\begin{equation}\label{EQ19.5}
\prod_{i=0}^s \Psi_i(x)^{\varkappa_{\mathcal C_{\mathbb Q}}(i)}\in\mathbb Q[x].
\end{equation}

The above paragraph holds with $\mathbb Q$ replaced by $B_0$. So the $B_0[H]$-module 
$$\mathcal B_i:=\mathcal B_{i,_{\mathbb Q}}\otimes_{\mathbb Q} B_0=B_0[x]/\bigl(\Psi_i(x)\bigr)$$
is irreducible and for each $B_0[H]$-module $\mathcal C$ there exists a unique function
$$\varkappa_{\mathcal C}:\llbracket0,s\rrbracket\rightarrow\mathbb N\cup\{0\}$$
such that 
$$\mathcal C\cong\oplus_{i=0}^s \varkappa_{\mathcal C}(i)\mathcal B_i.\footnote{As in the case of $k[H]$ and $k[G]$-modules, the isomorphism classes of simple $B_0[G]$-modules are indexed by elements of the set $\mathbb J_{m,p^s}=C_m\times\llbracket0,p^s\rrbracket$.}$$
This implies that the following fact holds.

\begin{fact}\label{F1}
Each $B_0[H]$-module is rational.
\end{fact}

Moreover, for $i\in\llbracket0,s\rrbracket$ we have an isomorphism of $B_0$-algebras
$$\End_{B_0[H]}(\mathcal B_i)\cong B_i$$
and an identity
\begin{equation}\label{EQ20}
\dim_{B_0}(\mathcal B_i)=\begin{cases}1\quad\quad\quad\quad\quad\quad {\rm if}\quad i=0\\
p^{i-1}(p-1)\quad\;\;\, {\rm if}\quad i\ge 1.\end{cases}
\end{equation}

For $(i,j)\in\llbracket0,s\rrbracket^2$ we have 
\begin{equation}\label{EQ21}
\mathcal B_i^{\tau^{p^j}}=\begin{cases}0\quad\quad {\rm if}\quad j<i\\
\mathcal B_i\quad\;\; {\rm if}\quad j\ge i.\end{cases}
\end{equation}
Therefore for $i\in\llbracket0,s\rrbracket$ we have
\begin{equation}\label{EQ22}
\varkappa_{\mathcal C}(i)=\begin{cases} \dim_{B_0}(\mathcal C^{\tau})\quad\quad\quad\quad\quad\quad\quad\quad  {\rm if}\quad i=0\\
\frac{\dim_{B_0}(\mathcal C^{\tau^{p^i}})-\dim_{B_0}(\mathcal C^{\tau^{p^{i-1}}})}{p^i-p^{i-1}}\quad\quad\, {\rm if}\quad i\ge 1.\end{cases}
\end{equation}

Let $\zeta\in\Xi_s\cap\mathcal B$. For a $\mathcal B[H]$-module $\mathcal C$, let $\mathfrak E_{\zeta}(\mathcal C)$ be the eigenspace of the action of $\tau$ on $\mathcal C$ that  corresponds to eigenvalue $\zeta$ and let 
$$\kappa_{\mathcal C}(\zeta):=\dim_{\mathcal B}\bigl(\mathfrak E_{\zeta}(\mathcal C)\bigr)\in\mathbb N\cup\{0\}.$$

\begin{lemma}\label{L2} Let $\mathcal B$ be a field of characteristic $0$.  Let $\mathcal C$ and $\mathcal E$ be $\mathcal B[H]$-modules. For a field extension $\mathcal B\rightarrow\mathcal B'$ the following properties hold.

\medskip
{\bf (1)} The $\mathcal B[H]$-modules $\mathcal C$ and $\mathcal E$ are isomorphic if and only if the $\mathcal B'[H]$-modules $\mathcal C\otimes_\mathcal B \mathcal B'$ and $\mathcal E\otimes_\mathcal B \mathcal B'$ are isomorphic.

\smallskip
{\bf (2)} The $\mathcal B[H]$-module $\mathcal C$ is rational if and only if the $\mathcal B'[H]$-module $\mathcal C\otimes_\mathcal B \mathcal B'$ is rational.
\end{lemma}

\begin{proof}
The `only if' of either part (1) or part (2) is clear.

The `if' of part (1) is well-known and we recall one of its many standard proofs. As each $\mathcal B[H]$-module is semisimple, it suffices to show that if $\Psi$ is a monic irreducible factor of $x^{p^s}-1\in \mathcal B[x]$, then for the simple $\mathcal B[H]$-module $\mathcal D:=\mathcal B[x]/(\Psi)$, by denoting $\mathcal K:=\End_{\mathcal B[H]}(\mathcal D)$ which is $\mathcal D$ but viewed as a field, we have 
\begin{equation}\label{EQ23}
\dim_{\mathcal K}\bigl(\End_{\mathcal B[H]}(\mathcal D,\mathcal C)\bigr)=\dim_{\mathcal K}\bigl(\End_{\mathcal B[H]}(\mathcal D,\mathcal E)\bigr).
\end{equation}

We first assume that the field extension $\mathcal B\rightarrow\mathcal B'$ is finite. As the $\mathcal B'[H]$-modules $\mathcal C\otimes_\mathcal B \mathcal B'$ and $\mathcal E\otimes_\mathcal B \mathcal B'$ are isomorphic, we have an identity of dimensions $\dim_{\mathcal K}\bigl(\End_{\mathcal B'[H]}(\mathcal D\otimes_\mathcal B \mathcal B',\mathcal C\otimes_\mathcal B \mathcal B')\bigr)=\dim_{\mathcal K}\bigl(\End_{\mathcal B'[H]}(\mathcal D\otimes_\mathcal B \mathcal B',\mathcal E\otimes_\mathcal B \mathcal B')\bigr)$. Thus
$\dim_{\mathcal K}\bigl(\End_{\mathcal B[H]}(\mathcal D,\mathcal C\otimes_\mathcal B \mathcal B')\bigr)=\dim_{\mathcal K}\bigl(\End_{\mathcal B[H]}(\mathcal D,\mathcal E\otimes_\mathcal B \mathcal B')\bigr)$. So Equation (\ref{EQ23}) holds as it holds after multiplication by $[\mathcal B':\mathcal B]$. 

Next we assume that the field extension $\mathcal B\rightarrow\mathcal B'$ is infinite. By replacing both $\mathcal B$ and $\mathcal B'$ by finite field extensions of them based on the previous paragraph, we can assume that the characteristic polynomials of the $\mathcal B$-linear maps $\tau:\mathcal C\rightarrow\mathcal C$ and $\tau:\mathcal E\rightarrow\mathcal E$ are split; so $\Psi(x)=x-\zeta$ has degree $1$. The multiplicity of the eigenvalue $\zeta$ for the action of $\tau$ on $\mathcal C$ (resp.\ $\mathcal E$) equals to the multiplicity of the eigenvalue $\zeta$ for the action of $\tau$ on $\mathcal C\otimes_{\mathcal B} \mathcal B'$ (resp.\ $\mathcal C\otimes_{\mathcal B} \mathcal B'$). From this and the fact that the $\mathcal B'[H]$-modules $\mathcal C\otimes_\mathcal B \mathcal B'$ and $\mathcal E\otimes_\mathcal B \mathcal B'$ are isomorphic we get that Equation (\ref{EQ23}) holds. So part (1) holds. 

For the `if' of part (2), let $\mathcal C_{\mathbb Q}$ be a $\mathbb Q[H]$-module such that the $\mathcal B'[H]$-modules $\mathcal C\otimes_\mathcal B \mathcal B'$ and $\mathcal C_{\mathbb Q}\otimes_{\mathbb Q} \mathcal B'$ are isomorphic. As the $\mathcal B[H]$-modules $\mathcal C$ and $\mathcal C_{\mathbb Q}\otimes_{\mathbb Q} \mathcal B$ are isomorphic by part (1), the `if' of part (2) holds.
\end{proof}

\begin{lemma}\label{L3}
Let $\mathcal B$ be a field of characteristic $0$. Let $\mathcal C$ be a $\mathcal B[H]$-module. Then the following statements are equivalent.

\medskip
{\bf (1)} The  $\mathcal B[H]$-module $\mathcal C$ is rational.

\smallskip
{\bf (2)} The monic characteristic polynomial of the $\mathcal B$-linear map $\tau:\mathcal C\rightarrow \mathcal C$ belongs to $\mathbb Q[x]$.

\smallskip
{\bf (3)} For each $i\in\llbracket1,p^s\rrbracket$, the trace of  the $\mathcal B$-linear map $\tau^i:\mathcal C\rightarrow \mathcal C$ is a rational number.

\smallskip
{\bf (4)} For each $i\in\llbracket1,p^s-1\rrbracket$, the trace of  the $\mathcal B$-linear map $\tau^i:\mathcal C\rightarrow \mathcal C$ is a rational number.

\medskip
If $B_0$ is a subfield of $\mathcal B$, then the above four statements are also equivalent to the following one.

\medskip
{\bf (5)} The $\mathcal B[H]$-module $\mathcal C$ is defined over $B_0$.
\end{lemma}

\begin{proof}
We have $(1)\Rightarrow (2)$ by Equation (\ref{EQ19.5}). 

From Newton's identities we get that $(2)\Leftrightarrow (3)$. 

As the trace of the identity $\mathcal B$-linear map $\tau^{p^s}:\mathcal C\rightarrow \mathcal C$ is $\dim_{\mathcal B}(\mathcal C)\in\mathbb N\cup\{0\}$, we have $(3)\Leftrightarrow (4)$. 

To prove that $(2)\Rightarrow (1)$, let $\char_{\tau}(x)\in\mathbb Q[x]$ be the monic characteristic polynomial of the $\mathcal B$-linear map $\tau:\mathcal C\rightarrow \mathcal C$. As the set of zeros in $\mathbb C$ of $\char_{\tau}(x)$ is contained in the set of zeros of $x^{p^s-1}=\prod_{i=0}^s \Psi_i(x)$ and each $\Psi_i(x)$ with $i\in\llbracket0,s\rrbracket$ is irreducible, it follows that there exists a function $\varkappa_{\mathcal C}:\varkappa_{\mathcal C}:\llbracket0,s\rrbracket\rightarrow\mathbb N\cup\{0\}$ such that $\char_{\tau}(x)=\prod_{i=0}^s \Psi_i(x)^{\varkappa_{\mathcal C}(i)}$. We consider the $\mathbb Q[x]$-module $\mathcal C_{\mathbb Q}:=\oplus_{i=0}^s \varkappa_{\mathcal C}(i)\mathcal B_{i,\mathbb Q}$. From the identity $\char_{\tau}(x)=\prod_{i=0}^s \Psi_i(x)^{\varkappa_{\mathcal C}(i)}$ we get that, for an algebraic closure $\overline{\mathcal B}$ of $\mathcal B$, the $\overline{\mathcal B}$-modules $\mathcal C_{\mathbb Q}\otimes_{\mathbb Q} \overline{\mathcal B}$ and $\mathcal C\otimes_{\mathcal B} \overline{\mathcal B}$ are isomorphic. From this and Lemma \ref{L2}(2) we get that that part (1) holds. Thus $(2)\Rightarrow (1)$.

Hence the statements (1) to (4) are equivalent.

Assume now that $B_0$ is a subfield of $\mathcal B$. Clearly, $(1)\Rightarrow (5)$. The fact that $(1)\Rightarrow (5)$ follows from Fact \ref{F1}. So statements (1) to (5) are equivalent.
\end{proof}

\begin{lemma}\label{L4}
Suppose that $\chi$ is injective and $m$ is even. Let $\mathcal B$ be a field of characteristic $0$. We consider a short exact sequence of $\mathcal B[G]$-modules
\begin{equation}\label{EQ24}
0\rightarrow\mathcal C_1\rightarrow\mathcal C\rightarrow\mathcal C_2\rightarrow 0
\end{equation}
 such that $\mathcal C_1$ and $\mathcal C_2$ are dual to each other. If the $\mathcal B[H]$-module $\mathcal C$ is rational, then the $\mathcal B[H]$-modules $\mathcal C_1$ and $\mathcal C_2$ are isomorphic and rational.
\end{lemma}

\begin{proof}
It suffices to show that $\mathcal C_1$ is rational. Let $\mathcal C_{\mathbb Q}$ be a $\mathbb Q[H]$-module such that we have an isomorphism $\mathcal C_{\mathbb Q}\otimes_{\mathbb Q} \mathcal B\cong \mathcal C$.

Based on Lemma \ref{L2}(2) we can assume that $B_s\subset \mathcal B$. Let $\zeta\in\Xi_s$ be arbitrary. As Short Exact Sequence (\ref{EQ24}) splits, we have
\begin{equation}\label{EQ25}
\kappa_{\mathcal C}(\zeta)=\kappa_{\mathcal C_1}(\zeta)+ \kappa_{\mathcal C_2}(\zeta).
\end{equation}
As $m$ is even and $\chi$ is injective, for $\sigma_1:=\sigma^{\frac{m}{2}}$ we have $\sigma_1\tau\sigma_1^{-1}=\tau^{-1}$. This implies that $\sigma_1\bigl(\mathfrak E_{\zeta}(\mathcal C_1)\bigr)=\mathfrak E_{\zeta^{-1}}(\mathcal C_1)$. Thus 
\begin{equation}\label{EQ26}
\kappa_{\mathcal C_1}(\zeta)=\kappa_{\mathcal C_1}(\zeta^{-1}).
\end{equation}
As $\mathcal C_1$ and $\mathcal C_2$ are dual $B[G]$-modules, it follows that
\begin{equation}\label{EQ27}
\kappa_{\mathcal C_2}(\zeta)=\kappa_{\mathcal C_1}(\zeta^{-1}).
\end{equation}
From Equations (\ref{EQ25}) to (\ref{EQ27}) we get first that 
\begin{equation}\label{EQ28}
\kappa_{\mathcal C_1}(\zeta)=\kappa_{\mathcal C_2}(\zeta)
\end{equation}
and second that
\begin{equation}\label{EQ29}
\kappa_{\mathcal C}(\zeta)=2\kappa_{\mathcal C_1}(\zeta).
\end{equation}

From Equation (\ref{EQ29}) we get that $\varkappa_{\mathcal C_{\mathbb Q}}(i)$ is even for each $i\in\llbracket0,s\rrbracket$. Thus we can consider the $\mathbb Q[H]$-module 
$$\mathcal C_{1,\mathbb Q}:=\oplus_{i=0}^s \frac{\varkappa_{\mathcal C_{\mathbb Q}}(i)}{2}\mathcal B_{i,\mathbb Q}.$$
Clearly, $\kappa_{\mathcal C_{1,\mathbb Q}\otimes_{\mathbb Q} \mathcal B}(\zeta)=\frac{\kappa_{\mathcal C}(\zeta)}{2}$. From this and Equations (\ref{EQ28}) and (\ref{EQ29}) we get that $\kappa_{\mathcal C_{1,\mathbb Q}\otimes_{\mathbb Q} \mathcal B}(\zeta)=\kappa_{\mathcal C_1}(\zeta)$. As $\zeta\in\Xi_s$ is arbitrary, the $\mathcal B[H]$-modules $\mathcal C_{1,\mathbb Q}\otimes_{\mathbb Q} \mathcal B=\bigl(\oplus_{i=0}^s \frac{\varkappa_{\mathcal C_{\mathbb Q}}(i)}{2}\mathcal B_i\bigr)\otimes_{B_0} B$, $\mathcal C_1$, and $\mathcal C_2$ are isomorphic and the lemma holds.
\end{proof}

\begin{remark}\normalfont\label{R1}
If in Lemma \ref{L4} we have $\mathcal B=B$ and the $B[G]$-module $\mathcal C$ is almost rational, then in general the $B[G]$-module $\mathcal C_1$ (or $\mathcal C_2$) is not rational. For instance, suppose that $B$ and $G$ are such that there exists $i\in\llbracket1,s\rrbracket$ and two non-isomorphic dual to each other simple $B_0[G]$-module structures $\mathcal B_{i,1}$ and $\mathcal B_{i,2}$ that extend the self-dual $B_0[H]$-module structure of $\mathcal B_i$ with the property that we have a direct sum decomposition $\mathcal B_{i,1}\otimes_{B_0} B=\mathcal D_1\oplus\mathcal D_2$ into non-zero $B[G]$-modules. So we also have a direct sum decomposition $\mathcal B_{i,2}\otimes_{B_0} B=\mathcal D_3\oplus\mathcal D_4$ into non-zero $B[G]$-modules such that $\mathcal D_3\cong \mathcal D_1^{\vee}$ and $\mathcal D_4\cong\mathcal D_2^{\vee}$. For a pair $(s,t)\in\mathbb N^2$ with $s\neq t$, we take the $B[G]$-modules $\mathcal C_1:=s\mathcal D_1\oplus t\mathcal D_2\oplus t\mathcal D_3\oplus s\mathcal D_4$, $\mathcal C_2:=t\mathcal D_1\oplus s\mathcal D_2\oplus s\mathcal D_3\oplus t\mathcal D_4$, and $\mathcal C:=\mathcal C_1\oplus\mathcal C_2$. So $\mathcal C_1$ (or $\mathcal C_2$) is not an almost rational $B[G]$-module but 
$$\mathcal C\cong [(s+t)\mathcal B_{i,1}\oplus (s+t)\mathcal B_{i,2}]\otimes_{B_0} B$$ 
is a self-dual rational $B[G]$-module.
\end{remark}

\begin{lemma}\label{L5}
Let $V$ be a $k[H]$-module such that it has a rational lift $\mathcal V$ to $R$ with $r_{\mathcal V}=0$. Then $V$ has an increasing, separated, and exhaustive filtration
\begin{equation}\label{EQ30}
\{0\}=V_0\subset V_1\subset V_2\subset\cdots\subset V_{s-1}\subset V_s=V
\end{equation}
by $k[V]$-submodules such that for each $i\in\llbracket1,s\rrbracket$, for the quotient $k[H]$-module $W_i:=V_i/V_{i-1}$ the following properties hold.

\medskip
{\bf (1)} The dimension $\dim_k(W_i)$ is divisible by $p^i-p^{i-1}$.

\smallskip
{\bf (2)} We have $W_i=W_i^{\tau^{p^i}}$.
\end{lemma}

\begin{proof}
Let $\mathcal C_0$ be a $B_0[H]$-module such that there exists a bijective $B[H]$-linear map $f:\mathcal C_0\otimes_{B_0} B\rightarrow \mathcal V\otimes_R B$. As $r_{\mathcal V}=0$, we have $\varkappa_{\mathcal C_0}(0)=0$. For each $i\in\llbracket1,s\rrbracket$, let $\mathcal C_i$ be the unique $B_0[H]$-submodule of $\mathcal C_0$ such that $\mathcal C_i\cong\oplus_{j=1}^i \varkappa_{\mathcal C_0}(j)\mathcal B_j$ and let $\mathcal V_i:=\mathcal V\cap f(\mathcal C_i)$, the intersection being taken inside $\mathcal V[\frac{1}{p}]=\mathcal V\otimes_R B$. Clearly, $\mathcal V_i$ is an $R[H]$-submodule of $\mathcal V$ such that $\mathcal V/\mathcal V_i$ is a free $R$-module. Defining $\mathcal V_0:=\{0\}\subset\mathcal V$, we have an increasing, separated, and exhaustive filtration
$$\{0\}=\mathcal V_0\subset\mathcal V_1\subset\mathcal V_2\subset\cdots\subset\mathcal V_{s-1}\subset\mathcal V_s=\mathcal V$$
and we take its reduction modulo $\upsilon$ to be the Filtration (\ref{EQ30}). Hence for each $i\in\llbracket1,s\rrbracket$, $W_i$ is the reduction modulo $\upsilon$ of the quotient $R[H]$-module $\mathcal W_i:=\mathcal V_i/\mathcal V_{i-1}$. 

We have isomorphisms of $B[H]$-modules
$$\mathcal W_i\otimes_R B\cong (\mathcal C_i\otimes_{B_0} B)/(\mathcal C_{i-1}\otimes_{B_0} B)\cong [\varkappa_{\mathcal C_0}(i)\mathcal B_i]\otimes_{B_0} B.$$
 From this, the identity $\dim_{B_0}\bigl(\varkappa_{\mathcal C_0}(i)\mathcal B_i\bigr)=\varkappa_{\mathcal C_0}(i)(p^i-p^{i-1})$, and Equation (\ref{EQ21}), we get that $p^i-p^{i-1}$ divides the rank of the free $R$-module $\mathcal W_i$ and $\mathcal W_i=\mathcal W_i^{\tau^{p^i}}$. As $W_i=\mathcal W_i/\upsilon\mathcal W_i$, we get that parts (1) and (2) hold.\end{proof}

\section{On  lifts of $k[H]$-modules}\label{S4}

In this section we study various classes of lifts of a $k[H]$-module $V$ to $R$. Taking $(G,m)=(H,1)$, the isomorphism classes of $H$-modules are parameterized by $\mathbb J_{p^s}:=\llbracket1,p^s\rrbracket$, with the isomorphism class that corresponds to $b\in\mathbb J_{p^s}$ defined by 
$$U(b):=U([0]_1,b)\cong k[x]/(x-1)^b.$$ 
For a $k[H]$-module $V$, denoting $\mu_V(b):=\mu_V([0]_1,b)$ we get the multiplicity function 
$$\mu_V:\mathbb J_{p^s}\rightarrow\mathbb N$$ 
of $V$: we have a $k[H]$-linear isomorphism
$$V\cong\oplus_{b\in\mathbb J_{p^s}} \mu_V(b)U(b).$$

\begin{definition}\label{D3}
Let $V$ be a $k[H]$-module and let $\mathcal V$ be a lift of $V$ to $R$. 

\medskip
{\bf (1)} We call $\mathcal V$ an end-hull of $V$ if the $R$-algebra homomorphism
$$\End_{R[H]}(\mathcal V)\rightarrow\End_{k[H]}(V)$$ 
defined by the reduction modulo $\upsilon$ is surjective.

\smallskip
{\bf (2)} Suppose that $R$ contains $R_s$. For $t\in\mathbb Z\setminus p\mathbb Z$ we say that $\mathcal V$ is potentially $t$-invariant if the $B[H]$-modules $\mathcal V\otimes_R B$ and $\mathcal V^{(t)}\otimes_R B$ are isomorphic and we say that $\mathcal V$ is $t$-invariant if the $R[H]$-modules $\mathcal V$ and $\mathcal V^{(t)}$ are isomorphic.
\end{definition}

\begin{remark}\normalfont\label{R2}
If $\mathcal V$ is a lift of $V$ to $R$ which is an end-hull, then the reduction modulo $\upsilon$  group homomorphism $\Aut_{R[H]}(\mathcal V)\rightarrow\Aut_{k[H]}(V)$ is also surjective.
\end{remark}

\begin{theorem}\label{T3} 
Let $V$ be a $k[H]$-module. Let $R$ be such that it contains $R_s$. Let $t\in\mathbb Z\setminus p\mathbb Z$. Let $o$ be the order of $[t]_{p^s}\in (\mathbb Z/p^s\mathbb Z)^{\times}$. Then the following properties hold.

\medskip
{\bf (1)} The following statements are equivalent.

\medskip\noindent
{\bf (1.a)} The $k[H]$-module $V$ is cyclic.

\smallskip\noindent
{\bf (1.b)} The $k[H]$-module $V$ is indecomposable.

\smallskip\noindent
{\bf (1.c)} The $k[H]$-module $V$ is strongly indecomposable.

\smallskip\noindent
{\bf (1.d)} We have $\imath_V\in\{0,1\}$.

\smallskip\noindent
{\bf (1.e)} Each lift $\mathcal V$ of $V$ to $R$ is cyclic.

\smallskip\noindent
{\bf (1.f)} Each lift $\mathcal V$ of $V$ to $R$ is cyclic and strongly indecomposable.

\medskip
{\bf (2)} The $k[H]$-module $V$ has a lift $\mathcal H$ to $R$ which is an end-hull of $V$ and a direct sum of cyclic $R[H]$-modules. Moreover, if $\jmath_V=0$, then we can also assume that $r_{\mathcal H}=0$.

\smallskip
{\bf (3)} Suppose that $V$ is cyclic. Let $\mathcal V_1$ and $\mathcal V_2$ be two lifts of $V$. Then the $R[H]$-modules $\mathcal V_1$ and $\mathcal V_2$ are isomorphic if and only if the $B[H]$-modules $\mathcal V_1\otimes_R B$ and $\mathcal V_2\otimes_R B$ are isomorphic.

\smallskip
{\bf (4)} The $k[H]$-modules $V$ and $V^{(t)}$ are isomorphic.

\smallskip
{\bf (5)} The following statements are equivalent.

\medskip\noindent
{\bf (5.a)} There exists a lift $\mathcal V$ of $V$ to $R$ which is $t$-invariant and a direct sum of cyclic $R[H]$-modules (resp.\ which is $t$-invariant and a direct sum of cyclic $R[H]$-modules with $r_{\mathcal V^H}=0$).

\smallskip\noindent
{\bf (5.b)} For each $b\in\llbracket1,p^s\rrbracket\setminus o\llbracket1,\lfloor\frac{p^s}{t}\rfloor\rrbracket$ such that $\mu_V(b)\neq 0$, $\mu_V(b)$ is a sum of cardinals of orbits of the action of $[t]_{p^s}$ on $\mathcal P_b(\Xi_s)$ (resp.\ $\mathcal P_b(\Xi_s^{\ast})$), equivalently of subsets $\Phi$ of $\mathcal P_b(\Xi_s)$ (resp.\ $\mathcal P_b(\Xi_s^{\ast})$) with $[t]_{p^s}\cdot\Phi=\Phi$.

\medskip
{\bf (6)} Suppose that $p$ is odd, $[t]_{p^s}=[p^s-1]_{p^s}$ (equivalently, $o=2$), and $\dim_k(V)$ is even. Then $V$ has a lift to $R$ which is a direct sum of cyclic $R[H]$-modules and is $t$-invariant.

\medskip
{\bf (7)} Suppose that $p$ is odd, $[t]_{p^s}=[p^s-1]_{p^s}$, $\mu_V(d)$ is even for each odd $d\in\llbracket1,p^s\rrbracket$, and $\jmath_V=0$. Then $V$ has a lift $\mathcal V$ to $R$ which is a direct sum of cyclic $R[H]$-modules and is $t$-invariant with $r_{\mathcal V}=0$.
\end{theorem}

\begin{proof}
The case $\dim_k(V)=0$ is trivial for all parts. Thus we can assume that $\dim_k(V)\in\mathbb N$, equivalently, that $\imath_V\ge 1$.

The equivalence $(1.a)\Leftrightarrow (1.b)$ holds for all torsion modules over a quotient of a principal ideal domain such as $k[x]/(x-1)^{p^s}$. 

The implication $(1.b)\Rightarrow (1.c)$ follows from Equation (\ref{EQ18}) as $k[x]/(x-1)^b$ is a local ring (see also \cite{A}, Ch.\ II, Sect.\ 4, Lem.\ 1). 

The principal ideal $\upsilon R[G]$ of $R[G]$ is contained in the radical of $R[G]$ by Jacobson's Theorem (see \cite{Bo}, Sect.\ 9, Thm.\ 1). Hence the implication  $(1.d)\Rightarrow (1.e)$ holds as for $\imath_V=1$, if $v\in\mathcal V$ is such that its reduction modulo $\upsilon$ generates $V$, then from Nakayama's Lemma (see \cite{Bo}, Sect.\ 9, Thm.\ 2) applied to the ideal $\upsilon R[G]$ of $R[G]$ we get that $v$ generates $\mathcal V$.

The implications $(1.c)\Rightarrow (1.d)$, $(1.e)\Rightarrow (1.a)$, and $(1.f)\Rightarrow (1.e)$ are clear. 

We check that $(1.e)\Rightarrow (1.f)$. As $\mathcal V$ is cyclic, it is isomorphic to $R[x]/(\Theta)$ with $\Theta\in R[x]$ a divisor of $x^{p^s}-1$; so $\Theta\in B[x]$ is separable. As $R_s\subset R$, $\Theta$ is also split. Hence the $B$-algebra $\End_{B[H]}(\mathcal V)$ is isomorphic to $B^{\dim_k(V)}$. This implies that the $R$-algebra $\End_{R[H]}(\mathcal V)$ is local if and only if the $k$-algebra $\End_{k[H]}(V)$ is local. From this and the fact that $(1.e)\Rightarrow (1.c)$ we get that the $R$-algebra $\End_{R[H]}(\mathcal V)$ is local. Thus $(1.e)\Rightarrow (1.f)$. So part (1) holds. 

For parts (2) to (7), to ease notation, we can assume that we have a direct sum decomposition of $k[H]$-modules
$$V=\oplus_{i=1}^{\imath_V} k[H]/(x-1)^{b_i}$$
such that $(b_i)_{i\in\llbracket1,\imath_V\rrbracket}$ is a non-decreasing sequence in $\llbracket1,b_{\imath_V}\rrbracket$.

For part (2), let $\nu:=(\lambda_1,\ldots,\lambda_{b_{\imath_V}})\in\Xi_s^{b_{\imath_V}}$ have distinct entries. For $i\in\llbracket1,\imath_V\rrbracket$, let $\mathcal V_i:=R[x]/\bigl(\prod_{j=1}^{b_i} (x-\lambda_j)\bigr)$. Thus the direct sum
$$\mathcal H_{\nu}:=\oplus_{i=1}^{\imath_V} \mathcal V_i$$
of cyclic $R[H]$-modules lifts $V$. To show that $\mathcal H_{\nu}$ is an end-hull of $V$ it suffices to show that for each pair $(i_1,i_2)\in\llbracket1,\imath_V\rrbracket^2$, the $R$-linear map 
$$\mathfrak l_{i_1,i_2}:\Hom_{R[H]}(\mathcal V_{i_1},\mathcal V_{i_2})\rightarrow \Hom_{k[H]}(k[x]/(x-1)^{b_{i_1}},k[x]/(x-1)^{b_{i_2}}\bigr),$$
which is the reduction modulo $\upsilon$, is surjective. 

Using duals, to check this we can assume that $i_2\le i_1$; using Equation (\ref{EQ17}) and its analog $R[H]$-linear identification
$$R[x]/\bigl(\prod_{j=1}^{b_{i_2}} (x-\lambda_j)\bigr)=\Hom_{R[H]}(\mathcal V_{i_1},\mathcal V_{i_2})$$ 
that maps every $\alpha\in R[x]/\bigl(\prod_{j=1}^{b_{i_2}} (x-\lambda_j)\bigr)$ to the $R[H]$-linear map $\break f_{\alpha}\in\Hom_{R[H]}(\mathcal V_{i_1},\mathcal V_{i_2})$ with $f_{\alpha}\bigl(1+(\prod_{j=1}^{b_{i_1}} (x-\lambda_j))\bigr)=\alpha$, we get that $\mathfrak l_{i_1,i_2}$ is surjective. If $\jmath_V=0$, then $b_{\imath_V}\le p^s-1$, and thus we can assume that $\{\lambda_i|i\in\llbracket1,b_{\imath_V}\rrbracket\}\subset\Xi_s^{\ast}$; this implies that $\mathcal H_{\nu}^H=\{0\}$, i.e., $r_{\mathcal H_{\nu}}=0$. So part (2) holds.

The `only if' of part (3) is clear. For the `if' of part (3), both $\mathcal V_1$ and $\mathcal V_2$ are cyclic by part (1). Let $\mathfrak R:=R[[y]]$ be the $R$-algebra of formal power series in the indeterminate $y:=x-1$; it is a regular local ring of dimension $2$. Let $\mathfrak I:=(x^{p^s}-1)\mathfrak R=[(y+1)^{p^s}-1]\mathfrak R$; we have $R[x]/(x^{p^s}-1)=\mathfrak R/\mathfrak I$. For $\iota\in\{1,2\}$, we view $\mathcal V_{\iota}$ as an $\mathfrak R$-module annihilated by $\mathfrak I$. As $\depth_{\mathfrak R}(\mathcal V_{\iota})=1$ and $\depth(\mathfrak R)=2$, from the Auslander--Buchsbaum formula we get that the $\mathfrak R$-module $\mathcal V_{\iota}$ has projective dimension $1$. So we have a projective resolution 
$$0\rightarrow\mathfrak I_{\iota}\rightarrow\mathfrak R\rightarrow \mathcal V_{\iota}\rightarrow 0$$
of $\mathfrak R$-modules. As $\mathfrak R$ is a unique factorization domain, it follows that there exists $\Upsilon_{\iota}\in\mathfrak R$ such that $\mathfrak I_{\iota}=\Upsilon_{\iota}\mathfrak R$. As $\Upsilon_{\iota}$ annihilates $\mathcal V_{\iota}$, its image in $\mathfrak R/\mathfrak I=R[x]/(x^{p^s}-1)$ is the same as the image in $R[x]/(x^{p^s}-1)$ of a divisor $\Theta_{\iota}$ of $x^{p^s}-1$ in $R[x]$ of degree $\dim_k(V)$ which we can assume to be monic. As the monic divisors of $x^{p^s}-1$ over $R$ and $B$ are the same and the $B[H]$-modules $\mathcal V_1\otimes_R B$ and $\mathcal V_2\otimes_R B$ are isomorphic, it follows that $\Theta_1=\Theta_2$. Hence $\mathcal V_1$ and $\mathcal V_2$ are both isomorphic to $R[x]/(\Theta_1)$. So the `if' of part (3) holds. Thus part (3) holds.

To check that part (4) holds, we can assume that $V$ is cyclic. So $V^{(t)}$ is also cyclic and $\dim_k(V^{(t)})=\dim_k(V)$. Thus $V$ and $V^{(t)}$ are isomorphic to $k[x]/(x-1)^{\dim_k(V)}$. So part (4) holds.

For part (5), let $\mathfrak b:=\{b_i|i\in\llbracket1,\imath_V\rrbracket\}$. 

We show that $(5.a)\Rightarrow (5.b)$. Based on the equivalence $(1.a)\Leftrightarrow (1.f)$, we have a direct sum decomposition $\mathcal V=\oplus_{i=1}^{\imath_V} \mathcal V_i$ into cyclic $R[H]$-modules such that for each $i\in\llbracket1,\imath_V\rrbracket$ the $R$-module $\mathcal V_i$ is free of rank $b_i$. For $i\in\llbracket1,\imath_V\rrbracket$, let $\Theta_i\in R[x]$ be a divisor of $x^{p^s}-1$ such that the $R[H]$-module $\mathcal V_i$ is isomorphic to $R[x]/(\Theta_i)$ and let $\Lambda_i:=\{\lambda\in\Xi_s|\Theta_i(\lambda)=0\}$; we have $|\Lambda_i|=b_i$ as $\Theta_i$ is separable over $B$ and the $R$-module $\mathcal V_i$ has rank $b_i$. 

For $b\in\mathfrak b$, let $I_b:=\{i\in\llbracket1,\imath_V\rrbracket|b_i=b\}$; we have $|I_b|=\mu_V(b)$. As $\mathcal V_i\cong R[x]/(\Theta_i)=R[x]/\bigl(\prod_{\lambda\in\Lambda_i} (x-\lambda)\bigr)$, from Equation (\ref{EQ15}) we get that 
\begin{equation}\label{EQ31}
\mathcal V_i^{(t)}\cong R[x]/\bigl(\prod_{\lambda\in\Lambda_i} (x-\lambda^{t'})\bigr)=R[x]/\bigl(\prod_{\lambda\in [t']_{p^s}\cdot\Lambda_i} (x-\lambda)\bigr).
\end{equation}
As $\mathcal V\cong\mathcal V^{(t)}$, as the isomorphism class of each $\mathcal V_i$ (resp.\ $\mathcal V_i^{(t)}$) is uniquely determined by $\Lambda_i$ (resp.\ $[t']_{p^s}\cdot \Lambda_i$), and as the isomorphism class of $\mathcal V$ (resp.\ $\mathcal V^{(t)}$) is uniquely determined by the isomorphism classes of the $\mathcal V_i$s (resp.\ $\mathcal V^{(t)}_i$s) by the supplement \cite{AF}, Ch.\ 3, Sect.\ 12, Thm.\ 12.4 to the Krull--Schmidt--Remak--Azumaya theorem (e.g., see \cite{AF}, Ch.\ 3, Sect.\ 12, Thm.\ 12.6), it follows that for each $b\in\mathfrak b$ there exists a permutation $\vartheta_b$ of $I_b$ such that $\Lambda_i=[t']_{p^s}\cdot \Lambda_{\vartheta_b(i)}$ for each $i\in I_b$. This implies that $\mu_V(b)=|I_b|$ is a sum of cardinals of orbits of the permutation action of $[t']_{p^s}$ or $[t]_{p^s}=[t']^{-1}_{p^s}$ on the subset $\{\Lambda_i|i\in I_b\}$ of $\mathcal P_b(\Xi_s)$ (resp.\ $\mathcal P_b(\Xi_s^{\ast})$). Thus $(5.a)\Rightarrow (5.b)$.

To show that $(5.b)\Rightarrow (5.a)$, we can assume that $\mathfrak b=\{b\}$ has a unique element; thus $b_1=\cdots=b_{\imath_V}=b$ and we use only $b$. If $o\mid b$, then let $\Phi_b\in\mathcal P_b(\Xi_s^{\ast})$ be such that $|\Phi_b|=1$ and the unique element of $\Phi_b$ is a disjoint union of $\frac{b}{o}$ subsets of $\Xi_s$ (resp.\ $\Xi_s^{\ast}$) of the form 
$$\{\lambda^{t^i}|i\in\llbracket1,o\rrbracket\}$$
for a suitable $\lambda\in \Xi_s^{\ast}$; as $[t]_{p^s}\cdot\Phi_b=\Phi_b$ and $\mu_V(b)=\imath_V$ is a sum of $\imath_V$ terms equal to $1=|\Phi_b|$, it follows that the hypothesis of part (5.b) holds even if $o\mid b$.

Thus, to prove that $(5.b)\Rightarrow (5.a)$ we can also assume that there exists a subset $\Phi_b$ of $\mathcal P_b(\Xi_s)$ (resp.\ $\mathcal P_b(\Xi_s^{\ast})$) with $[t]_{p^s}\cdot\Phi_b=\Phi_b$ and $|\Phi_b|=\imath_V$. We take 
$$\mathcal V:=\oplus_{\Lambda\in\Phi_b} R[x]/\bigl(\prod_{\lambda\in\Lambda} (x-\lambda)\bigr).$$
As $\mathcal V^{(t)}\cong\oplus_{\Lambda\in [t']_{p^s}\cdot\Phi_b} R[x]/\bigl(\prod_{\lambda\in\Lambda} (x-\lambda)\bigr)$ by Equation (\ref{EQ15}) and we have $[t]_{p^s}\cdot\Phi_b=\Phi_b$, we get that $\mathcal V^{(t)}\cong\mathcal V$. Thus $(5.b)\Rightarrow (5.a)$. So part (5) holds.

For parts (6) and (7), we list $\Xi_s^{\ast}=\{\lambda_i|i\in\llbracket1,p^s-1\rrbracket\}$ in such a way that for each $j\in\llbracket1,\frac{p^s-1}{2}\rrbracket$ we have $\lambda_{2j}=\lambda_{2j-1}^{-1}$. 

For $i\in\llbracket1,\imath_V\rrbracket$ with $b_i$ even (resp.\ odd), let $\Lambda_i:=\{\lambda_i|i\in\llbracket1,b_i\rrbracket\}$ (resp.\ $\Lambda_i:=\{1\}\cup\{\lambda_i|i\in\llbracket1,b_i-1\rrbracket\}$). 

Let $\mathcal V:=\oplus_{i=1}^{\imath_V} R[x]/\bigl(\prod_{\lambda\in\Lambda_i} (x-\lambda)\bigr)$; it is a direct sum of cyclic $R[H]$-modules by part (1) which lifts $V$. As for each $i\in\llbracket1,\imath_V\rrbracket$ we have an identity $[p^s-1]_{p^s}\cdot\Lambda_i=\Lambda_i$, $\mathcal V$ is $t$-invariant by Equation (\ref{EQ31}). So part (6) holds.

For part (7), for a fixed odd $b\in\mathfrak b$, we modify the $\Lambda_i$s with $b_i=b$ by replacing $1$ half times by $\lambda_b$ and half times by $\lambda_{b+1}=\lambda_b^{-1}$. In terms of cardinals, for such a $b$ we have $b\le p^s-2$ as $\jmath_V=0$ and $\mu_V(b)\in 2\mathbb N$ by hypotheses, hence $\mu_V(b)$ is a sum of $2$s and $2$ is the cardinality of the subset 
$$\Phi_b:=\{\lambda_i|i\in\llbracket1,b_i\rrbracket\}\cup\{\lambda_i|i\in\llbracket1,b_i-1\rrbracket\cup\{b_i+1\}\}$$ 
of $\mathcal P_b(\Xi_s^{\ast})$ for which we have $[t]_{p^s}\cdot \Phi_b=\Phi_b$. Thus from Equation (\ref{EQ31}) we get $\mathcal V$ is $t$-invariant. So part (7) holds.\end{proof}

\begin{remark}\normalfont\label{R3}
{\bf (1)} If $H$ has at least $3$ elements, then Theorem \ref{T3}(3) does not hold if $V$ is not cyclic as one can easily see based on Krull--Schmidt--Remak--Azumaya Theorem and its mentioned supplement.

\smallskip
{\bf (2)} If $V$ is cyclic, then for each $R$, every lift of $V$ to $R$ is an end-hull. To check this we can assume that $R$ contains $R_s$ and this case follows from the surjectivity of $\mathfrak l_{i_1,i_2}$ for $i_1=i_2=1$ in the proof of Theorem \ref{T3}(2).
\end{remark}

\begin{corollary}\label{C1}
Suppose that $R$ contains $R_s$. Let $b\in\llbracket1,p^s\rrbracket$. Then the following properties hold.

\medskip
{\bf (1)} For each subset $\Lambda\in\mathcal P_b(\Xi_s)$, there exists a unique $R[H]$-module $\mathcal U(b;\Lambda)$ up to isomorphisms which is a lift of $U(b)$ and we have 
$$\kappa_{\mathcal U(b;\lambda)\otimes_R B}(\zeta)=\begin{cases}1\quad {\rm if}\quad \zeta\in\Lambda\\
0\quad {\rm if}\quad \zeta\notin\Lambda.\end{cases}$$ 
We have $\mathcal U(b;\Lambda)\cong R[x]/\bigl(\prod_{\zeta\in\Lambda} (x-\zeta)\bigr)$.

\smallskip
{\bf (2)} The isomorphism classes of lifts of $U(b)$ to $R$ are the $\mathcal U(b;\Lambda)$s with $\Lambda\in\mathcal P_b(\Xi_s)$.

\smallskip
{\bf (3)} The $R[H]$-module $\mathcal U(p^s-1;\Xi_s^{\ast})$ is the unique lift of $U(p^s-1)$ to $R$ with $r_{\mathcal U(p^s-1;\Xi_s^{\ast})}=0$ up to isomorphisms. Moreover, $\mathcal U(p^s-1;\Xi_s^{\ast})$ is uniquely defined over $R_0$, i.e., there exists a unique lift $\mathcal U_0(p^s-1;\Xi_s^{\ast})$ of $U(p^s-1)$ to $R_0$ up to isomorphisms with $\mathcal U_0(p^s-1;\Xi_s^{\ast})\otimes_{R_0} R\cong \mathcal U(p^s-1;\Xi_s^{\ast})$.
\end{corollary}

\begin{proof}
Parts (1) and (2) follows from Theorem \ref{T3}(2) and (3) and its proof. 

As $\Xi_s^{\ast}$ is the only subset of $\Xi_s$ that has $p^s-1$ elements and does not contain $1$, the first sentence of part (3) follows from parts (1) and (2) applied to $b=p^s-1$. Let
\begin{equation}\label{EQ32}
\mathcal U_0(p^s-1;\Xi_s^{\ast}):=R_0[x]/\bigl(\prod_{\zeta\in\Xi_s^{\ast}} (x-\zeta)\bigr)=R_0[x]/\Bigl(\sum_{i=0}^{p^s-1} x^i\Bigr).
\end{equation}
Its uniqueness follows from Theorem \ref{T3}(3) and Lemma \ref{L2} applied to the field extension $B_0\rightarrow B$. We have 
$$\mathcal U_0(p^s-1;\Xi_s^{\ast})\otimes_{R_0} R=R[x]/\bigl(\prod_{\zeta\in\Lambda} (x-\zeta)\bigr)\cong \mathcal U(p^s-1;\Xi_s^{\ast}).$$ So part (3) holds.
\end{proof}

\begin{lemma}\label{L6}
The following properties hold.

\medskip
{\bf (1)} We have a short exact sequence of $R_0[H]$-modules
$$0\rightarrow R_0[H]^H\rightarrow R_0[H]\rightarrow\mathcal U_0(p^s-1;\Xi_s^{\ast})\rightarrow 0.$$

{\bf (2)} Let $V$ be a $k[H]$-module such that it has a lift $\mathcal V$ to $R$ with $r_{\mathcal V}=0$. Then we have a commutative diagram of $R[H]$-modules
\begin{equation}\label{EQ34}
\xymatrix@R=10pt@C=21pt@L=2pt{
\Hom_{R[H]}\bigl(\mathcal U_0(p^s-1;\Xi_s^{\ast})\otimes_{R_0} R,\mathcal V\bigr) \ar[r]\ar[d] &\Hom_{R[H]}(R[H],\mathcal V)=\mathcal V\ar[d]\\
\Hom_{k[H]}\bigl(U(p^s-1),V\bigr) \ar[r] & \Hom_{k[H]}(k[H],V)=V\\}
\end{equation}
whose vertical arrows are reductions modulo $\upsilon$ and are surjective and whose horizontal arrows are isomorphisms.
\end{lemma}

\begin{proof}
Part (1) follows from Equation (\ref{EQ32}) and the fact that in the short exact sequence $0\rightarrow (\sum_{i=0}^{p^s-1} x^i)\rightarrow R_0[H]\rightarrow R_0[x]/(\sum_{i=0}^{p^s-1} x^i)\rightarrow 0$ of $R_0[H]$-modules, the ideal $(\sum_{i=0}^{p^s-1} x^i)=(\sum_{i=0}^{p^s-1}\tau^i)$ of $R_0[H]$ is $R_0[H]^H$.

For part (2), $r_{\mathcal V}=0$ implies $\Hom_{R[H]}(R[H]^H,\mathcal V)=0$. Based on this, by applying $\Hom_{R[H]}(-,\mathcal V)$ to the extension of Short Exact Sequence (\ref{EQ32}) to $R$ we get that Diagram (\ref{EQ34}) is as mentioned; so part (2) holds. 
\end{proof}

\section{On $k[G]$-modules with constant character functions}\label{S5}

As the category of $k[\Sigma]$-modules is semisimple, for each short exact sequence $0\rightarrow V_1\rightarrow V\rightarrow V_2\rightarrow 0$ of $k[G]$-modules we have an identity
$$\ell_V=\ell_{V_1}+\ell_{V_2}.$$
So the functions $\ell_V$, $\ell_{V_1}$, and $\ell_{V_2}$ are all constant if and only if two of them are constant. This explains the form of the first part of the following definition.

\begin{definition}\label{D4} 
Let $V$ be a $k[G]$-module with a constant function $\ell_V$. 

\medskip 
{\bf (1)} We say that $V$ is constant-indecomposable if each direct summand $k[G]$-submodule $U$ of it with $\ell_U$ constant is $\{0\}$ or $V$.

\smallskip
{\bf (2)} We say that $\natural_V$ holds if $\imath_V\in\llbracket2,m\rrbracket$ and there exists a direct sum decomposition $V=\oplus_{i=1}^{\imath_V} U_{a_i,b_i}$ such that, with $(a_{\imath_V+1},b_{\imath_V+1}):=(a_1,b_1)$, for each $i\in\llbracket1,\imath_V\rrbracket$ we have $a_{i+1}=a_i-b_i\epsilon$, the subset $\{a_i|i\in\llbracket1,\imath_V\rrbracket\}$ of $C_m$ has $\imath_V$ elements, and for every pair $(i_1,i_2)\in\llbracket1,\imath_V\rrbracket^2$ that satisfies $i_1<i_2<i_1+\imath_V$, we have $m\nmid \sum_{j=i_1}^{i_2-1} b_j$.
\end{definition} 

In the proof of Theorem \ref{T2} we use the following proposition.

\begin{proposition}\label{P1}
Suppose that $\chi$ is injective. Let $V$ be a $k[G]$-module with $\ell_V$ constant. Then $V$ is constant-indecomposable if and only if either $\imath_V=1$ or $\natural_V$ holds.
\end{proposition}

\begin{proof}
We can assume that $\imath_V\ge 2$; so $m\ge 2$. We consider a direct sum decomposition $V=\oplus_{i=1}^{\imath_V} U_{a_i,b_i}$ into indecomposable $k[G]$-modules; hence $(a_i,b_i)\in\mathbb J_{m,p^s}$ for each $i\in\llbracket1,\imath_V\rrbracket$. As $\ell_V$ is constant, it follows that $\dim_k(V)=m\eta_V\ge m$. So $\imath_V+\dim_k(V)\ge m+2$.

We show that the `only if' part holds. As $V$ is constant-indecomposable and $\imath_V\ge 2$, for each $i\in\llbracket1,m\rrbracket$ for the direct summand $U_{a_i,b_i}$ of $V$ the character function $\ell_{U_{a_i,b_i}}$ is not constant and hence $m\nmid b_i$. 

For each $i\in\llbracket1,\imath_V\rrbracket$, let $a_i':=a_i-b_i\epsilon\in C_m$. 

As $\ell_V$ is constant, we show by induction on $\imath_V+\dim_k(V)$ in $\mathbb N\cap [m+2,\infty)$ that by reindexing we can assume that $a_i'=a_{i+1}$ for each $i\in\llbracket1,\imath_V\rrbracket$. 

If $\imath_V+\dim_k(V)=m+2$, then $\imath_V=2$, $\dim_k(V)=b_1+b_2=m$, and we have $C_m=\{a_1-j\epsilon|j\in\llbracket0,b_1-1\rrbracket\}\cup \{a_2-j\epsilon|j\in\llbracket0,b_2-1\rrbracket\}$ by Equation (\ref{EQ3}); this implies that the unique $j\in\llbracket0,b_2-1\rrbracket$ with $a_1'=a_2-j\epsilon$ is $j=0$. So the base of the induction holds. 

For an integer $n>m+2$, we assume that the statement on reindexing holds if $\imath_V+\dim_k(V)<n$ and we show that it also holds if $\imath_V+\dim_k(V)=n$. If $a_1'\notin\{a_i|i\in\llbracket2,\imath_V\rrbracket\}$, then $\ell_V\bigl(a_1-(b_1-1)\epsilon\bigr)>\ell_V(a_1')$, a contradiction. Thus, by reindexing we can assume that  $a_1'=a_2$. If $\imath=2$, then we similarly argue that $a_2'=a_1$. Thus we can assume that $\imath_V\ge 3$. 

Let $\varepsilon\in\{0,1\}$ be such that it is $1$ if and only if $b_1+b_2> p^s$. We consider the $k[G]$-module $$V_1:=U\bigl(a_1,b_1+b_2-\varepsilon(p^s-1)\bigr)\bigoplus\oplus_{i=3}^{\imath_V} U(a_i,b_i).$$ 
The function $\ell_{V_1}$ is constant and we have identities $\eta_{V_1}=\eta_{V}-\frac{\varepsilon(p^s-1)}{m}$, $\dim_k(V_1)=\dim_k(V)-\varepsilon(p^s-1)$, $\imath_{V_1}=\imath_V-1\ge 2$, and 
$$a_1-[b_1+b_2-\varepsilon(p^s-1)]\epsilon=a_1-(b_1+b_2)\epsilon=a_1'-b_2\epsilon=a_2-b_2\epsilon=a_2'\in C_m.$$ 
Thus $\dim_k(V_1)+\imath_{V_1}=n-1-\varepsilon(p^s-1)<n$; hence, by the induction hypotheses, up to a reindexing of the indices in $\llbracket3,\imath_V\rrbracket$ we can assume that $a_i'=a_{i+1}$ for each $i\in\llbracket2,\imath_V\rrbracket$. This ends the inductive step and hence the induction.

If there exists a pair $(i_1,i_2)\in\llbracket1,\imath_V+1\rrbracket^2$ such that $i_1<i_2<i_1+\imath_V$ and $a_{i_1}=a_{i_2}$, then for the direct summand $W:=\oplus_{i=i_1}^{i_2-1} U_{a_i,b_i}$ of $V$ we have $\{0\}\subsetneq W\subsetneq V$ and $\ell_W$ is constant, a contradiction to $V$ being constant-indecomposable. Thus the elements $a_1,\ldots,a_{\imath_V}$ in $C_m$ are distinct and therefore $\imath_V\le m$. As $a_{i_1}\neq a_{i_2}$ is equivalent to the sum $\sum_{j=i_1}^{i_2-1} b_j$ not being divisible by $m$, $\natural_V$ holds. So the  `only if' part holds.

For the `if' part, it suffices to show that for each subset $A$ of $\llbracket1,\imath_V\rrbracket$ with $1\le |A|\le \imath_V-1$, the assumption that the $k[G]$-module $W:=\oplus_{i\in A} U_{a_i,b_i}$ has a constant character function leads to a contradiction. We can assume that $W$ is constant-indecomposable; so either $\imath_W=1$ or $\natural_W$ holds. We have $|A|=\imath_W$. We list $A=\{i_l|k\in\llbracket1,\imath_W\rrbracket\}$ such that $i_1<i_2<\cdots<i_{\imath_W}$; let $i_{\imath_W+1}:=i_1$. 
If $\imath_W=1$, then $m$ divides $b_{i_1}$, a contradiction to $\natural_V$ applied to $i_2=i_1+1$. If $\imath_W\ge 2$, then the fact that $\natural_W$ holds implies that there exists a permutation $\theta$ of $\llbracket1,\imath_W\rrbracket$ such that for $j\in\llbracket2,\imath_W+1\rrbracket$ we have $a_{i_{\theta(j)}}=a_{i_{\theta(j-1)}}-b_{i_{\theta(j-1)}}\epsilon$. As the subset $\{a_i|i\in\llbracket1,\imath_V\rrbracket\}$ of $C_m$ has $\imath_V$ elements, from this and the fact that $\natural_V$ holds we get that $\theta$ is the identity permutation and $|A|=\imath_V$, a contradiction to $|A|\le \imath_V-1$. Thus the `if' part holds.
\end{proof}

\smallskip
The following example shows that for $m\ge 2$, the inequality $\imath_V\le m$ in Definition \ref{D4}(2) is optimal.

\begin{ex}\normalfont\label{EX6}
Suppose that $m\ge 2$ and $\chi$ is injective. As $U_{[0]_m,p^s}$, when viewed as a $k[H]$-module, is isomorphic to $k[H]$, from Equation (\ref{EQ12}) we get that we have a $k[G]$-linear isomorphism
\begin{equation}\label{EQ35}
k[G]\cong V_{p^s}^{\circul}=\oplus_{a\in C_m} U(a,p^s)
\end{equation}
which we view as an identification
Thus $\ell_{k[G]}$ is constant, $\eta_{k[G]}=p^s$, and $\imath_{k[G]}=\jmath_{k[G]}=m$; so we use only $m$. To check that $\natural_{k[G]}$ holds, as we have $V_{p^s}^{\circul}=\oplus_{a\in C_m} U(a,p^s)$, we can assume that we have a direct sum decomposition $k[G]=\oplus_{i=1}^m U(a_i,b_i)$ such that $a_i=[1-(i-1)\epsilon]_m$ for each $i\in\llbracket1,m\rrbracket$ and that $b_1=\cdots=b_m=p^s$. As $a_i'=a_i-p^s\epsilon=a_i-\epsilon\in C_m$, we get that $a_i'=a_{i+1}$ for each $i\in\llbracket1,m\rrbracket$, with $a_{m+1}:=a_1$. Clearly, the set $\{a_i|i\in\llbracket1,m\rrbracket\}$ has $m$ elements. For every pair $(i_1,i_2)\in\llbracket1,m\rrbracket^2$ that satisfies $i_1<i_2<i_1+m$, we have a congruence $\sum_{j=i_1}^{i_2-1} b_j\equiv i_2-i_1\pmod{m}$ and thus $m\nmid \sum_{j=i_1}^{i_2-1} b_j$. So $\natural_{k[G]}$ holds.
\end{ex}

\begin{corollary}\label{C2}
For each $a\in C_m$, there exists a unique lift $\mathcal U(a,p^s)$ of $U(a,p^s)$ to $R$ up to isomorphisms. Moreover, $\mathcal U(a,p^s)$ is a cyclic projective $R[G]$-module.
\end{corollary}

\begin{proof}
We show by induction on $l\in\mathbb N\cup\{0\}$ that we have a direct sum decomposition $f_l:(R/\upsilon^{2^l}R)[G]=\oplus_{a\in C_m} \mathcal U_{a,p^s;2^l}$ into $(R/\upsilon^{2^l}R)[G]$-modules which lifts Equation (\ref{EQ35}) and, if $l\ge 1$, whose reduction modulo $\upsilon^{2^{l-1}}R$ is $f_{l-1}$. The base of the induction for $l=0$ holds by Equation (\ref{EQ35}). For $l\ge 2$, the passage from $l-1$ to $l$ follows from \cite{Bo}, Sect.\ 9, Subsect.\ 4, Cor.\ 1. As $R[G]$ is complete in the $\upsilon$-adic topology, we conclude that there exists a direct sum decomposition $f:R[G]=\oplus_{a\in C_m} \mathcal U_{a,p^s}$ into $R[G]$-modules which lifts Equation (\ref{EQ35}). So each $\mathcal U(a,p^s)$ is a projective lift of $U(a,b)$ to $R$ and it is cyclic as the $R[H]$-module $\mathcal U(a,p^s)$ is so by Theorem \ref{T3}(1). 

If $\mathcal U'(a,p^s)$ is a lift of $U(a,p^s)$ to $R$, then, as $\mathcal U(a,p^s)$ is a projective $R[G]$-module, there exists an $R[G]$-linear map $\mathcal U(a,p^s)\rightarrow \mathcal U'(a,p^s)$ whose reduction modulo $\upsilon$ is an isomorphism and hence itself is an isomorphism by Nakayama's Lemma.
\end{proof}

\begin{proposition}\label{P2}
Suppose that $\chi$ is injective. Let $V$ be a $k[G]$-module with $\ell_V$ constant. We assume that $V$ is constant-indecomposable and we consider a direct sum decomposition into indecomposable $k[G]$-modules $V=\oplus_{i=1}^{\imath_V} U(a_i,b_i)$. Let 
$$J:=\{a_i+\epsilon|i\in\llbracket1,\imath_V\rrbracket\}.$$
Let $a\in C_m$. We consider the $k[G]$-module $W:=U(a,1)$.  Then the following properties hold.

\medskip
{\bf (1)} We have an identity
$$\dim_k\bigl(\Hom_{k[G]}(V,W)\bigr)=\begin{cases}1\quad\textup{if}\quad a\in J\\0\quad\textup{if}\quad a\notin J.\end{cases}$$

{\bf (2)} Let $(c,b)\in\mathbb J_{m,p^s}$. Then 
$$\dim_k\bigl(\Ext^1_{k[G]}(U(c,b),W)\bigr)=\begin{cases}1\quad\textup{if}\quad  a=c+\epsilon\;\textup{and}\;b<p^s\\0\quad\textup{if}\quad a\neq c+\epsilon\;\textup{or}\;b=p^s.\end{cases}$$

{\bf (3)} Suppose that $\jmath_V=0$. Then 
$$\dim_k\bigl(\Hom_{k[G]}(V,W)\bigr)=\dim_k\bigl(\Ext^1_{k[G]}(V,W)\bigr).$$
\end{proposition}

\begin{proof}
Based on Proposition \ref{P1} we can assume we have $a_i'=a_{i+1}$ for each $i\in\llbracket1,\imath_V\rrbracket$, where $a_{\imath_V+1}:=a_1$ and $a_i':=a_i-b_i\epsilon$. This makes sense even if $\imath_V=1$ as in this case $m\mid b_1$ and hence $a_1'=a_1=a_2$. 

For $c\in C_m$ we have 
\begin{equation}\label{EQ36}
\Hom_{k[G]}\bigl(U(c,1),W\bigr)=\begin{cases}\End_{k[G]}(W)\quad\;\textup{if}\;c=a\\
0\quad\quad\quad\quad\quad\quad\textup{if}\;c\neq a.\end{cases}
\end{equation} 
We compute
\begin{equation}\label{EQ37}
\begin{split}
\Hom_{k[G]}(V,W)=\oplus_{i=1}^{\imath_V}\Hom_{k[G]}(U(a_i,b_i),W)\\
=\oplus_{i=1}^{\imath_V}\Hom_{k[G]}\bigl(U(a_i-(b_i-1)\epsilon,1),W\bigr)=\oplus_{i=1}^{\imath_V}\Hom_{k[G]}\bigl(U(a_{i+1}+\epsilon,1),W\bigr)\\
=\oplus_{i\in\llbracket1,\imath_V\rrbracket,a_{i+1}+\epsilon=a} \End_{k[G]}(W),\\
\end{split}
\end{equation}
where the second and fourth identities follow from the description of the unique composition series of each $U(a_i,b_i)$ and Equation (\ref{EQ36}) (respectively). 
From Equation (\ref{EQ37}) we get that $\dim_k\bigl(\Hom_{k[G]}(V,W)\bigr)\in\{0,1\}$ and that $\dim_k\bigl(\Hom_{k[G]}(V,W)\bigr)=1$ if and only if $a\in J$.

For part (2), we consider a short exact sequence of $k[G]$-modules 
\begin{equation}\label{EQ38}
0\rightarrow W\rightarrow V^+\rightarrow U(c,b)\rightarrow 0.
\end{equation} 
If $b=p^s$, then the $k[G]$-module $U(c,b)$ is projective and hence we have $V^+\cong W\oplus U(c,b)$ and $\Ext^1_{k[G]}(U(c,b),W)=0$. Thus we can assume that $b\le p^s-1$.

We have $\dim_k(V^+)=b+1$ and, as $V^+$ surjects onto $U(c,b)$, $V^+$ has a direct summand which is indecomposable of dimension at least $b$ over $k$. It follows that either $\imath_{V^+}=2$ and $V^+\cong U(c_1,1)\oplus U(c_2,b)$ or $\imath_{V^+}=1$ and $V\cong U(c_0,b+1)$, where $(c_0,c_1,c_2)\in C_m^3$. 

If $\imath_{V^+}=1$, then $V\cong U(c_0,b+1)$ and from Short Exact Sequence (\ref{EQ38}) we get by reasons of socles that $a=c_0$ and by reasons of uniserialness that $c=c_0-\epsilon$; thus $a=c+\epsilon$, $V\cong U(a,b)$, and the Short Exact Sequence (\ref{EQ38}) does not split. 

If $\imath_{V^+}=2$ and $b\ge 2$, then $V^+\cong U(c_1,1)\oplus U(c_2,b)$ and from Short Exact Sequence (\ref{EQ38}) we get that there exist surjective $k[G]$-linear maps $U(c_2,b)\rightarrow U(c,b)$ and $W\rightarrow U(c_1,1)$. Hence $a=c_1$ and $c=c_2$ and the Short Exact Sequence (\ref{EQ38}) splits.

If $\imath_{V^+}=2$ and $b=1$, then $V^+\cong U(c_1,1)\oplus U(c_2,b)$, $\{c_1,c_2\}=\{a,c\}$, and it is easy to see that the Short Exact Sequence (\ref{EQ38}) splits.

The last three paragraphs give that $\Ext^1_{k[G]}\bigl(U(c,b),W\bigr)=\{0\}$ if $a\neq c+\epsilon$. 

We are left to prove that if we have $a=c+\epsilon$ and $b\le p^s-1$, then $\dim_k\bigl(\Ext^1_{k[G]}(U(c,b),W)\bigr)=1$. As $V^+$ can be isomorphic to $U(a,b)$ and in such a case the Short Exact Sequence (\ref{EQ38}) does not split, we get that $\Ext^1_{k[G]}\bigl(U(c,b),W\bigr)\neq 0$. Each non-zero class $\beta\in\Ext^1_{k[G]}\bigl(U(c,b),W\bigr)$ is defined by a Short Exact Sequence (\ref{EQ38}) with $V^+=U(a,b)$. We fix a generator $\overline{v}\in U(c,b)$; each $\overline{v}^+\in V^+$ that maps to $\overline{v}$ is a generator of $V^+$. So we can assume that the surjective $k[G]$-linear map $V^+\rightarrow U(c,b)$ is fixed; its kernel is $\soc_{k[G]}(V^+)$. So the other non-zero classes in $\Ext^1_{k[G]}(U(c,b),W)$ are defined by injective $k[G]$-linear maps $W\rightarrow \soc_{k[G]}(V^+)$ and hence are all of the form $\gamma\beta$ with $\gamma\in k^{\times}$. Thus $\dim_k\bigl(\Ext^1_{k[G]}(U(c,b),W)\bigr)\le 1$ and part (2) holds.

As $\Ext^1_{k[G]}(V,W)=\oplus_{i=1}^{\imath_V} \Ext^1_{k[G]}\bigl(V,U(a_i,b_i)\bigr)$, part (3) follows from part (1) and part (2) applied with $(c,b)\in\{(a_i,b_i)|i\in\llbracket1,\imath_V\rrbracket\}$.
\end{proof}

\section{Proof of Theorem \ref{T2}(1)}\label{S6}

The image of $\mathcal V^H$ in $V$ is a $k$-vector subspace of $\soc_{k[G]}(V)$ of dimension $r_{\mathcal V}$. Thus $r_{\mathcal V}\le \imath_V$. 

Let $\Chi_B(H)$ be the group of characters $H\rightarrow B^{\times}$ and let $\psi_0$ be its identity element. As $R_s\subset R$, we have $|\Chi_B(H)|=p^s$; so $\Chi_B(H)$ is the dual of $H$. 

The group $\Sigma$ acts on $\Chi_B(H)$ via the rule: if $\psi\in\Chi_B(H)$, then $(\sigma\psi)(\tau)=\psi\bigl(\chi(\tau)\bigr)$. 

As we are assuming $\chi$ is injective, $m\mid p-1$ and this action has one orbit $\{\psi_0\}$ with $1$ element and $\frac{p^s-1}{m}$ orbits with $m$ elements.

We consider a direct sum decomposition 
$$\mathcal V\otimes_{R} B=\oplus_{j\in J} \mathcal V_j$$ 
into simple $B[G]$-modules. From the prior paragraph we get that for each $j\in J$ we have $\dim_{B}(\mathcal V_j)\in\{1,m\}$; we have $\dim_{B}(\mathcal V_j)=1$ if and only if  $\mathcal V_j\subset\mathcal V^H\otimes_R B$. 

Let $J_0:=\{j\in J|\dim_{B}(\mathcal V_j)=1\}$. We compute
\begin{equation}\label{EQ40}
\dim_k(V)=\dim_B(\mathcal V\otimes_R B)=\sum_{j\in J} \dim_B(\mathcal V_j)=|J_0|+m|J\setminus J_0|.
\end{equation}

In this paragraph we consider $j\in J\setminus J_0$. We have a direct sum decomposition $\mathcal V_j=\oplus_{l=1}^m \mathcal V_{j,l}$ into simple $B[H]$-modules of dimension $1$ over $B$. If $\psi_{j,l}\in\Chi_B(H)$ is such that $\tau$ acts on $\mathcal V_{j,l}$ via $\psi_{j,l}(\tau)$, then the set $\{\psi_{j,l}|l\in\llbracket1,m\rrbracket\}$ is an orbit of the action of $\Sigma$ on $\Chi_B(H)$ and hence $\psi_{j,l}\neq\psi_0$. Thus, there exists a $B$-basis $\mathcal B_j=\{v_{j,l}|l\in\llbracket1,m\rrbracket\}$ of $\mathcal V_j$ contained in $\cup_{l=1}^m \mathcal V_{j,l}$ and such that $\sigma(v_{j,l})=v_{j,l+1}$ if $l\in\llbracket1,m-1\rrbracket$ and $\sigma(v_{j,m+1})=v_{j,1}$. The (monic) characteristic polynomial of the $B$-linear endomorphism $\sigma:\mathcal V_{j}\rightarrow\mathcal V_j$ is $x^m-1=\prod_{i=1}^m (x-\zeta_m^i)$. Hence $\mathcal V_j$ contributes with value $1$ to each $\ell_V(a)$ with $a\in C_m$. 

Therefore $|J\setminus J_0|\le\eta_V$. Based on this and Equation (\ref{EQ40}), we estimate
\begin{equation}\label{EQ41}
r_{\mathcal V}=\dim_B(\mathcal V^H\otimes_R H)=|J_0|=\dim_k(V)-m|J\setminus J_0|\ge \dim_k(V)-m\eta_V.
\end{equation}
Equation (\ref{EQ5}) implies that $\sum_{a\in C_m} [\ell_V(a)-\eta_V]=\dim_k(V)-m\eta_V$. 

For each $a\in C_m$, we have an isomorphism $k[H]\rightarrow U_{a,p^s}$ of $k[H]$-modules. Thus we have an injective $k[H]$-linear map $\overline{f}:k[H]^{\jmath_V}\rightarrow V$ and hence an $R[H]$-linear map $f:k[H]^{\jmath_V}\rightarrow V$ whose reduction modulo $\upsilon$ is $\overline{f}$. It follows that we have a commutative diagram of $R[H]$-modules
$$\begin{tikzcd}[row sep=tiny]
R[H]^{\jmath_V} \arrow[rd, "f"'] \arrow[rr, "h"] & & \mathcal V \arrow[ld, "\xi"] \\
& V
\end{tikzcd}$$
in which $\xi$ is the quotient $R[H]$-linear map defined by $V=\mathcal V/\upsilon\mathcal V$. As $h$ modulo $\upsilon$, i.e., $\overline{f}$, is injective, it follows that $h$ is injective. Thus $h\bigl((R[H]^H)^{\jmath_V}\bigr)\subset \mathcal V^H$. Hence $\jmath_V\le r_{\mathcal V}$.

If $\ell_V$ is a constant function, the constant value is $\eta_V$, and Equation (\ref{EQ5}) implies that $\dim_k(V)=m\eta_V$; from this and Equation (\ref{EQ41}) we get that $r_{\mathcal V}=m[\eta_V-|J\setminus J_0|]$ is divisible by $m$. So Theorem \ref{T2}(1) holds.

\section{Proof of Theorem \ref{T2}(2)}\label{S7}

If $\mathcal V$ is a lift of $V$ to $R$ with $r_{\mathcal V}=0$, then from Theorem \ref{T2}(1) we get that $\jmath_V=0$ and that $\ell_V(a)=\eta_V$ for each $a\in\mathbb Z$, i.e., $\ell_V$ is a constant function. So the `only if' part holds.

Next we prove the `if' part. We can assume that $R=R_s$ and we use $R$ in what follows. If $m=1$, then $G=H$ and this case is well-known; for instance, see Theorem \ref{T3}(2). Thus we can assume that $m\ge 2$; so $p\ge m+1$ is odd.

We show by induction on $\eta_V\in\mathbb N\cup\{0\}$ that there exists a lift $\mathcal V$ of $V$ to $R$ with $r_{\mathcal V}=0$. If $\eta_V=0$, then $V=\{0\}$ and we can take $\mathcal V=\{0\}$; so the base of the induction holds. 

For the inductive step we can assume that $\eta_V\in\mathbb N$ is such that for each $k[G]$-module $W$ with the properties that the function $\ell_W$ is constant, $\jmath_W=0$, and $\eta_W\le\eta_V-1$, there exists a lift $\mathcal W$ of $W$ to $R$ with $r_{\mathcal W}=0$. 

As $\jmath_V=0$, we write $V=\oplus_{j=1}^{\imath_V} U_{a_j,b_j}$ with each $(a_j,b_j)\in\mathbb J_m\setminus (C_m\times\{p^s\})$. Let $(a_{\imath_V+1},b_{\imath_V+1}):=(a_1,b_1)$.

In this paragraph we assume that $V$ is not constant-indecomposable. So there exists a direct sum decomposition $V=V_1\oplus V_2$ with the function $\ell_{V_{\iota}}$ constant and $V_{\iota}\neq\{0\}$ for each $\iota\in\{1,2\}$. Moreover, we have $\jmath_{V_{\iota}}=0$ and $\eta_{V_{\iota}}=\frac{\dim_k(V_{\iota})}{m}<\frac{\dim_k(V)}{m}=\eta_V$. By the inductive assumption we get that there exists a lift $\mathcal V_{\iota}$ of $V_{\iota}$ to $R$ with $r_{\mathcal V_{\iota}}=0$. Then $\mathcal V:=\mathcal V_1\oplus\mathcal V_2$ is a lift of $V$ to $R$ with $r_{\mathcal V}=0$. 

So we can assume that $V$ is constant-indecomposable. We have $\imath_V\le m$ by Proposition \ref{P1}. If $\imath_V\ge 2$, then we can assume that $a_{i+1}=a_i-b_i\epsilon$ for each $i\in\llbracket1,\imath_V\rrbracket$, that the subset $\{a_i|i\in\llbracket1,\imath_V\rrbracket\}\subset C_m$ has $\imath_V$ elements, and that for every pair $(i_1,i_2)\in\llbracket1,\imath_V\rrbracket^2$ with $i_1<i_2<i_1+\imath_V$ we have $m\nmid \sum_{j=i_1}^{i_2-1} b_j$ by Proposition \ref{P1}.

Let $t\in\llbracket1,p^s-1\rrbracket\setminus p\mathbb Z$ be the unique element such that $\tau^t=\tau^{\epsilon}$; so the order of $[t]_{p^s}\in (\mathbb Z/p^s\mathbb Z)^{\times}$ is $m$. 

We consider three cases as follows, the first two being disjoint particular cases and the third one being the general case.

{\bf Case 1: $\imath_V=1$.} Let $\mathcal V=\mathcal H_{\nu}$ be an $R[H]$-module which is an end-hull lift of the $k[H]$-module $V$ with $r_{\mathcal V}=0$ by the proof of Theorem \ref{T3}(2), where $\nu=(\lambda_i)_{i\in\llbracket1,\dim_k(V)\rrbracket}$ is a $\dim_k(V)$-tuple of distinct elements of $\Xi_s^{\ast}$.  As $m\eta_V=\dim_k(V)$ by Equation (\ref{EQ5}), we have $m\eta_V\le p^s-1$. So we can assume that $[t]_{p^s}\cdot\{\lambda_i|i\in\llbracket1,\dim_k(V)\rrbracket\}=\{\lambda_i|i\in\llbracket1,\dim_k(V)\rrbracket\}$. Thus the $R[H]$-modules $\mathcal V$ and $\mathcal V^{(t)}$ are isomorphic by Equation (\ref{EQ31}). 

By a second induction on $l\in\mathbb N$ we show that there exists an element
$$\varsigma_l\in\Isom_{R[H]}(\mathcal V,\mathcal V^{(t)})$$
such that its reduction modulo $\upsilon$ is $\sigma$, it defines an element of $\Aut_{R}(\mathcal V)$ whose reduction modulo $\upsilon^l$ has order $m$ via the identification $V^{(t)}=V$ of $R$-modules, and, if $l\ge 2$, its reduction modulo $\upsilon^{l-1}$ is the reduction of $\varsigma_{l-1}$ modulo $\upsilon^{l-1}$. 

As the $R[H]$-modules $\mathcal V$ and $\mathcal V^{(t)}$ are isomorphic and $\mathcal V$ is an end-hull, it follows that the $R[H]$-linear map $\Isom_{R[H]}(\mathcal V,\mathcal V^{(t)})\rightarrow \Isom_{k[H]}(V,V^{(t)})$ is surjective. This implies that $\varsigma_1$ exists. So the base of the induction holds. 

For the inductive passage from $l-1\ge 1$ to $l$, we note that the order of the image of $\varsigma_{l-1}$ in $\Aut_R(\mathcal V/\upsilon^l\mathcal V)$ is either $m$, in which case we take $\varsigma_l:=\varsigma_{l-1}$, or $pm$, in which case we take $\varsigma_l:=(\varsigma_{l-1})^{1+mm'}$, where $m'\in\mathbb N$ is such that $p\mid 1+mm'$. This ends the second induction. 

As $R$ is complete, there exists a unique $\varsigma\in\Isom_{R[H]}(\mathcal V,\mathcal V^{(t)})$ such that its reduction modulo $\upsilon^l$ is $\varsigma_l$ for each $l\in\mathbb N$. We have identities $\varsigma^m=1_{\mathcal V}$ and $\varsigma\tau\varsigma^{-1}=\tau^t=\tau^{\epsilon}$ as they hold modulo $\upsilon^l$ for each $l\in\mathbb N$. Letting $\sigma\in\Sigma$ act on $\mathcal V$ as $\varsigma$ does, $\mathcal V$ gets the structure of an $R[G]$-module and hence we can take $\mathcal V$ to be this $R[G]$-module. Clearly, the $R[G]$-module $\mathcal V$ is a lift of $V$ to $R$ with $r_{\mathcal V}=0$ and it is cyclic as the $k[H]$-module $V$ is cyclic.

{\bf Case 2: $\imath_V\in\llbracket2,m\rrbracket$ and $\dim_k(V)\le p^s-1$.} We consider the cyclic $k[G]$-module $W:=U_{a_1,\dim_k(V)}$. By Case 1 we get that there exists a cyclic $R[G]$-module $\mathcal W$ which is a lift of $W$ to $R$ with $r_{\mathcal W}=0$ and it is associated to a $\dim_k(V)$-tuple $\nu=(\lambda_i)_{i\in\llbracket1,\dim_k(V)\rrbracket}$ of distinct elements of $\Xi_s^{\ast}$ with the property that $[t]_{p^s}\cdot\{\lambda_i|i\in\llbracket1,\dim_k(V)\rrbracket\}=\{\lambda_i|i\in\llbracket1,\dim_k(V)\rrbracket\}$. 

For $i\in\llbracket1,\dim_k(V)\rrbracket$, let $\mathfrak E_i$ be the eigenspace of the action of $\tau$ on $\mathcal W\otimes_R B$ that corresponds to the eigenvalue $\lambda_i$; we have $\dim_B(\mathfrak E_i)=1$. We define $\mathcal U_i:=\mathcal W\cap(\oplus_{j=1}^i \mathfrak E_j)$, the intersection being taken inside $\mathcal W\otimes_R B$. So $\mathcal U_i$ is an $R[H]$-submodule of $\mathcal W$ which, as an $R$-module, is a direct summand of rank $i$.

We consider an $R$-basis $\mathcal B=\{v_i|i\in\llbracket1,\dim_k(V)\rrbracket\}$ of $\mathcal W$ such that for each $i\in\llbracket1,\dim_k(V)\rrbracket$, $\{v_j|j\in\llbracket1,i\rrbracket\}$ is an $R$-basis of $\mathcal U_i$. With $\mathcal U_0:=\{0\}$ as the zero $R$-submodule of $\mathcal W$, we have the following property.

\medskip
{\bf ($\sharp$)} We have $\tau(v_i)-\lambda_iv_i\in\mathcal U_{i-1}$ for each $i\in\llbracket1,\dim_k(V)\rrbracket$.

\medskip
Let $\{\overline{v}_i|i\in\llbracket1,\dim_k(V)\rrbracket\}$ be the $k$-basis of $W$ which is the reduction of $\mathcal B$ modulo $\upsilon$. For each $i\in\llbracket1,\dim_k(W)\rrbracket$, let $U_i$ be the $k$-span of $\{\overline{v}_j|j\in\llbracket1,i\rrbracket\}$; it is the image of $\mathcal U_i$ in $W$. 

A third induction on $i\in\llbracket1,\dim_k(V)\rrbracket$ gives, based on property ($\sharp$), that $W_i$ is a $k[H]$-submodule of $W_i$. The unique composition series 
$$\{0\}=W_0\subsetneq W_1\subsetneq\cdots\subsetneq W_{\dim_k(V)}=W$$
of the $k[H]$-module $W$ is also the unique composition series of the $k[G]$-module $W$. So $W_i$ is a $k[G]$-submodule of $W$ for each $i\in\llbracket1,\dim_k(V)\rrbracket$. This implies that for each $i\in\llbracket1,\dim_k(V)\rrbracket$, $\mathcal W_i:=\upsilon\mathcal W+\mathcal U_i$ is an $R[G]$-submodule of $\mathcal W$. We note that $\mathcal W_{\dim_k(V)}=\mathcal W$. 

Let 
$$\mathcal V:=\sum_{i=1}^{\imath_V} \upsilon^{i-1} \mathcal W_{\sum_{j=1}^i b_j};$$ 
it is also an $R[G]$-submodule of $\mathcal W$. So 
$$\Bigl\{\upsilon^{r-1}v_j|r\in\llbracket1,\imath_V\rrbracket,\ j\in\Bigl\llbracket1+\sum_{l=1}^{r-1}b_l,\sum_{l=1}^r b_l\Bigr\rrbracket\Bigr\}$$ 
is an $R$-basis of $\mathcal V$. Based on this, property ($\sharp$) implies that the $k[G]$-module $\mathcal V/\upsilon\mathcal V$ is isomorphic to the direct sum $\oplus_{i=1}^{\imath_V} W_{\sum_{j=1}^i b_i} /W_{\sum_{j=1}^{i-1} b_i}$. Thus we have a chain of $k[G]$-linear isomorphisms 
$$\oplus_{i=1}^{\imath_V} W_{\sum_{j=1}^i b_i} /W_{\sum_{j=1}^{i-1} b_i}\cong \oplus_{i=1}^{\imath_V} U_{a_1,\sum_{j=1}^i b_i} /U_{a_1,\sum_{j=1}^{i-1} b_i}\cong \oplus_{i=1}^{\imath_V} U_{a_i,b_i}.$$ 
We conclude that we have a $k[G]$-linear isomorphism $\mathcal V/\upsilon\mathcal V\cong V$. We have $r_{\mathcal V}=0$ as $\mathcal V^H\subset \mathcal W^H=\{0\}$. So the $R[G]$-module $\mathcal V$ is a lift of $V$ to $R$ with $r_{\mathcal V}=0$.

{\bf Case 3: the general case.} We prove the existence of $\mathcal V$ by a fourth induction on $\imath_V\in\llbracket1,m\rrbracket$. The base of induction for $\imath_V=1$ holds by Case 1. For $\imath_V\ge 2$, the passage from $\imath_V-1$ to $\imath_V$ goes as follows. We consider two disjoint subcases as follows.

{\bf Subcase 3.1: there exists $i\in\llbracket1,\imath_V\rrbracket$ such that $b_i+b_{i+1}\le p^s-1$.} Up to a circular reindexing, we can assume that $b_1+b_2\le p^s-1$. Based on Case 2 we can assume that $\dim_k(V)\ge p^s$; so $\imath_V\ge 3$ and $m\nmid b_1+b_2$. 

Let $V_1:=U_{a_1,b_1+b_2}\bigoplus\oplus_{i=3}^{\imath_V} 
U_{a_i,b_i}$. We have $\imath_{V_1}=\imath_V-1$ and $\ell_{V_1}=\ell_V$ is constant. We check that $V_1$ is constant-indecomposable. We have inequalities $3\le \imath_V\le m$. As 
$$a_{12}':=a_1-(b_1+b_2)\epsilon=a_1'-b_2\epsilon=a_2-b_2\epsilon=a_3,$$
the set $\{a_{12}'\}\cup\{a_i'|i\in\llbracket3,\imath_V\rrbracket\}$ is equal to $\{a_1\}\cup\{a_i|i\in\llbracket3,\imath_V\rrbracket\}$ and thus has $\imath_{V_1}$ elements. For each $(i_1,i_2)\in\llbracket2,\imath_V\rrbracket^2$ with $i_1<i_2<i_1+\imath_{V_1}=i_1+\imath_V-1$ we have $m\nmid \sum_{j=i_1}^{i_2-1} b_j$ if $i_1\ge 3$ and $m\nmid b_1+b_2+\sum_{j=i_1+1}^{i_2-1} b_j$ if $i_1=2$ as $\natural_V$ holds. Thus $\natural_{V_1}$ holds and therefore $V_1$ is constant-indecomposable.

By induction, there exists a lift $\mathcal V_1$ of $V_1$ to $R$ with $r_{\mathcal V_1}=0$. Let $\mathcal V$ be the unique $R[G]$-submodule of $\mathcal V_1$ which contains $\upsilon\mathcal V_1$ and $\mathcal V/\upsilon\mathcal V_1$ is the $k[G]$-submodule $U_{a_1,b_1}$ of $V_1$. As in Case 2 we get that $\mathcal V/\upsilon V\cong V$. As $\mathcal V^H\subset\mathcal V_1^H=\{0\}$, $\mathcal V$ is a lift of $V$ to $R$ with $r_{\mathcal V}=0$.

{\bf Subcase 3.2: for each $i\in\llbracket1,\imath_V\rrbracket$ we have $b_i+b_{i+1}\ge p^s$.} We consider the $k[G]$-module $V^+:=\oplus_{i=1}^{\imath_V} U_{a_i,p^s-1}$. For each $i\in\llbracket1,\imath_V\rrbracket$ we have an injective $k[G]$-linear map $U_{a_i,b_i}\rightarrow U_{a_i,p^s-1}$; taking their direct sum we get an injective $k[G]$-linear map $V\rightarrow V^+$ which we view as an inclusion. We consider the $k[G]$-module $V':=V^+/V$. We have a direct sum decomposition $V'=\oplus_{i=1}^{\imath_V} U_{a_i',b_i'}$ where $b_i':=p^s-1-b_i$ and $a_i'=a_{i+1}$. For each $i\in\llbracket1,\imath_V\rrbracket$, as $m\nmid b_i$ and $b_i\le p^s-1$, we have $1\le b_i\le p^s-2$ and hence $1\le b_i'\le p^s-2$. Thus $\jmath_{V^+}=\jmath_{V'}=0$. Clearly, $\ell_{V^+}$ is constant of constant value $\frac{(p^s-1)\imath_V}{m}$. So $\ell_{V'}=\ell_{V^+}-\ell_V$ is also constant of constant value $\eta_{V'}=\eta_{V^+}-\eta_V$. Let $a_0':=a_{\imath_V}'$. We have 
$$a_i'':=a_i'-(p^s-1-b_i)\epsilon=a_{i+1}+b_i\epsilon=a_i=a_{i-1}'$$
for each $i\in\llbracket1,\imath_V\rrbracket$. Therefore the set $\{a_i''|i\in\llbracket1,\imath_V\rrbracket\}$ is equal to the set $\{a_i|i\in\llbracket1,\imath_V\rrbracket\}$ and hence it has $\imath_V$ elements.
For every pair $(i_1,i_2)\in\llbracket1,\imath_V\rrbracket^2$ with $i_1<i_2$ we have $\sum_{i=i_1}^{i_2} b_i'\equiv -\sum_{i=i_1}^{i_2} b_i\not\equiv 0\pmod{m}$. Thus $V'$ is also constant-indecomposable by Proposition \ref{P1}. 

As $b_1'+b_2'=2(p^s-1)-(b_1+b_2)\le 2(p^s-1)-p^s\le p^s-2$, from Subcase 3.1 we get that there exists a lift $\mathcal V'$ of $V'$ to $R$ such that $r_{\mathcal V'}=0$. 

By Case 1, for each $a\in C_m$ there exists an $R[G]$-module $\mathcal U_{a,p^s-1}$ which is a lift of $U_{a,p^s-1}$ to $R$ with $r_{\mathcal U_{a,p^s-1}}=0$. Such a lift is cyclic as it is so as an $R[H]$-module by Theorem \ref{T3}(1). One can check based on Theorem \ref{T3}(3) and Remark \ref{R1} that $\mathcal U_{a,p^s-1}$ is uniquely determined up to isomorphisms but this is not needed in what follows. Let 
$$V^{\perp}:=\oplus_{a\in C_m\setminus\{a_i|i\in\llbracket1,\imath_V\rrbracket\}} U_{a,p^s-1}$$ 
and 
$$V^{+,\perp}:=V^+\oplus V^{\perp}.$$ We have $V^{+,\perp}\cong V_{p^s-1}^{\circul}$ and thus $V^{+,\perp}$ has a lift $\mathcal V^{+,\perp}$ to $R$ isomorphic to $\mathcal V_{p^s-1}^{\circul}\cong R[G]\otimes_{R[H]} \mathcal U_{[0]_m,p^s-1}$ by Example \ref{EX4}. 

We consider the following $R[G]$-modules $\mathcal V^{\perp}:=\oplus_{a\in C_m\setminus\{a_i|i\in\llbracket1,\imath_V\rrbracket\}} \mathcal U_{a,p^s-1}$, $\mathcal V^{+,\perp}=\mathcal V^+\oplus \mathcal V^{\perp}$, and $\mathcal V^{',\perp}:=\mathcal V'\oplus\mathcal V^{\perp}$. We have a surjective $k[G]$-linear map 
$$\overline{f}:\mathcal V^{+,\perp}/\upsilon\mathcal V^{+,\perp}\rightarrow \mathcal V^{',\perp}/\upsilon\mathcal V^{',\perp}$$ 
which is the direct sum of $V^+\rightarrow V'$ and of the identity automorphism of $V^{\perp}=\mathcal V^{\perp}/\upsilon\mathcal V^{\perp}$; we have $V=\Ker(\overline{f})$. 

We show that there exists an $R[G]$-linear map $f:\mathcal V^{+,\perp}\rightarrow \mathcal V^{',\perp}$ whose reduction modulo $\upsilon$ is $\overline{f}$. As $\mathcal V^{+,\perp}\cong R[G]\otimes_{R[H]} \mathcal U_{[0]_m,p^s-1}$, it suffices to show that each $k[H]$-linear map $U_{[0]_m,p^s-1}\rightarrow V^{',\perp}$ lifts to an $R[H]$-linear map $\mathcal U_{[0]_m,p^s-1}\rightarrow\mathcal V^{',\perp}$, which holds by Lemma \ref{L6}(2). So $f$ exists. 

As $\overline{f}$ is surjective, $f$ is also surjective by Nakayama's Lemma. Therefore $\mathcal V:=\Ker(f)$ is a lift of $V$ to $R$. As $\mathcal V^H\subset (\mathcal V^{+,\perp})^H=\{0\}$, we have $r_{\mathcal V}=0$.

From the two subcases we get that $V$ has a lift $\mathcal V$ to $R$ with $r_{\mathcal V}=0$. This completes the inductive step and the fourth  induction on $\imath_V$. This also completes the first induction on $\eta_V$. Hence Theorem \ref{T2}(2) holds.

\section{Proof of Theorem \ref{T2}(3)}\label{S8}

Let $V_1$ be a direct summand of $V$ which is a supplement of $V_0$. We have $V=V_0\oplus V_1$, $\imath_{V_0}=\imath_V-\jmath_V$, and $\imath_{V_1}=\jmath_{V_1}=\jmath_V$. So $V_1$ is a direct sum of indecomposable projective $k[G]$-module and hence it is a projective $k[G]$-module. From Corollary \ref{C2} we get that there exists a projective lift $\mathcal V_1$ of $V_1$ to $R$.

Note that for the equivalence and the first implication below we do not require that $R$ contains $R_s$.

{\bf Proof that $(3.a)\Leftrightarrow (3.b)$.} If $V_0$ has a lift $\mathcal V_0$ to $R$, then $\mathcal V:=\mathcal V_0\oplus\mathcal V_1$ is a lift of $V$ to $R$. Thus $(3.b)\Rightarrow (3.a)$. 

Assume now that $V$ has a lift $\mathcal V$ to $R$. We have a commutative diagram of $R[G]$-modules
$$\xymatrix@R=10pt@C=21pt@L=2pt{
\mathcal V_1 \ar[r]^f\ar[d] &\mathcal V\ar[d]\\
V_1 \ar[r]^{\textup{incl}} & V\\}$$
whose vertical arrows are reductions modulo $\upsilon$, whose arrow $\textup{incl}:V_1\rightarrow V$ is the inclusion, and whose arrow $f$ exists as $\mathcal V_1$ is a projective $R[G]$-module. Clearly, $f$ is injective and its image is a direct summand of the $R$-module $\mathcal V$. So $\mathcal V_0:=\mathcal V/f(\mathcal V_1)$ is a lift of $V/V_1\cong V_0$ to $R$. Thus $(3.a)\Rightarrow (3.b)$.

We conclude that $(3.a)\Leftrightarrow (3.b)$. 

To prove that $(3.b)\Leftrightarrow (3.c)$, to ease the notation we can assume that $V=V_0$; hence $\jmath_V=0$.

{\bf Proof that $(3.b)\Rightarrow (3.c)$.} Let $\mathcal V$ be a lift of $V$ to $R$. Let $W$ be the image of $\mathcal V^H$ in $V$. Recall that $W$ is a $k$-vector subspace of $\soc_{k[G]}(V)$; so for each $a\in C_m$ we have $\ell_W(a)\le\ell_{\soc_{k[G]}(V)}(a)$. Let $U:=V/W$ and $\mathcal U:=\mathcal V/\mathcal V^H$. So $\mathcal U$ is a lift of $U$ to $R$ with $r_{\mathcal U}=0$. As $\jmath_V=0$, we also have $\jmath_U=0$. From Theorem \ref{T2}(2) we get that $\ell_U$ is constant. From this and the identity $\ell_V=\ell_W+\ell_U$ we get that $\eta_V=\eta_W+\eta_U$. For $a\in C_m$, we estimate
$$\ell_V(a)=\ell_W(a)+\ell_U(a)=\ell_W(a)+\eta_U=\ell_W(a)+\eta_V-\eta_W$$
$$\le\ell_{\soc_{k[G]}(V)}(a)+\eta_V-\eta_W\le \ell_{\soc_{k[G]}(V)}(a)+\eta_V.$$
Thus $(3.b)\Rightarrow (3.c)$.

{\bf Proof that $(3.c)\Rightarrow (3.b)$ when $R$ contains $R_s$.} From Equation (\ref{EQ11}) we get that for each $a\in C_m$ there exists a $k$-vector subspace $W_a$ of $\soc_{k[H]}(V)$ of dimension $\ell_V(a)-\eta_V$ on which $\sigma$ acts as the scalar multiplication by $\overline{\zeta}_m^a$. Let $W:=\sum_{a\in C_m} W_a$; we have $W=\oplus_{a\in C_m} W_a$. We consider a direct sum decomposition $W=\oplus_{i\in A} W_i$ into $k[G]$-modules which as $k$-vector spaces have dimension $1$. For $i\in A$, let $\mathcal W_i$ be an $R[G]$-module which is a lift of $W_i$ to $R$ with $r_{\mathcal W_i}=1$; so $\mathcal W_i=\mathcal W_i^H$ is uniquely determined up to isomorphism. Let $\mathcal W:=\oplus_{i\in A} \mathcal W_i$; we have $\mathcal W=\mathcal W^H$.

Let $U:=V/W$. As $\ell_W(a)=\ell_V(a)-\eta_V$ for each $a\in C_m$, we have $\ell_U(a)=\eta_V$ for each $a\in C_m$. Thus $\ell_U$ is constant of constant value $\eta_U=\eta_V$. As $\jmath_V=0$, we also have $\jmath_U=0$. We consider a direct sum decomposition $U=\oplus_{j\in J} U_j$ into constant-indecomposable $k[G]$-module. For each $j\in J$, let $\mathcal U_j$ be a lift of $U_j$ to $R$ with $r_{\mathcal U_j}=0$ by Theorem \ref{T2}(2). 

We show that there exists a short exact sequence of $R[G]$-modules
$$0\rightarrow\mathcal W\rightarrow\mathcal V\rightarrow\mathcal U\rightarrow 0$$
whose reduction modulo $\upsilon$ is the short exact sequence of $k[G]$-modules 
$$0\rightarrow W\rightarrow V\rightarrow U\rightarrow 0.$$
To show this it suffices to show that the image of the reduction modulo $\upsilon$ homomorphism $\Ext^1_{R[G]}(\mathcal U,\mathcal W)\rightarrow \Ext^1_{k[G]}(U,W)$ contains the class $\beta\in\Ext^1_{k[G]}(U,W)$ that defines $0\rightarrow W\rightarrow V\rightarrow U\rightarrow 0$. As we have direct sum decompositions $\Ext^1_{R[G]}(\mathcal U,\mathcal W)=\oplus_{(i,j)\in A\times J} \Ext^1_{R[G]}(\mathcal U_j,\mathcal W_i)$ and $\Ext^1_{k[G]}(U,W)=\oplus_{(i,j)\in A\times J} \Ext^1_{k[G]}(U_j,W_i)$, to show the existence of $\mathcal V$ we can assume that $|A\times J|=1$ and that  $\beta\neq 0$. Thus $U$ is constant-indecomposable, $\dim_k(W)=1$, and $\Ext^1_{k[G]}(U,W)\neq \{0\}$. From Proposition \ref{P2}(1) and (3) we get that $\dim_k\bigl(\Ext^1_{k[G]}(U,W)\bigr)=1$.

By applying $\Hom_{R[G]}(-,W)$ to the short exact sequence 
$$0\rightarrow\mathcal U\overset{\upsilon}\rightarrow\mathcal U\rightarrow U\rightarrow 0,$$ the resulting long exact $\Ext^{*}(-,W)$ complex contains the short exact sequence
\begin{equation}\label{EQ42}
0\rightarrow\Hom_{R[G]}(\mathcal U,W)\rightarrow\Ext_{R[G]}^1(U,W)\rightarrow\Ext_{R[G]}^1(\mathcal U,W)\rightarrow 0
\end{equation}
of $R$-modules annihilated by $\upsilon$ and hence of $k$-vector spaces. As $r_{\mathcal U}=0$ and $\mathcal W=\mathcal W^H$, we have $\Hom_{R[G]}(\mathcal U,\mathcal W)=\{0\}$. Based on this, by applying $\Hom_{R[G]}(\mathcal U,-)$ to the short exact sequence $0\rightarrow\mathcal W\overset{\upsilon}\rightarrow\mathcal W\rightarrow W\rightarrow 0$ we get an exact complex
\begin{equation}\label{EQ43}
0\rightarrow\Hom_{R[G]}(\mathcal U,W)\rightarrow\Ext_{R[G]}^1(\mathcal U,\mathcal W)\overset{\upsilon}\rightarrow\Ext_{R[G]}^1(\mathcal U,\mathcal W)\rightarrow\Ext_{R[G]}^1(\mathcal U,W).
\end{equation}

As $\Hom_{R[G]}(\mathcal U,W)=\Hom_{k[G]}(U,W)$ and $\Ext^1_{k[G]}(U,W)\neq \{0\}$, from Proposition \ref{P2}(1) and (3) we get that
$$\dim_k\bigl(\Hom_{R[G]}(\mathcal U,W)\bigr)=1.$$

With respect to the surjective ring homomorphism $R[G]\rightarrow k[G]$, the functorial $R$-linear map 
$$\Ext^1_{k[G]}(U,W)\rightarrow \Ext^1_{R[G]}(U,W)$$ 
(see \cite{CE}, Ch.\ VI, Sect.\ 4, p.\ 118) is injective as its composite with the $R$-linear map $\Ext_{R[G]}^1(U,W)\rightarrow\Ext_{R[G]}^1(\mathcal U,W)$ is the injective $R$-linear map $\Ext^1_{k[G]}(U,W)\rightarrow \Ext^1_{R[G]}(\mathcal U,W)$. 

The direct sum $k$-linear map 
$$\Hom_{R[G]}(\mathcal U,W)\oplus\Ext^1_{k[G]}(U,W)\rightarrow \Ext_{R[G]}^1(U,W)$$ 
is an isomorphism as one can check based on the fact that, as $R$ is a principal ideal domain, for a short exact sequence $0\rightarrow W\rightarrow \tilde V\rightarrow U\rightarrow 0$ of $R[G]$-modules, the $R$-module $\tilde V$ is either a $k$-vector space or the direct sum of a $k$-vector space with $R/\upsilon^2R$. From this, $\dim_k\bigl(\Ext^1_{k[G]}(U,W)\bigr)=1$, and Short Exact Sequence (\ref{EQ42}) we get that the $k$-linear map $$\Ext^1_{k[G]}(U,W)\rightarrow\Ext^1_{R[G]}(\mathcal U,W)$$ 
is an isomorphism of $1$-dimensional $k$-vector spaces. 

From this and $\dim_k\bigl(\Hom_{R[G]}(\mathcal U,W)\bigr)=1$, by reasons of length of torsion $R$-modules applied to the Exact Complex (\ref{EQ43}) we get that the $R$-linear 
$$\Ext_{R[G]}^1(\mathcal U,\mathcal W)\rightarrow\Ext_{R[G]}^1(\mathcal U,W)$$ 
is surjective. 

From the last two paragraphs we get that $\mathcal V$ exists. As $\mathcal V$ is a lift of $V$ to $R$, we get that $(3.c)\Rightarrow (3.b)$ if $R\subset R_s$.

From the last two implications we get that $(3.b)\Leftrightarrow (3.c)$. From this and $(3.a)\Leftrightarrow (3.b)$ we get that Theorem \ref{T2}(3) holds. Thus Theorem \ref{T2} holds.

\section{Upper and lower ramification jumps}\label{S9}

In this section we present some basic properties of upper and lower ramification jumps of cyclic extensions of $k((z))$ of Galois group $C_{p^s}$, to be called simply as $C_{p^s}$-extensions of $k((z))$, that are required in what follows. We simply call them upper and lower jumps. For the following classical result we refer to \cite{OP}, Lem.\ 3.5.

\begin{lemma}\label{L7}
For a strictly increasing sequence $u_1<\cdots<u_s$ of positive integers, the following statements are equivalent.

\medskip
{\bf (1)} There exists a cyclic $C_{p^s}$-extension of $k((z))$ whose upper jumps are $u_1<\cdots<u_s$.

\smallskip
{\bf (2)} The following conditions hold.

\medskip\noindent
{\bf (2.a)} We have $p\nmid u_1$.

\smallskip\noindent
{\bf (2.b)} For each $i\in\llbracket2,s\rrbracket$, either $u_i=pu_{i-1}$ or $u_i>pu_{i-1}$ and $p\nmid u_i$.
\end{lemma}

\begin{proposition}\label{P3}
We consider a cyclic $C_{p^s}$-extension of $k((z))$, its lower jumps $l_1<\cdots <l_s$, and its upper jumps $u_1<\cdots<u_s$. Then the following properties hold. 

\medskip
{\bf (1)} The lower and upper jumps determine uniquely each other by the recursive identities $l_1:=u_1$ and $l_i:=l_{i-1}+p^{i-1}(u_i-u_{i-1})$ for each $i\in\llbracket2,s\rrbracket$. 

\smallskip
{\bf (2)} For each $j\in\llbracket1,s\rrbracket $ we have $l_j=u_1+\sum_{i=1}^{j-1} p^i(u_{i+1}-u_i)$.

\smallskip
{\bf (3)} For each $j\in\llbracket1,s\rrbracket$ we have identities 
$$\sum_{i=1}^j \frac{l_i}{p^i}=\frac{\sum_{i=1}^j p^{i-1}u_i}{p^j}$$
and
$$p^{j-1}u_j=l_j+(p-1)\sum_{i=1}^{j-1}l_ip^{j-1-i}.$$

{\bf (4)} For $b\in\llbracket0,p^s-1\rrbracket$, we write
\[
b=\delta_{1,b}+\delta_{2,b}p+\cdots+\delta_{s,b}p^{s-1}
\]
with $\delta_{i,b}\in\llbracket0,p-1\rrbracket$ for each $i\in\llbracket1,s\rrbracket$ and define
\[
N_{s,b}:=\sum_{i=1}^s p^{s-i}\bigl[p-1+(p-1-\delta_{i,b})l_i\bigr].
\] Then for each $b\in \llbracket1,p^s-1\rrbracket$, if $j\in\llbracket0,s-1\rrbracket$ is the largest such that $p^j\mid b$, then we have identities 
\begin{equation*}
\begin{aligned}
N_{s,b-1}-N_{s,b}&=p^{s-j-1}\Big[l_{j+1}-(p-1)\sum_{i=1}^j p^{j+1-i}l_i\Big]\\
&=p^{s-j-1}[l_1+\sum_{i=1}^j p^i(u_{i+1}-pu_i)].
\end{aligned}
\end{equation*}
In particular, $N_{s,b-1}-N_{s,b}\ge p^{s-j-1}l_1\ge l_1>0$ and therefore we have $\frac{l_{j+1}}{p^{j+1}}-(p-1)\sum_{i=1}^j\frac{l_i}{p^i}\ge\frac{l_1}{p^{j+1}}>0$.
\end{proposition}

\begin{proof}
Part (1) is only Herbrand's Theorem (see \cite{S}, Ch.\ IV, Sect.\ 3, Lem.\ 5) for cyclic $p$-power extensions of $k((z))$ (see \cite{S}, Ch.\ IV, Sect.\ 3, Ex., p.\ 76 applied to $(n,i_0,\ldots,i_{n-1})=(s,u_1,u_2-u_1,\ldots,u_s-u_{s-1})$). 

As $l_j=l_1+\sum_{i=1}^{j-1} (l_{i+1}-l_i)$, part (2) follows from part (1).

We prove part (3) by induction on $j\in\llbracket1,s\rrbracket$. The base of the induction for $j=1$ holds as $l_1=u_1$ by part (1). 

For the inductive step, let $j\in\llbracket1,s-1\rrbracket$ be such that we have identities $\sum_{i=1}^j \frac{l_i}{p^i}=\frac{\sum_{i=1}^j p^iu_i}{p^{j+1}}=\frac{\sum_{i=1}^j p^{i-1}u_i}{p^{j}}$ and $p^{j-1}u_j=l_j+(p-1)\sum_{i=1}^{j-1}l_ip^{j-1-i}$. Based on this and part (2) we get that
$$\sum_{i=1}^{j+1} \frac{l_i}{p^i}=\Big(\sum_{i=1}^j \frac{l_i}{p^i}\Big)+\frac{l_{j+1}}{p^{j+1}}=\frac{\sum_{i=1}^j p^{i-1}u_i}{p^j}+\frac{u_1+\sum_{i=1}^j p^i(u_{i+1}-u_i)}{p^{j+1}}$$
is equal to $\frac{\sum_{i=1}^{j+1} p^{i-1}u_i}{p^{j+1}}$. Hence 
$$p^ju_{j+1}=\Big(\sum_{i=1}^{j+1} p^{i-1}u_i\Big)-\Big(\sum_{i=1}^j p^{i-1}u_i\Big)=p^{j+1}\Big(\sum_{i=1}^{j+1} \frac{l_i}{p^i}\Big)-p^j\Big(\sum_{i=1}^{j} \frac{l_i}{p^i}\Big)$$ is equal to $l_{j+1}+(p-1)\sum_{i=1}^{j}l_ip^{j-i}$. This ends the inductive step and the induction. So part (3) holds.

For part (4), we first assume that $j=0$; so $p\nmid b$. We have $\delta_{1,b-1}=\delta_{1,b}-1$ and $\delta_{i,b-1}=\delta_{i,b}$ for each $i\in\llbracket2,s\rrbracket$. Hence $N_{s,b-1}-N_{s,b}=p^{s-1}l_1\ge l_1>0$ and so part (4) holds if $j=0$.

Suppose now that $j\ge 1$; so $p\mid b$. We have $\delta_{1,b}=\cdots=\delta_{j,b}=0$ and $\delta_{j+1,b}\in\llbracket1,p-1\rrbracket$. Hence $\delta_{1,b-1}=\cdots=\delta_{j,b-1}=p-1$, $\delta_{j+1,b-1}=\delta_{j+1,b}-1$, and $\delta_{i,b-1}=\delta_{i,b}$ for each $i\in\llbracket j+2,s\rrbracket$. Thus 
\begin{equation*}
\begin{aligned}
N_{s,b-1}-N_{s,b}&=p^{s-j-1}l_{j+1}-(p-1)\sum_{i=1}^j p^{s-i}l_i\\
&=p^{s-j-1}\Big[l_{j+1}-(p-1)\sum_{i=1}^j p^{j+1-i}l_i\Big]
\end{aligned}
\end{equation*}
and hence the first identity of part (4) holds.
Based on parts (2) and (3) we compute
\[
\begin{split}
l_{j+1}-(p-1)\sum_{i=1}^j p^{j+1-i}l_i
&=l_1+\sum_{i=1}^j p^i(u_{i+1}-u_i)
-(p-1)\sum_{i=1}^j p^iu_i\\
&=l_1+\sum_{i=1}^j p^i(u_{i+1}-pu_i).
\end{split}
\]
We conclude that 
$N_{s,b-1}-N_{s,b}=p^{s-j-1}[l_1+\sum_{i=1}^j p^i(u_{i+1}-pu_i)]$, i.e., the second identity of part (4) holds. By Lemma \ref{L7}, we have $u_{i+1}\ge pu_i$ for each $i\in\llbracket1,j\rrbracket$. So part (4) also holds if $j\ge 1$.
\end{proof}

\begin{corollary}\label{C3}
Let $l_0:=0$. For a strictly increasing sequence $l_1<\cdots<l_s$ of positive integers, the following statements are equivalent.

\medskip
{\bf (1)} There exists a cyclic $C_{p^s}$-extension of $k((z))$ whose lower jumps are $l_1<\cdots<l_s$.

\smallskip
{\bf (2)} The following conditions hold.

\medskip\noindent
{\bf (2.a)} We have $p\nmid l_1$.

\smallskip\noindent
{\bf (2.b)} For each $j\in\llbracket1,s-1\rrbracket$ we have $p^j\mid l_{j+1}-l_j$ and
$$l_{j+1}-l_j\ge (p-1)\sum_{i=0}^{j-1} p^{j-i}(l_{i+1}-l_i).$$

\noindent
{\bf (2.c)} For each $j\in\llbracket1,s-1\rrbracket$ such that $l_{j+1}-l_j> (p-1)\sum_{i=0}^{j-1} p^{j-i}(l_{i+1}-l_i)$, the positive integer $l_1+\sum_{i=1}^{j}\frac{l_{i+1}-l_i}{p^i}$ is not divisible by $p$.
\end{corollary}

\begin{proof}
We prove that $(1)\Rightarrow (2)$. We consider the upper jumps $u_1<\cdots<u_s$ of the cyclic $C_{p^s}$-extension of $k((z))$ whose lower jumps are $l_1<\cdots<l_s$. From Lemma \ref{L7}(2.a) and the first part of Proposition \ref{P3}(1) we get that part (2.a) holds. From the second part of Proposition \ref{P3}(1) we get that the divisibility of part (2.b) holds. The argument that proved Proposition \ref{P3}(2) also gives that for each $j\in\llbracket1,s-1\rrbracket$ we have $u_{j+1}=l_1+\sum_{i=1}^j \frac{l_{i+1}-l_i}{p^i}$; so
\begin{equation}\label{EQ46}
\begin{split}
u_{j+1}-pu_j
&=l_1+\sum_{i=1}^j \frac{l_{i+1}-l_i}{p^i}
-p\biggl(l_1+\sum_{i=1}^{j-1} \frac{l_{i+1}-l_i}{p^i}\biggr)\\
&=\frac{l_{j+1}-l_j}{p^j}-(p-1)\sum_{i=0}^{j-1} \frac{l_{i+1}-l_i}{p^i}.
\end{split}
\end{equation}
and hence, as $u_{j+1}-pu_j\ge 0$ by Lemma \ref{L7}, we get that the inequality of part (2.b) holds. So part (2.b) holds.

If $j\in\llbracket1,s-1\rrbracket$ is such that $l_{j+1}-l_j> (p-1)\sum_{i=0}^{j-1} p^{j-i}(l_{i+1}-l_i)$, then from the last paragraph we get that $u_{j+1}-pu_j>0$ and therefore $p\nmid u_{j+1}=l_1+\sum_{i=1}^{j}\frac{l_{i+1}-l_i}{p^i}$ by Lemma \ref{L7}, so part (2.c) also holds. So $(1)\Rightarrow (2)$.

We prove that $(2)\Rightarrow (1)$. We define $u_1:=l_1$ and $u_{j+1}:= l_1+\sum_{i=1}^{j}\frac{l_{i+1}-l_i}{p^i}$ for each $j\in\llbracket1,s-1\rrbracket$. From parts (2.a) and (2.b) and the strict inequalities $l_1<\cdots<l_s$ we get that we have strict inequalities $u_1<\cdots<u_s$ between positive integers. As $u_1=l_1$, we have $p\nmid u_1$ by part (2.a). From Equation (\ref{EQ46}) and part (2.b) we get that for each $j\in\llbracket1,s-1\rrbracket$ we have $u_{j+1}\ge pu_j$. If there exists  $j\in\llbracket1,s-1\rrbracket$ such that $u_{j+1}> pu_j$, then $p\nmid u_{j+1}$ by part (2.c). Hence Lemma \ref{L7}(2.a) and (2.b) hold and thus there exists a cyclic $C_{p^s}$-extension of $k((z))$ whose upper jumps are $u_1<\cdots<u_s$ by Lemma \ref{L7}. As $u_1=l_1$ and $u_{j+1}-u_j=\frac{l_{j+1}-l_j}{p^j}$ for each $j\in\llbracket1,s-1\rrbracket$, from Proposition \ref{P3}(1) we get that the lower jumps of this cyclic $C_{p^s}$-extension of $k((z))$ are $l_1<\cdots<l_s$. Hence $(2)\Rightarrow (1)$.
\end{proof}

\section{Actions of $G$ on curves over $k$. Part I: general principles}\label{S10}

After reviewing the notation for curves that are used until the end, we include three general results on faithful actions of $G$ on them.

{\bf Notation.} For each connected smooth projective curve $X$ over $k$, let $\textup{g}_X\in\mathbb N\cup\{0\}$ be its genus, $V_X:=H^0(X,\Omega_X)$, $W_X:=H^1(X,\mathcal O_X)$, and $M_X:=H^1_{dR}(X)$. We have $\dim_k(V_X)=\dim_k(W_X)=\textup{g}_X$. 

Let $\Frob_0$ be the Frobenius automorphism of $R_0$. Let $\mathcal M_X:=H^1_{\crys}(X/R_0)$ be the first crystalline cohomology group of $X$; it is a free $R_0$-module of rank $2\textup{g}_X$ equipped with an injective $\Frob_0$-linear endomorphism $\varphi:\mathcal M_X\rightarrow\mathcal M_X$ such that $p\mathcal M_X\subset\varphi(\mathcal M_X)\subset\mathcal M_X$ and $\mathcal M_X/\varphi(\mathcal M_X)\cong k^{\textup{g}_X}$. So for each $(\alpha,v)\in R_0\times\mathcal M_X$ we have $\varphi(\alpha x)=\Frob_0(\alpha)\varphi(v)$. By functoriality, $\mathcal M_X$ is an $R_0[G]$-module.

We have a functorial identification $M_X=\mathcal M_X/p\mathcal M_X$, hence an identification of $R_0[G]$-modules that explains the notation, and a functorial Hodge--de Rham short exact sequence (see \cite{Be}, Ch.\ V, Cor.\ 2.3.7) of $k$-vector spaces
\begin{equation}\label{EQ47}
0\rightarrow V_X\rightarrow M_X\rightarrow W_X\rightarrow 0.
\end{equation}

Until the end we fix a faithful action $\rho:G\rightarrow\Aut(X)$. So Short Exact Sequence (\ref{EQ47}) is of $k[G]$-modules. We consider the quotient curve $Y:=X/H$ and the quotient morphism $\pi:X\rightarrow Y$.

\begin{proposition}\label{P4}
The following properties hold.

\medskip
{\bf (1)} The pullback $R_0$-linear map $$\pi^*:\mathcal M_Y\rightarrow\mathcal M_X$$ is injective, has an image contained in $\mathcal M_X^H$, and we have $r_{\mathcal M_X}=2\textup{g}_Y$.

\smallskip
{\bf (2)} We have $Y\cong\mathbb P^1_k$ if and only if $\mathcal M_X^H=\{0\}$.

\smallskip
{\bf (3)} If $Y\cong\mathbb P^1_k$, then $\jmath_{M_X}=\jmath_{V_X}=0$.

\smallskip
{\bf (4)} Suppose that $\chi$ is injective and $Y\cong\mathbb P^1_k$. Then $m\mid 2\textup{g}_X$ and for each $a\in C_m$ we have $\ell_{V_X}(a)+\ell_{V_X}(-a)=\frac{2\textup{g}_X}{m}$. In particular, if $m=2$, then $\textup{g}_X$ is even and $\ell_{V_X}$ is constant of constant value $\frac{\textup{g}_X}{2}$.
\end{proposition}

\begin{proof}
Part (1) is well-known to specialists. As we could not find a good reference for it, among many approaches towards proving it, we include here a short and standard one due to Ofer Gabber. The morphism $\pi$ is finite flat of degree $p^s=|H|$. The image of $\pi^*:\mathcal M_Y\rightarrow\mathcal M_X$ is contained in $\mathcal M_X^H$. Let 
$$\pi_*:\mathcal M_X\rightarrow\mathcal M_Y$$ 
be the direct image $R$-linear map (see \cite{Be}, Ch.\ VII, Subsect.\ 2.2, Equations (2.2.3) and (2.2.5)). For $v\in\mathcal M_Y$ we have the projection formula $\pi_*\bigl(\pi^*(v)\bigr)=p^sv$ (cup product with $p^s$, see \cite{Be}, Ch.\ VII, Prop.\ 2.2.3) and for $w\in\mathcal M_X$ we have $\pi^*\bigl(\pi_*(w)\bigr)=\sum_{i=0}^{p^s-1} (\tau^i)^*(w)$ as one can see by a direct computation. Thus $\pi_*\circ\pi^*\in\End_{R_0}(\mathcal M_Y)$ is the multiplication by $p^s$ and hence it is injective as $\mathcal M_Y$ is torsion free. So $\pi^*$ is also injective.

For $w\in\mathcal M_X^H$ we have $\pi^*\bigl(\pi_*(w)\bigr)=p^sw$. As the $R_0$-module $\mathcal M_X$ is also torsion free, it follows that the $B_0$-linear map $\mathcal M_X^H\otimes_{R_0} B_0\rightarrow\mathcal M_Y\otimes_{R_0} B_0$, which is the restriction of $\pi_*\otimes 1_{B_0}$, has an inverse induced by $\pi^*\otimes 1_{B_0}$ and therefore it is a bijection. So $r_{\mathcal M_X}$ is the rank $2\textup{g}_Y$ of $\mathcal M_Y$ and part (1) holds.

Part (2) follows from part (1) as $\textup{g}_Y=0$ if and only if $Y\cong\mathbb P^1_k$.

Fo part (3), we note that $r_{\mathcal M_X}=0$ by part (2), and hence $\jmath_{M_X}=0$ by Theorem \ref{T2}(1). As $\jmath_{M_X}=0$ and $\jmath_{V_X}$ is a $k[G]$-submodule of $\jmath_{M_X}$, we get that $\jmath_{V_X}=0$. So part (3) holds.

For part (4), $\ell_{M_X}$ is a constant function of constant value $\eta_{M_X}$ by parts (2) and (3) and Theorem \ref{T2}(2). Thus $2\textup{g}_X=\dim_k(M_X)=m\eta_{M_X}$. Hence $m\mid 2\textup{g}_X$. 

The de Rham (or crystalline) cohomology gives that the $k[G]$-modules $V_X$ and $W_X$ are dual to each other; so $\ell_{W_X}(a)=\ell_{V_X}(-a)$ for each $a\in C_m$. As $M_X$ and $V_X\oplus W_X$ are isomorphic as $k[\Sigma]$ modules, we have an identity $\ell_{M_X}(a)=\ell_{V_X}(a)+\ell_{W_X}(a)$ 
for each $a\in C_m$. So for each $a\in C_m$ we compute
$$\ell_{V_X}(a)+\ell_{V_X}(-a)=\ell_{V_X}(a)+\ell_{W_X}(a)=\ell_{M_X}(a)=\eta_{M_X}=\frac{2\textup{g}_X}{m}.$$
Thus part (4) holds.\end{proof}

\smallskip
Proposition \ref{P4}(1) and Equation (\ref{EQ20}) give the following consequence.

\begin{corollary}\label{C4}
For each $j\in\llbracket0,s\rrbracket$ let $H_{s-j}$ be the subgroup of $H$ of order $p^{s-j}$ and let $Y_j:=X/H_{s-j}$; so $X=Y_s$ and $Y=Y_0$. We consider the $B_0[H]$-module $\mathcal C_X:=\mathcal M_X\otimes_{R_0} B_0$. Then for $i\in\llbracket0,s\rrbracket$ we have
$$\varkappa_{\mathcal C_X}(i)=\begin{cases} 2\textup{g}_{Y_i}\quad\quad\quad\quad\quad {\rm if}\quad i=0\\
\frac{2\textup{g}_{Y_i}-2\textup{g}_{Y_{i-1}}}{p^i-p^{i-1}}\quad\,\;\;\;\, {\rm if}\quad i\ge 1.\end{cases}$$
\end{corollary}

\begin{proof}
The case $i=0$ follows directly from Proposition \ref{P4}(1). If $i\ge 1$, then $\dim_{B_0}(\mathcal C_X^{\tau^{p^i}})=\dim_{B_0}\bigl(H^1_{\crys}(Y_i/R_0)\otimes_{R_0} B_0\bigr)=2\textup{g}_{Y_i}$ based on Proposition \ref{P4}(1) applied to the quotient morphism $X\rightarrow Y_i=X/H_{s-i}$ as $H_{s-i}$ is generated by $\tau^{p^i}$. Based on the last two sentences we also get that $\dim_{B_0}(\mathcal C_X^{\tau^{p^{i-1}}})=\dim_{B_0}\bigl(H^1_{\crys}(Y_{i-1}/R_0)\otimes_{R_0} B_0\bigr)=2\textup{g}_{Y_{i-1}}$. Based on the last two identities between dimensions of $B_0$-vector spaces and Equation (\ref{EQ22}) we get that the corollary holds.
\end{proof}

\smallskip
We have the following application of Lemma \ref{L4}.

\begin{proposition}\label{P5}
Suppose that $p$ is odd, $\chi$ is injective, $m$ is even, and there exists a smooth projective curve $\mathcal X$ over $\Spec R$ such that its reduction modulo $\upsilon$ is $X$ and $\rho$ lifts to a homomorphism $\varrho:G\rightarrow\Aut(\mathcal X)$. Then the following properties hold.

\medskip
{\bf (1)} The $R[H]$-modules $H^0(\mathcal X,\Omega_{\mathcal X})$ and $H^1(\mathcal X,\mathcal O_{\mathcal X})$ are rational lifts of the $k[H]$-modules $V_X$ and $W_X$ (respectively).

\smallskip
{\bf (2)} The $B[H]$-modules $H^0(\mathcal X,\Omega_{\mathcal X})\otimes_R B$ and $H^1(\mathcal X,\mathcal O_{\mathcal X})\otimes_R B$ are isomorphic to half of the $B[H]$-module $\mathcal C_X$ of Corollary \ref{C4}, i.e., are isomorphic to $\mathcal R_X:=\textup{g}_{Y_0}\mathcal B_0\bigoplus\oplus_{i=1}^s \frac{\textup{g}_{Y_i}-\textup{g}_{Y_{i-1}}}{p^i-p^{i-1}}\mathcal B_i$.
\end{proposition}

\begin{proof}
We note that $\mathcal V_{\mathcal X}:=H^0(\mathcal X,\Omega_{\mathcal X})$ is a lift of $V_X$ to $R$, $\mathcal W_{\mathcal X}:=H^1(\mathcal X,\mathcal O_{\mathcal X})$ is a lift of $W_X$ to $R$, $\mathcal M_{\mathcal X}:=H^1_{dR}(\mathcal X)$ is a lift of $M_X$ to $R$, and $\mathcal W_{\mathcal X}$ and $\mathcal V_{\mathcal X}$ are dual to each other. The functorial Hodge--de Rham short exact sequence
\begin{equation}\label{EQ48}
0\rightarrow\mathcal V_{\mathcal X}\rightarrow\mathcal M_{\mathcal X}\rightarrow\mathcal W_{\mathcal X}\rightarrow 0
\end{equation}
of $R$-modules is also of $R[G]$-modules and lifts Short Exact Sequence (\ref{EQ47}). The $p$-adic Hodge theory provides a functorial isomorphism
\begin{equation}\label{EQ49}
\mathcal M_{\mathcal X}\otimes_R B\rightarrow\mathcal M_X\otimes_{R_0} B
\end{equation}
of $B$-vector spaces and thus also of $B[G]$-modules. So $\mathcal M_{\mathcal X}\otimes_R B$ is a rational $B[H]$-module by Lemma \ref{L3}. From this and Lemma \ref{L4} we get that part (1) holds and that $B[H]$-modules $H^0(\mathcal X,\Omega_{\mathcal X})\otimes_R B$ and $H^1(\mathcal X,\mathcal O_{\mathcal X})\otimes_R B$ are isomorphic. From this and the last paragraph of the proof of Lemma \ref{L4} we get that part (2) holds.\end{proof}

\section{Actions of G on curves, Part II: numerical results}\label{S11}

With the notation of Section \ref{S10}, we recall a definition (see \cite{P}, Sect.\ 4B).

\begin{definition}\label{D5}
The morphism $\pi:X\rightarrow X/G$ is an \emph{HKG cover} if the following conditions hold.

\medskip
{\bf (1)} We have $X/G\cong \mathbb P_k^1$ and the number of branch points is $\min(2,m)$.

\smallskip
{\bf (2)} There exists one branch point over which $\pi$ is totally and wildly ramified with inertia group $G$. 

\smallskip
{\bf (3)} If $m>1$, the second branch point is tamely ramified and has $\Sigma$ as an inertia group.
\end{definition}

\begin{lemma}\label{L8} Suppose that $X\rightarrow X/G$ is an HKG cover. Then $Y\cong\mathbb P^1_k$.
\end{lemma}

\begin{proof} Recall that $Y=X/H$. As $X/G\cong\mathbb P^1_k$ by Definition \ref{D5}(1), we can assume that $H\neq G$, i.e., $m>1$. The morphism $X\rightarrow X/G$ factors as
$$X\rightarrow Y=X/H\rightarrow X/G.$$
The Galois group of $Y\rightarrow X/G$ is $G/H\cong\Sigma$. At the wild branch point $\infty\in (X/G)(k)$ the inertia group of $X\rightarrow X/G$ is $G$ and at the second branch point in $(X/G)(k)$ the inertia group of $X\rightarrow X/G$ is $\Sigma$. As both the groups $G$ and $\Sigma$ surject onto $G/H$, it follows that $Y\rightarrow X/G\cong\mathbb P_k^1$ is a tame cyclic cover of degree $m$ totally ramified at both branch points in $(X/G)(k)$. So
$$2\textup{g}_Y-2=-2m+2(m-1)=-2$$
by the Riemann--Hurwitz formula applied to it. So $\textup{g}_Y=0$, i.e., $Y\cong \mathbb P_k^1$.
\end{proof}

\smallskip
The next theorem (proved in \cite{K}, Main Thm.\ 1.4.1) classifies the HKG covers $X\rightarrow X/G$.

\begin{theorem}[{\bf Harbater--Katz--Gabber}]
\label{T4}
Let $\phi:G\rightarrow \Aut_k(k[[y]])$ be a faithful local action. Then there exists a pair $(X,\rho_{\phi})$ uniquely determined up to isomorphisms, with $X$ a  connected smooth projective curve over $k$ and a faithful global action $\rho_{\phi}:G\rightarrow\Aut(X)$, such that the quotient morphism $X\rightarrow X/G$ is an HKG cover whose completion at the unique wildly ramified $k$-valued point of $X$ is isomorphic to $\phi$.
\end{theorem}

Whenever $X\rightarrow X/G$ is an HKG cover, the lower and upper jumps of $H$ refer to its action on the completed local ring at the unique wildly ramified point of $X$ and we call $H$ as the wild subgroup of $G$.

Until the end of this section we assume that $X\rightarrow X/G$ is an HKG cover and we explain how the results of \cite{BCK} were applied in \cite{KT2} to compute in this case the multiplicity function of the $k[G]$-module
$V_X$. The main inputs are the following ones.
\begin{enumerate}
\item The upper jumps of the wild cyclic subgroup $H$ at the wild branch point $\infty\in Y(k)$.
\item The action of $\Sigma$ on $Y$, or equivalently the character $\chi:\Sigma\rightarrow \mathbb F_p^\times$.
\end{enumerate}
The output, accomplished in three steps, is the multiplicity function 
$$\mu_X:=\mu_{V_X}:\mathbb J_{m,p^s}\rightarrow\mathbb N\cup\{0\}.$$

{\bf Step 1: computing the divisors that determine the $\tau$-filtration.} For $b\in\llbracket0,p^s\rrbracket$, let
$$V_X^{(b)}:=\Ker((\tau-1)^b:V_X\rightarrow V_X).$$
So $\big(V_X^{(b)}\big)_{b=0}^{p^s}$ is the increasing, separated, and exhaustive $\tau$-filtration of $V_X$
by $k[G]$-submodules. We note that $V_X^{(1)}=\soc_{k[G]}(V_X)$. 

As for each $b\in\llbracket0,p^s-1\rrbracket$, $\tau$ acts trivially on the quotient factor $V_X^{(b+1)}/V_X^{(b)}$, this quotient factor is a $k[\Sigma]$-module. 

By \cite{BCK}, Prop.\ 4.1 and Lem.\ 4.2, for each $b\in\llbracket0,p^s-1\rrbracket$ there exists an effective divisor $D_b$ on $Y$ such that, with the notation of Display (\ref{EQ9}), we have a $k[\Sigma]$-linear isomorphism
\begin{equation}\label{EQ50}
V_X^{(b+1)}/V_X^{(b)}\cong H^0\bigl(Y,\Omega_Y(D_b)\bigr)(b\epsilon).
\end{equation}
As $Y\cong\mathbb P_k^1$ and $X\rightarrow X/G$ is an HKG cover, $D_b$ is supported only at the wild branch $\infty\in Y(k)$ (see \cite{BCK}, proof of Prop.\ 4.1, bottom of p.\ 18); so there exists a unique $d_b\in\mathbb N\cup\{0\}$ such that $D_b=d_b\infty$.

Fixing an identification $Y=\mathbb P^1_k$, we write $\mathbb A^1_k=\mathbb P^1_k\setminus\{\infty\}=\Spec\, k[x]$. 

The $d_b$s are computed in three substeps via the upper and lower jumps of $H$ as follows.

\begin{proposition}\label{P6}
Suppose that $X\rightarrow X/G$ is an HKG cover. Let
$$l_1<\cdots<l_s$$
be the lower jumps of the wild cyclic subgroup $H=C_{p^s}$ at the wild branch point. Then the following properties hold.

\medskip
\smallskip
{\bf (1) (Bleher--Chinburg--Kontogeorgis)} For each $b\in\llbracket0,p^s-1\rrbracket$, with the notation of Proposition \ref{P3}(4), we have
\begin{equation}\label{EQ51}
d_b=\left\lfloor \frac{N_{s,b}}{p^s}\right\rfloor=\left\lfloor \frac{\sum_{i=1}^s p^{s-i}[p-1+(p-1-\delta_{i,b})l_i]}{p^s}\right\rfloor.
\end{equation}

{\bf (2)} The set $\{x^idx|i\in\llbracket0,d_b-2\rrbracket\}$ is a $k$-basis of $H^0\bigl(\mathbb P_k^1,\Omega_{\mathbb P_k^1}(D_b)\bigr)$. In particular, if $d_b\le 1$, then this $k$-vector space is $\{0\}$.

\smallskip
{\bf (3)} For $b\in\llbracket0,p^s-1\rrbracket$, we have $d_b=0$ if and only if $b=p^s-1$.

\smallskip
{\bf (4)} The sequence $(d_b)_{b\in\llbracket0,p^s\rrbracket}$ is non-increasing.
\end{proposition}

\begin{proof}
Part (1) holds by \cite{BCK}, Eqs.\ (4.18) and (4.15) (cf.\ paragraph above \cite{KT2}, Rmk.\ 16).

As $D_b=d_b\infty$ and $\Omega_{\mathbb P_k^1}\cong\mathcal O_{\mathbb P_k^1}(-2\infty)$, we have
$$\Omega_{\mathbb P_k^1}(D_b)\cong \mathcal O_{\mathbb P_k^1}\bigl((d_b-2)\infty\bigr).$$
Thus $H^0\bigl(\Omega_{\mathbb P_k^1}(D_b)\bigr)=0$ if $d_b\le 1$. If $d_b\ge 2$, then on $\mathbb A_k^1=\mathbb P_k^1\setminus\{\infty\}$ with coordinate $x$, the global sections of $\Omega_{\mathbb P_k^1}(D_b)$ are the forms $f(x)dx$ with $f\in k[x]$ such that $\deg(f)\le d_b-2$. So part (2) holds.

For part (3), as $\delta_{i,p^s-1}=p-1$ for
each $i\in\llbracket1,s\rrbracket$, Equation (\ref{EQ51}) gives
\[
d_{p^s-1}=\left\lfloor\frac{1}{p^s}\sum_{i=1}^s p^{s-i}(p-1)\right\rfloor
=\left\lfloor\frac{p^s-1}{p^s}\right\rfloor=0.
\]

Assume now that $b\in\llbracket1,p^s-2\rrbracket$. Then there exists a smallest $j\in\llbracket1,s\rrbracket$ such that $\delta_{i,b}\in\llbracket1,p-2\rrbracket$. As $l_j\ge 1$, we have
$$\sum_{i=1}^s p^{s-i}[p-1+(p-1-\delta_{i,b})l_i]\ge \Big[\sum_{i=1}^{j-1} p^{s-i}(p-1)\Big]+p^{s-j}(p-1+1)=p^s$$
and hence $d_b\ge 1$. So part (3) holds.

As $d_{p^s}=0$ by definition, $d_{p^s-1}=0$ by part (3), and $d_{b-1}-d_b\ge 0$ for each $b\in\llbracket1,p^s-1\rrbracket$ by part (1) and Proposition \ref{P3}(4), we get that part (4) holds.\end{proof}

\smallskip
{\bf Step 2: computing the quotient factors of the $\tau$ filtration of $V_X$.} For $(a,b)\in C_m\times\llbracket0,p^s\rrbracket$, following the notation of \cite{KT2}, Rmk.\ 16 let $\mu_{b,a}\in\mathbb N\cup\{0\}$ be such that it is $0$ if $b=p^s$ and we have a $k[\Sigma]$-linear isomorphism
$$H^0\bigl(\mathbb P_k^1,\Omega_{\mathbb P_k^1}(D_b)\bigr)\cong\oplus_{a\in C_m} \mu_{b,a}T_a$$
if $b\in\llbracket0,p^s-1\rrbracket$. Hence, for $b\in\llbracket0,p^s-1\rrbracket$,  Equation (\ref{EQ50}) becomes
\begin{equation}\label{EQ52}
V_X^{(b+1)}/V_X^{(b)}\cong \oplus_{a\in C_m} \mu_{b,a}T_{a-b\epsilon}.
\end{equation}

As $\sigma(\infty)=\infty$, $\sigma$ induces an automorphism of $k[x]$ of order $m$ prime to $p$. Thus $\sigma$ fixes one $k$-valued point $\alpha$ of $\mathbb A^1_k$ (it is unique if $m>1$). So, replacing $x$ by $x-\alpha$ and, from now on, taking the fixed primitive root $\overline{\zeta}_m$ to be the scalar by which $\sigma$ acts on $x$, we can assume that $\sigma(x)=\overline{\zeta}_m x$. Then $\sigma(x^i dx)=\overline{\zeta}_m^{i+1}x^i dx$ for each $i\in\mathbb N\cup\{0\}$. If $d_b\le 1$, then $\mu_{b,a}=0$ for all $a$. If $d_b\ge 2$, using the $k$-basis $\{x^i dx\mid i\in\llbracket0,d_b-2\rrbracket\}$ of eigenvectors of the action of $\sigma$ on $H^0\bigl(\mathbb P_k^1,\Omega_{\mathbb P_k^1}(D_b)\bigr)$, we get that $\mu_{b,a}$ is the number of exponents $i\in\llbracket0,d_b-2\rrbracket$ such that $i+1\equiv a\pmod m$. Equivalently, for $d_b\ge 2$ we have
\begin{equation}\label{EQ53}
\mu_{b,a}=\left\lfloor\frac{d_b-1}{m}\right\rfloor+\mu_b(a),
\end{equation}
where $\mu_b(a)\in\{0,1\}$ is $1$ if and only if there exists a positive integer $\break j\in\bigl\llbracket m\bigl\lfloor\frac{d_b-1}{m}\bigr\rfloor+1,d_b-1\bigr\rrbracket$ with $j\equiv a\pmod m$.

\begin{ex}\normalfont\label{EX7}
Suppose that $G\cong D_{p^s}$ and $X\rightarrow X/G$ is an HKG cover. Thus $p$ is odd, $m=2$, $\chi$ is injective, $\zeta_2=-1$, and $\sigma(x)=\overline{\zeta}_2 x=-x$. For $b\in\llbracket0,p^s-1\rrbracket$, we check that 
\begin{equation}\label{EQ54}
\mu_{b,[0]_2}=\left\lceil\frac{d_b-2}{2}\right\rceil\;\;\;\textup{and}\;\;\;
\mu_{b,[1]_2}=\left\lceil\frac{d_b-1}{2}\right\rceil
\end{equation}
if $d_b\ge 1$ and $\mu_{b,[0]_2}=\mu_{b,[1]_2}=0$ if $d_b=0$. The case $d_b\le 1$ is clear. For $d_b\ge 2$, as $\sigma(x^i dx)=(-1)^{i+1}x^i dx$, $x^idx$ contributes to $\mu_{b,[0]_2}$ (resp.\ $\mu_{b,[1]_2}$) when $i$ is odd (resp.\ even); so, as the set $\llbracket0,d_b-2\rrbracket$ has $\lceil\frac{d_b-2}{2}\rceil$ odd integers and $\lceil\frac{d_b-1}{2}\rceil$ even integers, Equation (\ref{EQ54}) holds.
\end{ex}

{\bf Step 3: computing the $\mu_X$.} The next result proved in \cite{KT2}, Cor.\ 19 follows directly from Equation (\ref{EQ52}).

\begin{proposition}[{\bf Kontogeorgis--Terezakis}]
\label{P7}
Suppose that $X\rightarrow X/G$ is an HKG cover. Then for $(a,b)\in\mathbb J_{m,p^s}$ we have
\begin{equation}\label{EQ55}
\mu_X(a,b)=\mu_{b-1,a}-\mu_{b,a}.
\end{equation}
\end{proposition}

\begin{ex}\normalfont\label{EX8}
Suppose that $X\rightarrow X/G$ is an HKG cover. 

\medskip
{\bf (1)} As $d_{p^s-1}=0$ by Proposition \ref{P6}(3), we have $\mu_{p^s-1,a}=0$ by the very definitions. Therefore we have $\mu_X(a,p^s)=\mu_{p^s-1,a}=0$ by Equation (\ref{EQ55}). Hence $\jmath_{V_X}=0$.

\smallskip
{\bf (2)} As we have $p^s-2=(p-2)+\sum_{i=1}^{s-1} (p-1)p^i$, Equation (\ref{EQ51}) gives
\[
d_{p^s-2}
=
\biggl\lfloor \frac{1}{p^s}\Big[p^{s-1}l_1+\Big(\sum_{l=1}^s p^{s-l}(p-1)\Big)\Big]\biggr\rfloor
=
\left\lfloor \frac{p^s-1+p^{s-1}l_1}{p^s}\right\rfloor
=\left\lfloor \frac{l_1}{p}\right\rfloor+1,
\]
where for the last identity we used that $p\nmid l_1$ by Corollary \ref{C3}(2.a). 
\end{ex}

\begin{ex}\normalfont\label{EX9}
Suppose that $(p,s,m)=(3,2,2)$, $\chi$ is injective (equivalently, an isomorphism), and $X\rightarrow X/G$ is an HKG cover whose wild cyclic subgroup $H=C_9$ has upper jumps
$$u_1=5\;\;\;\textup{and}\;\;\;u_2=17.$$
Thus $G\cong D_9$ and the corresponding lower jumps are
$$l_1=u_1=5\;\;\;\textup{and}\;\;\;l_2=l_1+3(u_2-u_1)=5+3(12)=41.$$
For $b=\delta_{1,b}+\delta_{2,b}3$ with $(\delta_{1,b},\delta_{2,b})\in\{0,1,2\}^2$, Equation (\ref{EQ51}) gives
$$d_b=\Bigl\lfloor \frac{3\bigl(2+5(2-\delta_{1,b})\bigr)+2+41(2-\delta_{2,b})}{9}\Bigr\rfloor.$$
Thus
$$(d_0,d_1,d_2,d_3,d_4,d_5,d_6,d_7,d_8)=(13,11,10,8,7,5,4,2,0).$$
Using Equation (\ref{EQ54}) for $b\in\llbracket0,8\rrbracket$, we get the semisimple multiplicities
$$\bigl((\mu_{0,[0]_2},\mu_{0,[1]_2}),\ldots,(\mu_{7,[0]_2},\mu_{7,[1]_2}),(\mu_{8,[0]_2},\mu_{8,[1]_2})\bigr)$$
$$=\bigl((6,6),(5,5),(4,5),(3,4),(3,3),(2,2),(1,2),(0,1),(0,0)\bigr).$$
Based on this, Equation (\ref{EQ55}) gives that the non-zero indecomposable multiplicities are $\mu_X([0]_2,1)$, $\mu_X([1]_2,1)$, $\mu_X([0]_2,2)$, $\mu_X([0]_2,3)$, $\mu_X([1]_2,3)$, $\mu_X([1]_2,4)$, $\mu_X([0]_2,5)$, $\mu_X([1]_2,5)$, $\mu_X([0]_2,6)$,
$\mu_X([0]_2,7)$, $\mu_X([1]_2,7)$, and $\mu_X([1]_2,8)$, all of them being equal to $1$. Therefore
\begin{align*}
H^0(X,\Omega_X)\cong &\ U_{[0]_2,1}\oplus U_{[1]_2,1}\oplus U_{[0]_2,2}\oplus U_{[0]_2,3}\oplus U_{[1]_2,3}\oplus U_{[1]_2,4}\\
&\oplus U_{[0]_2,5}\oplus U_{[1]_2,5}\oplus U_{[0]_2,6}\oplus U_{[0]_2,7}\oplus U_{[1]_2,7}\oplus U_{[1]_2,8}.
\end{align*}
\end{ex}

{\bf Notation.} In order to make Equation (\ref{EQ55}) more precise we introduce extra notation. For each $a\in C_m$, we consider the non-decreasing function $\nu_a:\mathbb Z\rightarrow\mathbb N\cup\{0\}$ defined by the rule 
$$\nu_a(n):=\begin{cases}|\{j\in\llbracket1,n\rrbracket\mid [j]_m=a\}|\quad\;\,\textup{if}\quad n\in\mathbb N\\
0\quad\quad\quad\quad\quad\quad\quad\quad\quad\quad\quad\textup{if}\quad n\in\mathbb Z\setminus\mathbb N.\end{cases}$$

For $(n,a)\in\mathbb Z\times C_m$, let $(\r_n,\r_a)\in\llbracket0,m-1\rrbracket^2$ be such that $[\r_n]_m=[n]_m$ and $[\r_a]_m=a$. 

Let $\delta_s:\llbracket0,p^s\rrbracket\rightarrow\mathbb N\cup\{0\}$ be the function defined by 
$\delta_s(t):=d_t$ for $t\in\llbracket0,p^s-1\rrbracket$ and
$\delta_s(p^s):=0$. 

For $j\in\llbracket0,s\rrbracket$, let $A_{s;j}:=(\llbracket1,p^s\rrbracket\cap p^j\mathbb N)\setminus p^{j+1}\mathbb N$ be the subset of $\llbracket1,p^s\rrbracket$ formed by integers that are divisible by $p^j$ but are not divisible by $p^{j+1}$; we have a disjoint union 
$$\llbracket1,p^s\rrbracket=\sqcup_{j=0}^s A_{s;j}.$$

{\bf Basic formulas.} We have
\[
\nu_a(n)=\max\Bigl(0,\Big\lfloor\frac{n+\r_{-a}}{m}\Big\rfloor\Bigr)=
\begin{cases}
\big\lfloor\frac{n+\r_{-a}}{m}\big\rfloor\quad\quad\;\textup{if}\quad n\ge m-\r_{-a},\\
0\quad\quad\quad\quad\quad\;\textup{if}\quad n<m-\r_{-a}.
\end{cases}
\]
In particular, if either $n\in\mathbb N\cup\{0\}$ or $n=-1$ and $a\neq [0]_m$ (i.e., and $\r_{-a}\ge 1$), then $\nu_a(n)=\lfloor\frac{n+\r_{-a}}{m}\rfloor$.

For each $n\in\mathbb Z$ with $n\ge -\r_{-a}$ and $(\q,\r)\in (\mathbb N\cup\{0\})\times\llbracket0,m\rrbracket$, we have 
$\nu_a(n+m\q)-\nu_a(n)=\q$ and therefore, as $\nu_a$ is non-decreasing, 
\begin{equation}\label{EQ56}
\nu_a(n+m\q+\r)-\nu_a(n)\in\{\q,\q+1\}.
\end{equation} 

As $\mu_{b,a}=\nu_a\bigl(\delta_s(b)-1\bigr)$ by Equation (\ref{EQ53}) if $\delta_s(b)\ge 2$ and by definitions if $\delta_s(b)\in\{0,1\}$, Equation (\ref{EQ55}) becomes
\begin{equation}\label{EQ57}
\mu_X(a,b)=\nu_a\bigl(\delta_s(b-1)-1\bigr)-\nu_a\bigl(\delta_s(b)-1\bigr)
\end{equation}
for each $(a,b)\in C_m\times\llbracket1,p^s\rrbracket$. 

\begin{proposition}[{\bf The Interval of Integers Principle}]
\label{P8}
Suppose that $X\rightarrow X/G$ is an HKG cover. For each integer $j\in\llbracket0,s-1\rrbracket$ let the pair $(\l_{j,0},\l_{j,-1})\in (\mathbb N\cup\{0\})\times\llbracket0,m-1\rrbracket$ be such that 
$$\Big\lfloor\frac{l_{j+1}}{p^{j+1}}-(p-1)\sum_{i=1}^j\frac{l_i}{p^i}\Big\rfloor=m\l_{j,0}+\l_{j,-1};$$
e.g., $(\l_{0,0},\l_{0,-1})\in (\mathbb N\cup\{0\})\times\llbracket0,m-1\rrbracket$ is such that $\lfloor\frac{l_1}{p}\rfloor=m\l_{0,0}+\l_{0,-1}$. Let $b\in\llbracket1,p^s-1\rrbracket$ and let the pair $(\q_{s,b},\r_{s,b})\in \mathbb Z\times\llbracket0,m-1\rrbracket$ be such that $\delta_s(b-1)-\delta_s(b)=m\q_{s,b}+\r_{s,b}$.  Then the following properties hold.

\medskip
{\bf (1)} If either $a=[0]_m$ with $\delta_s(b)\ge 1$ (i.e., with $b\le p^s-2$) or $a\neq [0]_m$, then we have identities
\begin{equation}\label{EQ58}
\begin{split}\mu_X(a,b)=&\q_{s,b}+\Big\lfloor\frac{\delta_s(b)-1+\r_{-a}+\r_{s,b}}{m}\Big\rfloor-\Big\lfloor\frac{\delta_s(b)-1+\r_{-a}}{m}\Big\rfloor\\
=&\begin{cases}\q_{s,b}+1\quad\;\textup{if}\quad \r_{\delta_s(b)-1+\r_{-a}}+\r_{s,b}\ge m\\
\q_{s,b}\quad\quad\quad\textup{if}\quad \r_{\delta_s(b)-1+\r_{-a}}+\r_{s,b}<m.\end{cases}
\end{split}
\end{equation}
In particular, we have $\mu_X(a,b)=\q_{s,b}+1$ if and only if $\r_{s,b}\neq 0$ (i.e., $\r_{s,b}\in\llbracket1,m-1\rrbracket$) and $a\in [\delta_s(b)]_m+\{[i]_m|i\in\llbracket0,\r_{s,b}-1\rrbracket\}$.

{\bf (2)} We have identities 
\begin{equation*}
\q_{s,p^s-1}=\Big\lfloor\frac{\lfloor \frac{l_1}{p}\rfloor+1}{m}\Big\rfloor=\begin{cases}\l_{0,0}+1\quad\textup{if}\quad\l_{0,-1}=m-1\\
\l_{0,0}\quad\quad\;\;\,\textup{if}\quad \l_{0,-1}\in\llbracket0,m-2\rrbracket,\end{cases}
\end{equation*} 
$\r_{s,p^s-1}=\Big\lfloor \frac{l_1}{p}\Big\rfloor+1-m\Big\lfloor\frac{\lfloor \frac{l_1}{p}\rfloor+1}{m}\Big\rfloor$, and
\begin{equation}\label{EQ59}
\begin{split}\mu_X([0]_m,p^s-1)=\q_{s,p^s-1}+\Big\lfloor\frac{\delta_s(p^s-1)-1+\r_{s,p^s-1}}{m}\Big\rfloor-1-\Big\lfloor\frac{\delta_s(p^s-1)-1}{m}\Big\rfloor\\
=\l_{0,0}=\begin{cases}\q_{s,p^s-1}\quad\quad\quad\textup{if}\quad \r_{s,p^s-1}\ge 1\\
\q_{s,p^s-1}-1\quad\;\textup{if}\quad \r_{s,p^s-1}=0\end{cases}=\begin{cases}\Big\lfloor\frac{\lfloor \frac{l_1}{p}\rfloor+1}{m}\Big\rfloor\quad\quad\quad\;\;\textup{if}\quad m\nmid \lfloor \frac{l_1}{p}\rfloor+1\\
\Big\lfloor\frac{\lfloor \frac{l_1}{p}\rfloor+1}{m}\Big\rfloor-1\quad\;\;\;\textup{if}\quad m\mid\lfloor \frac{l_1}{p}\rfloor+1.\end{cases}
\end{split}
\end{equation}

{\bf (3)} For each $a\in C_m$ we have inequalities 
$$\mu_X([0]_m,p^s-1)\le\mu_X(a,p^s-1)\le \mu_X([0]_m,p^s-1)+1.$$

{\bf (4)} For each pair $(a_1,a_2)\in C_m^2$ and every $b\in\llbracket1,p^s-1\rrbracket$ we have inequalities $0\le |\mu_X(a_1,b)-\mu_X(a_2,b)|\le 1$.

\smallskip
{\bf (5)} Let $j\in\llbracket0,s-1\rrbracket$. We assume that $b\in A_{s;j}$; so $p^j\mid b$ but $p^{j+1}\nmid b$. Then the following properties hold.

\medskip\noindent
{\bf (5.a)} We have identities
$$\q_{s,b}=\begin{cases}\l_{j,0}+1\quad\quad\textup{if}\quad\l_{j,-1}=m-1\quad\textup{and}\quad \r_{s,b}=0\\
\l_{j,0}\quad\quad\quad\;\;\,\textup{if}\quad\l_{j,-1}=m-1\quad\textup{and}\quad \r_{s,b}>0\\
\l_{j,0}\quad\quad\quad\;\;\,\textup{if}\quad \l_{j,-1}\in\llbracket0,m-2\rrbracket.\end{cases}$$

\noindent
{\bf (5.b)} If $j=0$ we assume that $b\neq p^s-1$. Then for each $a\in C_m$ we have inequalities 
$$\l_{j,0}\le \q_{s,b}\le\mu_X(a,b)\le \l_{j,0}+1\le\q_{s,b}+1.$$

\noindent
{\bf (5.c)} If $\l_{j,-1}=m-1$, $\r_{s,b}=0$, and for $j=0$ we have $b\neq p^s-1$, then $\mu_X(a,b)=\q_{s,b}=\l_{j,0}+1$.

\noindent
{\bf (5.d)} There exists a uniquely determined subset $\mathbb J_{X;j}$ of $C_m\times A_{s;j}$ such that 
$$\mu_X(a,b)=\begin{cases}\l_{j,0}+1\quad\;\textup{if}\quad (a,b)\in\mathbb J_{X;j}\\
\l_{j,0}\quad\quad\quad\textup{if}\quad (a,b)\in(C_m\times A_{s;j})\setminus\mathbb J_{X;j}.\end{cases}$$

{\bf (6)} Let $\mathbb J_X:=\sqcup_{j=0}^{s-1}\mathbb J_{X;j}$. We have an identity
$$\textup{g}_X=\frac{mp^s(p-1)}{2}\sum_{j=0}^{s-1} \l_{j,0}p^{s-1-j}+\sum_{(a,b)\in\mathbb J_X} b.$$
\end{proposition}

\begin{proof}
As $\delta_s(b-1)-1\ge 0$ by Proposition \ref{P6}(3), we have an identity $\nu_a\bigl(\delta_s(b-1)-1\bigr)=\lfloor\frac{\delta_s(b-1)-1+\r_{-a}}{m}\rfloor=\q_{s,b}+\lfloor\frac{\delta_s(b)-1+\r_{-a}+\r_{s,b}}{m}\rfloor$. 

For part (1) we have an identity $\nu_a\bigl(\delta_s(b)-1\bigr)=\lfloor\frac{\delta_s(b)-1+\r_{-a}}{m}\rfloor$ and hence part (1) follows from Equation (\ref{EQ57}).

For parts (2) to (5) we use the notation of Proposition \ref{P3}(4). 

We first prove part (5.a) when $j=0$. As $p\nmid b$, we have an identity $N_{s,b-1}=N_{s,b}+p^{s-1}l_1$ by Proposition \ref{P3}(4) and thus Equation (\ref{EQ51}) gives
\begin{equation}\label{EQ60}
\delta_s(b-1)-\delta_s(b)=\Big\lfloor\frac{l_1}{p}+\frac{N_{s,b}}{p^s}\Big\rfloor-\Big\lfloor\frac{N_{s,b}}{p^s}\Big\rfloor\in\Big\{\Big\lfloor\frac{l_1}{p}\Big\rfloor,\Big\lfloor\frac{l_1}{p}\Big\rfloor+1\Big\}.
\end{equation}
Thus $m\q_{s,b}+\r_{s,b}\in\{m\l_{0,0}+\l_{0,-1},m\l_{0,0}+\l_{0,-1}+1\}$ and from this we get that part (5.a) holds for $j=0$.

For part (2), we have $\delta_s(p^s-1)=0$ by Proposition \ref{P6}(3), $\r_{-[0]_m}=0$, and $\delta_s(p^s-2)=\left\lfloor \frac{l_1}{p}\right\rfloor+1$ by Example \ref{EX8}(2). Therefore $\q_{s,p^s-1}=\Big\lfloor\frac{\lfloor \frac{l_1}{p}\rfloor+1}{m}\Big\rfloor$ and $\r_{s,p^s-1}=\Big\lfloor \frac{l_1}{p}\Big\rfloor+1-m\Big\lfloor\frac{\lfloor \frac{l_1}{p}\rfloor+1}{m}\Big\rfloor$. The other identities of part (2) follow from them and the case $j=0$ of part (5.a). So part (2) holds. 

Based on parts (1) and (2) and the case $j=0$ of part (5.a), to prove part (3) we can assume that $\mu_X([0]_m,p^s-1)=\q_{s,p^s-1}-1$. We have $m\mid\lfloor \frac{l_1}{p}\rfloor+1$, $\l_{0,-1}=m-1$, $\q_{s,p^s-1}=\l_{0,0}+1$, $\r_{s,p^s-1}=0$, and $\mu_X([0]_m,p^s-1)=\q_{s,p^s-1}-1=\l_{0,0}$ by part (2).  As $\r_{s,p^s-1}=0$, from part (1) we get that for each $a\in C_m\setminus\{[0]_m\}$ we have $\mu_X(a,p^s-1)=\q_{s,p^s-1}$. Thus part (3) holds.

Based on part (3), to prove part (4) we can assume that $b\le p^s-2$, and this case follows from part (1). So part (4) holds. 

Part (5.b) holds for $j=0$ by part (1). Thus to prove parts (5.a) and (5.b) we can assume that $s\ge 2$ and $j\in\llbracket1,s-1\rrbracket$. 

We have $N_{s,b-1}=N_{s,b}+l_{j+1}p^{s-j-1}-(p-1)\sum_{i=1}^j p^{s-i}l_i$ by Proposition \ref{P3}(4). Using Equation (\ref{EQ51}) we get that
$$\delta_s(b-1)-\delta_s(b)=\Big\lfloor \frac{N_{s,b}}{p^s}+\frac{l_{j+1}}{p^{j+1}}-(p-1)\sum_{i=1}^j\frac{l_i}{p^i}\Big\rfloor-\Big\lfloor\frac{N_{s,b}}{p^s}.\Big\rfloor$$ 
So $m\q_{s,b}+\r_{s,b}\in\{m\l_{j,0}+\l_{j,-1},m\l_{j,0}+\l_{j,-1}+1\}$. Thus, if $\l_{j,-1}\le m-2$, then $\q_{s,b}=\l_{j,0}$ and part (5.a) holds. 

If $\q_{s,b}=\l_{j,0}$, then from part (1) we get that part (5.b) holds. So to prove parts (5.a) and (5.b) we can assume that $\l_{j,-1}=m-1$. So we have the belonging relation $\delta_s(b-1)-\delta_s(b)\in\{m(\l_{j,0}+1)-1,m(\l_{j,0}+1)\}$.

We first assume that $\q_{s,b}>\l_{j,0}$. From the belonging relation and the inequalities $\delta_s(b-1)-\delta_s(b)\ge m\q_{s,b}\ge m(\l_{j,0}+1)$ we get the identity $\delta_s(b-1)-\delta_s(b)=m(\l_{j,0}+1)$. So $\q_{s,b}=\l_{j,0}+1$ and $\r_{s,b}=0$. As $\r_{s,b}=0$, we have $\mu_X=\q_{s,b}$ by part (1). So part (5.b) holds. 

Next we assume that $\q_{s,b}=\l_{j,0}$. From this and the belonging relation we get that $\delta_s(b-1)-\delta_s(b)=m(\l_{j,0}+1)-1$ and hence $m\nmid \delta_s(b-1)-\delta_s(b)$. 

From the last two paragraphs we get that part (5.a) holds if $\l_{j,-1}=m-1$. So part (5.a) holds. 

Part (5.c) follows from parts (5.a) and (5.b). 

Part (5.d) follows from part (5.b) and, when $b=p^s-1$, also from part (2). Thus part (5) holds.

For part (6), we first note that 
\begin{equation*}
\begin{aligned}
\sum_{b\in A_{s;j}} b&=p^j\Big(\sum_{i=1}^{p^{s-j}} i\Big)-p^{j+1}\Big(\sum_{i=1}^{p^{s-j-1}} i\Big)\\
&=\frac{p^{s}(p^{s-j}+1)}{2}-\frac{p^{s}(p^{s-j-1}+1)}{2}=\frac{p^{2s-j}-p^{2s-j-1}}{2}=\frac{p^{2s-j-1}(p-1)}{2}.\\
\end{aligned}
\end{equation*}
Based on this and $\mu_X(a,p^s)=0$ for each $a\in C_m$, we compute
\begin{equation*}
\begin{aligned}
\textup{g}_X&=\dim_k(V_X)=\sum_{(a,b)\in\mathbb J_{m,p^s}}\mu_X(a,b)b=\sum_{j=0}^{s-1} \sum_{(a,b)\in C_m\times A_{s;j}}\mu_X(a,b)b\\
&=\sum_{j=0}^{s-1} \Big[\sum_{(a,b)\in (C_m\times A_{s;j})\setminus\mathbb J_{X;j}} \l_{j,0}b+\sum_{(a,b)\in \mathbb J_{X;j}} (\l_{j,0}+1)b\Big]\\
&=m\Big(\sum_{j=0}^{s-1} \l_{j,0}\sum_{b\in A_{s;j}} b\Big)+\sum_{j=0}^{s-1}\sum_{(a,b)\in \mathbb J_{X;j}}b\\
&=\frac{mp^s(p-1)}{2}\Big(\sum_{j=0}^{s-1} \l_{j,0}p^{s-1-j}\Big)+\sum_{(a,b)\in\mathbb J_X} b.\\
\end{aligned}
\end{equation*}
So part (6) holds.\end{proof}

\begin{definition}\label{D6} Suppose that $m\ge 2$ and $X\rightarrow X/G$ is an HKG cover. For a subset $A$ of $\llbracket1,p^s\rrbracket$, we consider the restriction
$$\mu_{X;A}:=\mu_X|C_m\times A:C_m\times A\rightarrow\mathbb N\cup\{0\}$$
of $\mu_X$ to $C_m\times A$; so $\mu_X=\mu_{X;\llbracket1,p^s\rrbracket}$. If $A=\{b\}$, let $\mu_{X;b}:=\mu_{X,\{b\}}$.

\medskip
{\bf (1)} For $m\ge 2$ (resp.\ $m\ge 3$) we say that $\mu_{X;A}$ has exceptional $+1$ (resp.\ $-1$) jumps or that $\mu_{X;A}$ has jump type $1$ (resp.\ $-1$) if there exists $b\in A\cap\llbracket 1,p^s-1\rrbracket$ such that the function $\mu_{X;b}$ is non-constant and for each such $b$, $\mu_{X;b}$ attains its maximum (resp.\ minimum) at precisely one element of the source.

\smallskip
{\bf (2)} We say that $\mu_{X;A}$ has jump type $0$ if either $m=2$ and it does not have jump type $ 1$ or $m\ge 3$ and it does not have jump types $-1$ or $1$. 

\smallskip
{\bf (3)} We say that $\mu_{X;A}$ has regular $\pm 1$ jumps if for each $b\in A$ the function $\mu_{X;b}$ is non-constant, attains its maximum at two or more elements of the source, and attains its minimum at two or more elements of the source.

\smallskip
{\bf (4)} For $m=2$ (resp.\ $m\ge 3$), we fix an $s$-tuple $(\epsilon_0,\ldots,\epsilon_{s-1})\in\{0,1\}^s$ (resp. $(\epsilon_0,\ldots,\epsilon_{s-1})\in\{-1,0,1\}^s$). We say the multiplicity function $\mu_X$ has jump type $(\epsilon_0,\ldots,\epsilon_{s-1})$ if for each $j\in\llbracket0,s-1\rrbracket$, $\mu_{X;A_{s,j}}$ has jump type $\epsilon_j$.

\smallskip
{\bf (5)} We call the $2s$-tuple 
$$(\l_{0,0},|\mathbb J_{X;0}|,\l_{1,0},|\mathbb J_{X;1}|,\ldots,\l_{s-1,0},|\mathbb J_{X;s-1}|)\in(\mathbb N\cup\{0\})^{2s} $$
 as the multiplicity type of $\mu_X$. 
\end{definition}

Clearly, if $\mu_{X;A}$ has regular $\pm 1$ jumps, then it has jump type $0$.

For a function $\mathfrak f$, let $\Im(\mathfrak f)$ be its image. 

\begin{ex}\normalfont\label{EX9.5}
Suppose that $m\ge 3$ and $X\rightarrow X/G$ is an HKG cover. For the lower jumps $l_1<\cdots<l_s$, we use the notation of Proposition \ref{P8}.

\medskip
{\bf (1)} If $\r_{s,p^s-1}=0$, then $\Im(\mu_{X;p^s-1})=\{q_{s,p^s-1}-1,\q_{s,p^s-1}\}$ and moreover $|\mu^{-1}_{X;p^s-1}(q_{s,p^s-1}-1)|=1$ by the proof of Proposition \ref{P8}(3); hence $\mu_{X;p^s-1}$ has jump type $-1$. If  $\r_{s,p^s-1}>0$, then $\Im(\mu_{X;p^s-1})=\{q_{s,p^s-1},\q_{s,p^s-1}+1\}$ and $|\mu^{-1}_{X;p^s-1}(q_{s,p^s-1}+1)|=\r_{s,p^s-1}$ by Proposition \ref{P8}(1) and (4); hence $\mu_{X;p^s-1}$ has jump type $1$ if $\r_{s,p^s-1}=1$, has jump type $-1$ if $\r_{s,p^s-1}=m-1$, and has regular $\pm 1$ jumps if $\r_{s,p^s-1}\in\llbracket2,m-2\rrbracket$. 

\smallskip
{\bf (2)} Let $b\in\llbracket1,p^s-2\rrbracket$. If $\r_{s,b}=0$, then the function $\mu_{X;b}$ is constant by Proposition \ref{P8}(1) and thus it has jump type $0$. As in part (1), by Proposition \ref{P8}(1), if $\r_{s,b}=1$ then $\mu_{X;b}$ has jump type $1$, if $\r_{s,b}=m-1$ then $\mu_{X;b}$ has jump type $-1$, and if $\r_{s,b}\in\llbracket2,m-2\rrbracket$ then $\mu_{X;b}$ has regular $\pm 1$ jumps.

\smallskip
{\bf (3)} Let $(\epsilon_0,\ldots,\epsilon_{s-1})\in\{-1,0,1\}^s$ be the jump type of $\mu_X$. From parts (1) and (2) we get that $\epsilon_0$ is $1$ if and only if we have $\r_{s,b}=1$ for each $b\in A_{s;0}$ and it is $-1$ if and only if $\r_{s,p^s-1}\in\{0,m-1\}$ and $\r_{s,b}=m-1$ for each $b\in A_{s;0}\setminus\{p^s-1\}$. Similarly, if $j\in\llbracket1,s-1\rrbracket$, from parts (1) and (2) we get that $\epsilon_j$ is $1$ (resp.\ $-1$) if and only if we have $\r_{s,b}=1$ (resp.\ $\r_{s,b}=m-1$) for each $b\in A_{s;j}$.

\smallskip
{\bf (4)} If $\bigl[\big\lfloor\frac{l_1}{p}\big\rfloor\bigr]_m\in\{[i]_m|i\in\llbracket2,m-3\rrbracket\}$, then from Equation (\ref{EQ60}) we get that $\r_{s,b}\le \in\llbracket2,m-2\rrbracket$ for each $b\in A_{s;0}$ and hence $\epsilon_0=0$ by part (3). 
\end{ex}

\section{Actions of G on curves, Part III: basic results}\label{S12}

\begin{theorem}\label{T5}
Suppose that $X\rightarrow X/G$ is an HKG cover. Then $\chi$ is injective if and only if the character function $\ell_{V_X}$ is constant.
\end{theorem}

\begin{proof}
To check this we can assume that $m\ge 2$.

Recall that $Y\cong\mathbb P_k^1$. Let $\overline{\zeta}\in k^\times$ be the eigenvalue of the action of $\sigma$ on the tangent space of $Y$ at the branch point $\infty\in Y(k)$. The induced action of $\Sigma$ on $Y$ is faithful, so $\overline{\zeta}$ is a primitive $m$-th root of unity.

For the `only if' part we fix $l\in\llbracket1,m-1\rrbracket$. The automorphism $\sigma^l$ fixes precisely two $k$-valued points of $Y$: the image $\infty$ of the wildly ramified point $P\in X(k)$ and another point $O_0$ by Definition \ref{D5}(3). Clearly, $G$ fixes $P$. Let $O\in X(k)$ be a point above $O_0$ whose stabilizer is $\Sigma$. As $H\cap\Sigma$ is trivial, the points of $X$ above $O_0$ are uniquely of the form $h(O)$ with $h\in H$. For each $h\in H$ we have
$$\sigma^l(h(O))=h^{\chi(\sigma^l)}(O).$$
As $\chi$ is injective and $\sigma^l$ is not the identity element, the automorphism $\chi(\sigma^l)$ fixes only the identity element of $H$. Thus $\sigma^l$ fixes only $O$ among the points of $X$ above $O_0$. So $\sigma^l$ fixes exactly two points of $X$.

The tangent eigenvalues of $\sigma^l$ on $Y$ at its two fixed points $\infty$ and $O_0$ are $\overline{\zeta}^l$ and $\overline{\zeta}^{-l}$ (respectively). The morphism $X\rightarrow Y$ is \'etale at each point $O$ above $O_0$, so the tangent eigenvalue on $X$ at $O$ is $\overline{\zeta}^{-l}$. It remains to identify the tangent eigenvalue at $P$. 

Let $x$ be a local parameter of $X$ at $P$. As $X\rightarrow Y$ is totally ramified of degree $p^s$ at $P$, let $y$ be a local parameter of $Y$ at $\infty$ such that
$$y\equiv x^{p^s}\pmod{(x^{p^s+1})}.$$
If $\omega\in k^{\times}$ is such that
$$\sigma(x)\equiv\omega x\pmod{(x^2)},$$
then, on applying $\sigma$ to the congruence for $y$, we obtain
$$\sigma(y)\equiv\sigma(x)^{p^s}\equiv\omega^{p^s}x^{p^s}\pmod{(x^{p^s+1})}.$$
On the other hand, the tangent eigenvalue of $\sigma$ on $Y$ at $\infty$ is $\overline{\zeta}$, and therefore
$$\sigma(y)\equiv\overline{\zeta}y\equiv\overline{\zeta}x^{p^s}\pmod{(x^{p^s+1})}.$$
By comparing the coefficients of $x^{p^s}$ we get $\omega^{p^s}=\overline{\zeta}$. Moreover, $\sigma^m=1$ implies that $\omega^m=1$. As $\chi$ is injective, $m\mid p-1$, and hence $p^s\equiv 1\pmod m$. It follows that $\omega^{p^s}=\omega$, so $\omega=\overline{\zeta}$. So the tangent eigenvalue of $\sigma^l$ at $P$ is $\overline{\zeta}^l$.

We now lift only the tame pair $(X,\Sigma)$ to characteristic $0$. After lifting the quotient curve $X/\Sigma$ together with the branch divisor, the cyclic tame cover $X\rightarrow X/\Sigma$ and its $\Sigma$-action lift by \cite{W}, Cor.\ 3.1.3. Thus there exists an $R$ and a connected smooth projective curve $\mathcal X$ over $R$ together with an automorphism that lifts $\sigma$.

The cover $X\rightarrow X/\Sigma$ is totally ramified at $P$ and $O$ and unramified elsewhere, as every non-trivial power of $\sigma$ fixes precisely these two points. The tame lifting has the local form $z=w^m$ above each lifted branch section; see \cite{W}, Prop.\ 2.3.4. Thus $P$ and $O$ lift uniquely to $R$-valued points of $\mathcal X$ fixed by $\sigma$, and no non-identity element of $\Sigma$ has any other fixed points on the generic fiber. At these $R$-valued points, the tangent eigenvalues of $\sigma$ are $m$-th roots of unity reducing to $\overline\zeta$ and $\overline\zeta^{-1}$, hence are $\zeta$ and $\zeta^{-1}$, where $\zeta\in R_0^{\times}\subset R^\times$ is the Teichm\"uller lift of $\overline\zeta$. Consequently, for every $l\in\llbracket1,m-1\rrbracket$, the generic fiber has exactly two fixed points of $\sigma^l$, with tangent eigenvalues $\zeta^l$ and $\zeta^{-l}$.

We apply the Lefschetz fixed point formula of Atiyah--Bott \cite{AB}, Sect.\ 4, Thm.\ 4.12 to the automorphism $\sigma^l$ of the generic fiber $\mathcal X_B$, where $B=\Frac(R)$. If $\zeta_{\mathcal P}$ denotes the tangent eigenvalue of $\sigma^l$ at a point $\mathcal P\in\mathcal X_B$ fixed by it, i.e., at a $\mathcal P\in \mathcal X_B^{\sigma^l}$, then
$$\sum_{i=0}^1 (-1)^i\textup{Tr}\bigl(\sigma^l|H^i(\mathcal X_B,\mathcal O_{\mathcal X_B})\bigr)=\sum_{\mathcal P\in \mathcal X_B^{\sigma^l}}\frac{1}{1-\zeta_{\mathcal P}}.$$
Using $H^0(\mathcal X_B,\mathcal O_{\mathcal X_B})=B$ and Serre duality, this becomes
$$\textup{Tr}\bigl(\sigma^{-l}|H^0(\mathcal X_B,\Omega_{\mathcal X_B})\bigr)=
1-\frac{1}{1-\zeta^l}-\frac{1}{1-\zeta^{-l}}=0.$$
As $l$ runs through $\llbracket1,m-1\rrbracket$, so does $-l$ modulo $m$. Therefore
$$\textup{Tr}\bigl(\sigma^l|H^0(\mathcal X_B,\Omega_{\mathcal X_B})\bigr)=0$$
for all $l\in\llbracket1,m-1\rrbracket$.

We now translate this trace calculation back into the character function on the special fiber. As $m$ is invertible in $R$, the polynomial
$$T^m-1=\prod_{a\in C_m}(T-\zeta_m^a)\in R[T]$$
has distinct roots over $R$. Thus the finite free $R$-module $H^0(\mathcal X,\Omega_{\mathcal X/R})$ decomposes as the direct sum of the kernels of $\sigma-\zeta_m^a$. Each kernel is finite free over $R$, and reducing this decomposition modulo the maximal ideal of $R$ gives the corresponding eigenspace decomposition of $V_X$. As $k[\Sigma]$ is semisimple, the dimension of the $\overline{\zeta}_m^a$-eigenspace of $V_X$ is precisely $\ell_{V_X}(a)$. Hence the rank of the $\zeta_m^a$-eigenspace of $H^0(\mathcal X_B,\Omega_{\mathcal X_B})$ is $\ell_{V_X}(a)$. Thus
$$0=\textup{Tr}\bigl(\sigma^l|H^0(\mathcal X_B,\Omega_{\mathcal X_B})\bigr)
=\sum_{a\in C_m}\ell_{V_X}(a)\zeta_m^{al}\in R$$
for each $l\in\llbracket1,m-1\rrbracket$. Hence, by denoting
$$F(T):=\sum_{a=0}^{m-1}\ell_{V_X}(a)T^a\in \mathbb Z[T],$$
we have $F(\zeta_m^l)=0$ for each $l\in\llbracket1,m-1\rrbracket$. As $F(1)=\dim_k(V_X)=\textup{g}_X$, the polynomial
$$F_0(T):=F(T)-\frac{\textup{g}_X}{m}(1+T+\cdots+T^{m-1})\in\mathbb Q[T]$$
has degree at most $m-1$ and vanishes at each $m$-th root of unity in $\mathbb C$. Therefore $F_0(T)=0$. Hence
$$F(T)=\frac{\textup{g}_X}{m}(1+T+\cdots+T^{m-1}).$$
So $\ell_{V_X}(a)=\frac{\textup{g}_X}{m}$ for each $a\in C_m$ and thus the `only if' part holds.

For the `if' part recall that $Q(G)=\Ker(\chi)$, $q=|Q(G)|$, $m=qd$, and the character function $\ell_{V_X}$ gives the dimensions of the eigenspaces of the action of $\sigma$ on $V_X=H^0(X,\Omega_X)$. The group $Q(G)$ is generated by $\sigma^d$. Therefore the $\overline{\zeta}_m^a$-eigenspace of $\sigma$ is fixed pointwise by $Q(G)$ if and only if $\overline{\zeta}_m^{ad}=1$, equivalently, if and only if $a\in qC_m$. As $qC_m$ has $d$ elements and $\ell_{V_X}$ is constant, $V_X^{Q(G)}$ is the direct sum of $d$ eigenspaces, each of dimension $\frac{\dim_k(V_X)}{m}$. Thus
$$\dim_k(V_X^{Q(G)})=d\frac{\dim_k(V_X)}{m}=\frac{\dim_k(V_X)}{q}=\frac{\textup{g}_X}{q}.$$

The quotient morphism $X\rightarrow X/Q(G)$ is tame. For the tame quotient morphism $X\rightarrow X/Q(G)$, regular differentials on $X/Q(G)$ pull back isomorphically onto the $Q(G)$-invariant regular differentials on $X$. Thus
$$\dim_k(V_X^{Q(G)})=\textup{g}_{X/Q(G)}.$$
Combining the two equalities we get
$$\textup{g}_{X/Q(G)}=\frac{\textup{g}_X}{q}.$$

The point $P$ is fixed by $Q(G)$. Let $O\in X$ be a point above $O_0$; its stabilizer is $\Sigma$ and the points of $X$ above $O_0$ are $h(O)$ with $h\in H$. For each $h\in H$ we have
$$\sigma^d\bigl(h(O)\bigr)=h^{\chi(\sigma^d)}(O)=h(O).$$
So all points of $X$ above $O_0$ are fixed by $Q(G)$.

These are the only points of $X$ fixed by non-trivial elements of $Q(G)$ as any such fixed point on $X$ maps to a fixed point of the corresponding non-trivial automorphism of $Y$, and the cyclic tame action of $Q(G)$ on $Y\cong\mathbb P_k^1$ has only the two fixed points $\infty$ and $O_0$. So the quotient morphism $X\rightarrow X/Q(G)$ has exactly $p^s+1$ ramification points, namely $P$ and the $p^s$ points above $O_0$, each with ramification index $q$. From this and the Riemann--Hurwitz formula we get that
$$2\textup{g}_X-2=q(2\textup{g}_{X/Q(G)}-2)+(q-1)(p^s+1).$$
Substituting $\textup{g}_{X/Q(G)}=\frac{\textup{g}_X}{q}$ we get
$$2(q-1)=(q-1)(p^s+1).$$
As $p^s+1\ge 3$, we get that $q=1$. So $\chi$ is injective and the `if' part holds.
\end{proof}

\begin{theorem}\label{T6}
Suppose that $\chi$ is injective. If $X\rightarrow X/G$ is not an HKG cover, then we assume that $m=2$ and $Y\cong\mathbb P_k^1$. Then, for every $R$ containing $R_s$, there exists a lift $\mathcal V$ of $V_X$ to $R$ such that $r_{\mathcal V}=0$.
\end{theorem}

\begin{proof}
If $X\rightarrow X/G$ is an HKG cover, then the function $\ell_{V_X}$ is constant by Theorem \ref{T5} and $\jmath_{V_X}=0$ by Example \ref{EX8}(1). If $m=2$ and $Y\cong\mathbb P_k^1$, then the function $\ell_{V_X}$ is constant by Proposition \ref{P4}(4) and we have $\jmath_{V_X}=0$ by Proposition \ref{P4}(3). 

As $\ell_{V_X}$ is constant and $\jmath_{V_X}=0$, the theorem holds by Theorem \ref{T2}(2).\end{proof}

\begin{corollary}\label{C5}
Suppose that $\chi$ is injective, $p$ is odd, $m$ is even, and $\rho$ is such that $Y\cong\mathbb P^1_k$. If there exists no lift $\mathcal V$ of $V_X$ to $R$ which is rational as an $R[H]$-module and satisfies $r_{\mathcal V}=0$, then the action $\rho$ does not lift to $R$ (equivalently, there exists no smooth projective curve $\mathcal X$ over $\Spec R$ whose reduction modulo $\upsilon$ is $X$ and for which we have a homomorphism $\varrho:G\rightarrow\Aut(\mathcal X)$ that lifts $\rho$).
\end{corollary}

\begin{proof}
It suffices to show that the assumption that $\mathcal X$ and $\varrho$ exist leads to a contradiction. We use the notation of the proof of Proposition \ref{P5}. We have $r_{\mathcal M_X}=r_{M_X}$ by Isomorphism (\ref{EQ49}). Together with Proposition \ref{P4}(1) and (2), this gives $r_{\mathcal M_X}=0$. Then Short Exact Sequence (\ref{EQ48}) gives $r_{\mathcal V_{\mathcal X}}=0$. By the hypothesis on $V_X$, this lift $\mathcal V_{\mathcal X}$ is not rational as an $R[H]$-module. This contradicts Proposition \ref{P5}.\end{proof}

\begin{remark}\normalfont\label{R5}
Suppose that $G=D_{p^s}$; so $p$ is odd, $m=2$, and $\chi$ is injective. Based on Theorem \ref{T4} and its analog over each spectrum $\Spec R$ (see \cite{K}, Prop.\ 1.4.2), $G$ is a local Oort group over $k$ if and only if each HKG cover $X\rightarrow X/G$ over $k$ lifts to some $R$ that contains $R_s$. Theorem \ref{T6} shows that the $k[G]$-module $V_X$ has a lift $\mathcal V$ to $R$ with $r_{\mathcal V}=0$.
\end{remark}

\section{On $B[G]$-modules of $1$-forms of lifts of HKG covers}\label{S13}

We recall the Artin--Schreier--Witt construction in the form used below. Let
$K$ be a field of characteristic $p$ and let $W_s(K)$ be the ring of $p$-typical
Witt vectors of length $s$ with coefficients in $K$. The additive endomorphism
$$\wp:=\Frob-\mathrm{id}:W_s(K)\rightarrow W_s(K),$$
is defined by subtracting the identity from the Witt-vector Frobenius. Here
$$\Frob(y_0,\ldots,y_{s-1})=(y_0^p,\ldots,y_{s-1}^p)$$
and the subtraction is taken in the additive group of $W_s(K)$. On valued
points, it is defined by the rule
$$(y_0,\ldots,y_{s-1})\mapsto(\wp_0,\ldots,\wp_{s-1}),$$
where $(\wp_0,\ldots,\wp_{s-1})\in K[y_0,\ldots,y_{s-1}]^s$ is an $s$-tuple of
polynomials. For instance $\wp_0(y_0,\ldots,y_{s-1})=y_0^p-y_0$.

For $\underline{z}=(z_0,\ldots,z_{s-1})\in W_s(K)$, we consider the Artin--Schreier--Witt $K$-algebra 
$$L:=K[y_0,\ldots,y_{s-1}]/(\wp_0-z_0,\ldots,\wp_{s-1}-z_{s-1})$$ that parameterizes solutions of the equation $\wp(\underline{y})=\underline{z}$.

We have $z_0\notin\{\beta^p-\beta|\beta\in K\}$ if and only if $L$ is a cyclic extension of $K$ of order $p^s$ (e.g., \cite{L}, Ch.\ 26, Sect.\ 5, Thm.\ 5). If this condition for $z_0$ holds, then the resulting Galois group $\Aut(L/K)=\{T_{[b]_{p^s}}|b\in\llbracket1,p^s\rrbracket\}$ is formed by the Witt-vector translations
\[
T_{[b]_{p^s}}(\underline{y}):=\underline{y}+[b]_{p^s}\;\;\;\textup{with}\;\;\; b\in W_s(\mathbb F_p)=\mathbb Z/p^s\mathbb Z=C_{p^s}.
\]
For references see \cite{S}, Ch.\ X, Sect.\ 3 a) and \cite{L}, Ch.\ 26, Sect.\ 5.

\begin{lemma}\label{L9}
With the notation above, let $g\in\Aut(K)$ and
$u\in(\mathbb Z/p^s\mathbb Z)^\times$ be such that $g(\underline{z})=u\cdot\underline{z}$ in $W_s(K)$, where $\cdot$ denotes
multiplication in the ring of $p$-typical Witt vectors of length $s$. Then the
rules
\[
\widetilde{g}|_K=g\;\;\;\textup{and}\;\;\;\widetilde{g}(\underline{y})=u\cdot\underline{y}
\]
are compatible with the equation $\wp(\underline{y})=\underline{z}$ and therefore define an extension $\widetilde{g}\in\Aut(L)$ of $g$ such that for each $b\in\llbracket1,p^s\rrbracket$, we have 
$\widetilde{g}T_{[b]_{p^s}}\widetilde{g}^{-1}=T_{u^{-1}[b]_{p^s}}$.
\end{lemma}

\begin{proof}
The endomorphism $\wp$ is additive and hence commutes with multiplication by elements
of $W_s(\mathbb F_p)=\mathbb Z/p^s\mathbb Z$. Thus
$$\wp\bigl(\widetilde{g}(\underline{y})\bigr)=\wp(u\cdot\underline{y})=u\cdot\wp(\underline{y})
=u\cdot\underline{z}=g(\underline{z})=\widetilde{g}(\underline{z}).$$
So the displayed assignment respects the Artin--Schreier--Witt equations and
defines an extension $\widetilde{g}:L\rightarrow L$ of $g$; its inverse is obtained in the same way
from $g^{-1}$ and $u^{-1}$. As $[b]_{p^s}\in W_s(\mathbb F_p)=C_{p^s}$ is fixed by $g$,
$\widetilde{g}T_{[b]_{p^s}}\widetilde{g}^{-1}$ sends $\underline{y}$ to
$\underline{y}+u^{-1}b$ and hence $\widetilde{g}T_{[b]_{p^s}}\widetilde{g}^{-1}=T_{u^{-1}[b]_{p^s}}$. So the lemma holds.
\end{proof}

\smallskip
In the HKG setting below, the relevant local extension is obtained by completing
at the unique wildly ramified point, so the base local field is isomorphic to
$k((z))$. Recall the following standard cyclic Artin--Schreier--Witt upper-jump criterion (see \cite{OP}, Lem.\ 3.4). 

\begin{lemma}[{\bf Obus--Pries}]
\label{L10}
Suppose that $\underline{z}=(z_0,\ldots,z_{s-1})\in W_s\bigl(k((z))\bigr)$ is in standard form in the sense of \cite{OP}, Sect.\ 3.2, i.e.,
each non-zero $z_i$ is a polynomial in $z^{-1}$ whose monomial exponents are not
divisible by $p$. For $i\in\llbracket1,s\rrbracket$, let $d_i\in\mathbb N\cup\{0\}$ be the pole order of $z_{i-1}$ with respect to the uniformizer $z$ of $k[[z]]$, with
$d_i:=-\infty$ if $z_{i-1}=0$. If $\wp(\underline{y})=\underline{z}$ defines a cyclic extension of $k((z))$ of degree
$p^s$, then its $i$-th upper jump is
\[
u_i:=\max(p^{i-j}d_j|j\in\llbracket1,i\rrbracket).
\]
In particular, $u_1=d_1$ and $u_i=\max(p u_{i-1},d_i)$ for each $i\in\llbracket2,s\rrbracket$.
\end{lemma}

\begin{proposition}\label{P10}
Recall that $G=C_{p^s}\rtimes_\chi C_m$ and $q=|Q(G)|$. Let $l_1<\cdots<l_s$ (resp.\ $u_1<\cdots<u_s$) be a strictly increasing sequence of natural numbers that occur as the lower (resp.\ upper) jumps of a cyclic $C_{p^s}$-extension of $k((z))$, i.e., such that the conditions of Corollary \ref{C3}(2) (resp.\ Lemma \ref{L7}(2)) hold. Then there exists an HKG cover $X\rightarrow X/G$ with $l_1<\cdots<l_s$ (resp.\ $u_1<\cdots<u_s$) as the lower (resp.\ upper) jumps of the wild cyclic subgroup $H=C_{p^s}$ if and only if 
\begin{equation}\label{EQ61}
\gcd(m,l_1)=q\;\;\;\;\textup{and}\;\;\;\;l_i\equiv l_1\pmod m\;\;\;\forall i\in\llbracket2,s\rrbracket
\end{equation}
(resp.\ $\gcd(m,u_1)=q$ and $u_i\equiv u_1\pmod m$ for all $i\in\llbracket2,s\rrbracket$).
\end{proposition}

\begin{proof}
For the `only if' part, let $P\in X(k)$ be the unique wildly ramified point of $X$. Completing at $P$ and at its image in $X/G$, we obtain a local
integral extension of complete discrete valuation rings of residue field $k$ which is generically Galois with Galois group $G$. Let $w_1<\cdots <w_s$ be the strictly increasing sequence of rational numbers which are the upper jumps of this local $G$-extension. 

We have $q=\frac{|{\rm Cent}_G(\tau)|}{p^s}$. Therefore, by \cite{OP}, Thm.\ 1.1(b), we have $\gcd(m,mw_1)=q$. 
Let $l_1< \cdots< l_s$ and $u_1< \cdots< u_s$ be the lower and upper (respectively) jumps of the local cyclic $H$-subextension. From \cite{OP}, Lem.\ 3.6 we get that for each $i\in\llbracket1,s\rrbracket$ we have $u_i=mw_i$.
As $l_1=u_1$ by Proposition \ref{P3}(1), we have $\gcd(m,l_1)=\gcd(m,u_1)=\gcd(m,mw_1)=q$.
Moreover $l_i\equiv l_1\pmod m$ for each $i\in\llbracket1,s\rrbracket$ by \cite{OP}, Prop.\ 4.2(i). As $l_i-l_{i-1}=p^{i-1}(u_i-u_{i-1})$ for each $i\in\llbracket2,s\rrbracket$ by Proposition \ref{P3}(1) and $\gcd(m,p)=1$, we get that $u_i\equiv u_{i-1}\pmod{m}$ for each $i\in\llbracket2,s\rrbracket$. Hence $u_i\equiv u_1\pmod{m}$ for each $i\in\llbracket1,s\rrbracket$. Thus the `only if' part holds.

For the `if' part, let $u_1<\cdots<u_s$ (resp. $l_1<\cdots<l_s$) be the upper (resp.\ lower) jumps corresponding to $l_1<\cdots<l_s$ (resp.\ $u_1<\cdots<u_s$). As $l_i-l_{i-1}=p^{i-1}(u_i-u_{i-1})$ for each $i\in\llbracket2,s\rrbracket$ by Proposition \ref{P3}(1), the congruence $l_i\equiv l_1\pmod m$ (resp.\ $u_i\equiv u_1\pmod m$) for each $i\in\llbracket1,s\rrbracket$ gives that $u_i\equiv u_1\pmod m$ (resp.\ $l_i\equiv l_1\pmod m$) for each $i\in\llbracket2,s\rrbracket$.

Let $w_i:=\frac{u_i}{m}\in \frac{1}{m}\mathbb N$ for each $i\in\llbracket1,s\rrbracket$. Then $w_1<\cdots<w_s$ and $mw_i=u_i$. We have
\[
\gcd(m,mw_1)=\gcd(m,u_1)=\gcd(m,l_1)=q
\]
and $mw_i\equiv mw_1\pmod m$ for all $i$. Moreover, we have $p\nmid mw_1$ and,  for each $i\in\llbracket2,s\rrbracket$, either $w_i=pw_{i-1}$ or $w_i>pw_{i-1}$ and $p\nmid mw_i$ by Lemma \ref{L7} and $\gcd(p,m)=1$. Thus the four conditions of \cite{OP}, Thm.\ 1.1, are satisfied for $G$ and the rational upper jumps $w_1<\cdots<w_s$. Therefore, there exists a local Galois extension of $k((z))$ with Galois group $G$ whose upper jumps are $w_1<\cdots<w_s$. By \cite{OP}, Lem.\ 3.6, its cyclic $p$-Sylow subextension has upper jumps $u_1<\cdots<u_s$, equivalently, lower jumps $l_1<\cdots<l_s$ by Proposition \ref{P3}(1). By the Harbater--Katz--Gabber Theorem, this local action globalizes to an HKG cover $X\rightarrow X/G$ with lower jumps $l_1<\cdots<l_s$ (resp.\  upper jumps $u_1<\cdots<u_s$). So the `if' part holds.
\end{proof}

\smallskip
We have the following application of Proposition \ref{P10} and of prior results Lemma \ref{L8} and Proposition \ref{P5}(2).

\begin{theorem}\label{C6}
Suppose that the morphism $X\rightarrow X/G$ is an HKG cover. Let $l_1<\cdots<l_s$ be the lower jumps and let $u_1<\cdots<u_s$ be the upper jumps of the cyclic subgroup $H=C_{p^s}$ acting on the completion of the local ring of $X$ at the wildly ramified point. For each $j\in\llbracket0,s\rrbracket$ let $H_{s-j}$ and $Y_j$ be as in Corollary \ref{C4} and, if $j\ge 1$, let 
$$\l_j:=l_j-1.$$ Then the following properties hold. 

\medskip
{\bf (1)} We have $\textup{g}_{Y_0}=0$ and for each $i\in\llbracket1,s\rrbracket$ we have
\[
\textup{g}_{Y_i}=\frac{p-1}{2}\sum_{j=1}^i \l_jp^{i-j}.
\]

{\bf (2)} Suppose that $p$ is odd, $q=1$, and $m$ is even. Let $\mathcal R_X$ be the $B_0[H]$-module of Proposition \ref{P5}(2). Then for each $i\in\llbracket1,s\rrbracket$ we have 
$$\frac{\textup{g}_{Y_i}-\textup{g}_{Y_{i-1}}}{p^i-p^{i-1}}=\frac{\l_i+(p-1)\sum_{j=1}^{i-1}\l_jp^{i-1-j}}{2p^{i-1}}=\frac{u_i-1}{2}\in\mathbb N\cup\{0\},$$
$p-1$ divides $\textup{g}_{Y_i}$, and $p^i-p^{i-1}$ divides $\textup{g}_{Y_i}-\textup{g}_{Y_{i-1}}$. Hence there exists a $B_0[H]$-linear isomorphism 
$$\mathcal R_X\cong\oplus_{i=1}^s \frac{\l_i+(p-1)\sum_{j=1}^{i-1}\l_jp^{i-1-j}}{2p^{i-1}}\mathcal B_i= \oplus_{i=1}^s \frac{u_i-1}{2}\mathcal B_i.$$

{\bf (3)} Let $\mathcal X$ be a smooth projective curve over $\Spec R$ whose reduction modulo $\upsilon$ is $X$ and such that there exists a homomorphism $\varrho:G\rightarrow\Aut(\mathcal X)$ that lifts $\rho$. Then the following properties hold.

\medskip\noindent
{\bf (3.a) (de Jong)} For each subgroup $\Delta$ of $G$, the quotient curve $\mathcal X/\Delta$ is smooth projective over $\Spec R$ with geometrically connected fibers.

\smallskip\noindent
{\bf (3.b)} For each subgroup $\Delta$ of $G$, the special fiber $(\mathcal X/\Delta)_k$ of $\mathcal X/\Delta$ is naturally identified with $X/\Delta$.

\smallskip\noindent
{\bf (3.c) (Chinburg--Guralnick--Harbater)} The KGB obstruction for the HKG cover $X\rightarrow X/G$ vanishes.

\smallskip\noindent
{\bf (3.d)} If $p$ is odd, $q=1$, and $m$ is even, then we have $B[H]$-linear isomorphisms
$$H^0(\mathcal X,\Omega_{\mathcal X})\otimes_R B\cong H^1(\mathcal X,\mathcal O_{\mathcal X})\otimes_R B\cong\Big(\oplus_{i=1}^s \frac{u_i-1}{2}\mathcal B_i\Big)\otimes_{B_0}B.$$
\end{theorem}

\begin{proof}
For part (1), as $Y_0=Y$, we have $\textup{g}_{Y_0}=\textup{g}_Y=0$ by Lemma \ref{L8}. 

Let $i\in\llbracket1,s\rrbracket$. The cover $Y_i\rightarrow Y_0$ is generically Galois of Galois group $H/H_{s-i}$ of order $p^i$. Its upper jumps are $u_1<\cdots<u_i$ by compatibility of upper numbering with quotients (see \cite{S}, Ch.\ IV, Sect.\ 3, Prop.\ 14); hence its lower jumps are $l_1<\cdots<l_i$ by Proposition \ref{P3}(1). It is ramified only over the unique wild branch point $\infty$. By Hilbert's different formula (see \cite{P}, Sect.\ 2.A, pp.\ 288--289 applied to $F^T=E$ and $e_T=i$), the different exponent at $\infty$ is
\[
e_i:=(l_1+1)(p^i-1)+\sum_{j=2}^i(l_j-l_{j-1})(p^{i-j+1}-1).
\]
The Riemann--Hurwitz formula gives
\[
2\textup{g}_{Y_i}-2=-2p^i+e_i.
\]
Simplifying the right hand side we get that part (1) holds. 

For part (2), as $q=1$ and $m$ is even, Proposition \ref{P10} gives that $\gcd(m,u_1)=\gcd(m,l_1)=1$ and that $u_i$ and $l_i$ are odd for each $i\in\llbracket1,s\rrbracket$. So $(u_i-1)/2$ is a nonnegative integer and $\l_j$ is even for each $j\in\llbracket1,s\rrbracket$. Hence $2$ divides $\sum_{j=1}^i \l_jp^{i-j}$ and from part (1) we get that $p-1$ divides $\textup{g}_{Y_i}$.

The identities between integers are clear for $i=1$. For $i\in\llbracket 2,s\rrbracket$, we compute
\begin{align*}
\frac{\textup{g}_{Y_i}-\textup{g}_{Y_{i-1}}}{p^i-p^{i-1}}&=\frac{\sum_{j=1}^i \l_jp^{i-j}-\sum_{j=1}^{i-1} \l_jp^{i-1-j}}{2p^{i-1}}=\frac{\l_i+(p-1)\sum_{j=1}^{i-1}\l_jp^{i-1-j}}{2p^{i-1}}\\
&=\frac{l_i+(p-1)\sum_{j=1}^{i-1}l_jp^{i-1-j}}{2p^{i-1}}-\frac{1+(p-1)\sum_{j=1}^{i-1}p^{i-1-j}}{2p^{i-1}}\\
&=\frac{p^{i-1}u_i}{2p^{i-1}}-\frac{p^{i-1}}{2p^{i-1}}=\frac{u_i-1}{2},
\end{align*}
where the first equality holds by part (1) and the fourth equality holds by Proposition \ref{P3}(3). From this we get that all the identities of part (2) hold. 
For $i\in\llbracket1,s\rrbracket$, as $(u_i-1)/2$ is a nonnegative integer, it follows that $p^i-p^{i-1}$ divides $\textup{g}_{Y_i}-\textup{g}_{Y_{i-1}}$

The $B_0[H]$-linear isomorphism exists by these identities and the definition of $\mathcal R_X$ in Proposition \ref{P5}(2). So part (2) holds.

For part (3.a), each fiber of $\mathcal X/\Delta$ is geometrically connected as it is dominated by a geometrically connected projective fiber of $\mathcal X$. As the quotient morphism $\mathcal X\rightarrow\mathcal X/\Delta$ is surjective with $\mathcal X$ smooth projective over $R$, $\mathcal X/\Delta$ is smooth projective over $\Spec R$ by \cite{dJ}, Prop.\ 4.2 (see also \cite{LL}, Prop. 1.6). So part (3.a) holds.

For part (3.b), as $\Delta$ acts freely on an open subscheme of $\mathcal X$ that contains the generic point of $X$, the fields of fractions of $(\mathcal X/\Delta)_k$ and $X/\Delta$ are naturally identified as $k$-algebras. Thus $(\mathcal X/\Delta)_k$ and $X/\Delta$ are also naturally identified. So part (3.b) holds.

For part (3.c), let $\mathfrak X:=\mathcal X_B$. As the morphism $\Spec B\rightarrow\Spec R$ is flat, we have an identity $(\mathcal X/\Delta)_B=\mathfrak X/\Delta$. As $\mathcal X/\Delta$ is smooth over $\Spec R$, its fiber $(\mathcal X/\Delta)_B=\mathfrak X/\Delta$ over $B$ and its fiber $(\mathcal X/\Delta)_k=X/\Delta$ over $k$ (see part (3.b)) have the same genus. So part (3.c) holds (cf.\ \cite{CGH2}, Thm.\ 4.1).

Part (3.d) follows from part (2) and Proposition \ref{P5}(2).
\end{proof}

\smallskip
Using extra previous results and proofs, we have the following strengthening of Theorem \ref{C6}(3) which shows that it is in fact a consequence of the vanishing of the KGB obstruction.

\begin{theorem}\label{C7}
Suppose that $p$ is odd, $m$ is even, $G$ is non-abelian, and the morphism $X\rightarrow X/G$ is an HKG cover associated to a faithful local action with a vanishing KGB obstruction. Let
$l_1<\cdots<l_s$ and $u_1<\cdots<u_s$ be, respectively, the lower and upper jumps of the wild cyclic subgroup $H=C_{p^s}$. For each $j\in\llbracket0,s\rrbracket$ let $H_{s-j}$ and $Y_j$ be as in Corollary \ref{C4} and, if $j\ge 1$, let $\l_j:=l_j-1$. Let $\mathfrak X$ be a geometrically connected smooth projective curve over a field $\mathfrak B$ of characteristic $0$ equipped with a faithful global action $\varrho_{\phi,\mathfrak B}:G\rightarrow\Aut(\mathfrak X)$ such that for every subgroup $\Delta$ of $G$, the quotient curves $\mathfrak X/\Delta$ and $X/\Delta$ have the same genus. Then the following properties hold.

\medskip
{\bf (1) (Obus)} The character $\chi$ is injective (equivalently, we have $q=1$). 

\smallskip
{\bf (2)} We have
$\mathfrak B[H]$-linear isomorphisms
\begin{equation*}
H^0(\mathfrak X,\Omega_{\mathfrak X/\mathfrak B})\cong H^1(\mathfrak X,\mathcal O_{\mathfrak X})\\
\cong
\Big(\oplus_{i=1}^s
\frac{u_i-1}{2}\mathcal B_{i,\mathbb Q}
\Big)\otimes_{\mathbb Q}\mathfrak B.
\end{equation*}
\end{theorem}

\begin{proof}
We have $u_1\equiv-1\pmod m$ by the first part of \cite{O1}, Prop.\ 5.9. From this and Proposition \ref{P10} or directly from the second part of \cite{O1}, Prop.\ 5.9 we get that part (1) holds.

Let $\mathfrak M_{\mathfrak X}:=H^1_{dR}(\mathfrak X)$. As de Rham cohomology is a Weil cohomology, a pullback--trace argument as in the proof of Proposition \ref{P4}(1) applied to
the quotient morphism $\mathfrak X\rightarrow \mathfrak X/H_{s-j}$ gives a $\mathfrak B$-linear isomorphism
\[
H^1_{dR}(\mathfrak X/H_{s-j})
\xrightarrow{\ \sim\ }
\mathfrak M_{\mathfrak X}^{H_{s-j}}.
\]
As $\textup{g}_{\mathfrak X/H_{s-j}}=\textup{g}_{Y_j}$ by our hypotheses, we get that $\dim_{\mathfrak B}(\mathfrak M_{\mathfrak X}^{H_{s-j}})=2\textup{g}_{Y_j}$. 

Let $\mathfrak V_{\mathfrak X}:=
H^0(\mathfrak X,\Omega_{\mathfrak X})$ and
$\mathfrak W_{\mathfrak X}:=
H^1(\mathfrak X,\mathcal O_{\mathfrak X})$. The Hodge--de Rham
short exact sequence
\[
0\longrightarrow\mathfrak V_{\mathfrak X}\longrightarrow
\mathfrak M_{\mathfrak X}\longrightarrow\mathfrak W_{\mathfrak X}
\longrightarrow0
\]
is $G$-equivariant, and Serre duality identifies $\mathfrak W_{\mathfrak X}$ with the dual of $\mathfrak V_{\mathfrak X}$.

For a non-identity element $h\in H$, let
$\operatorname{Fix}(h)$ be the set of fixed points of its action on the extension of $\mathfrak X$ to an algebraic closure of $\mathfrak B$. For $P\in\operatorname{Fix}(h)$, let $\lambda_P$
be the eigenvalue of $h$ on the tangent space of this extension at $P$. The Lefschetz formula
used in the proof of Theorem \ref{T5}, applied to $h$ and $h^{-1}$, gives that the trace of the $\mathfrak B$-linear map $h:\mathfrak M_{\mathfrak X}\rightarrow \mathfrak M_{\mathfrak X}$ is
$$2-\sum_{P\in\operatorname{Fix}(h)}
\Big(\frac{1}{1-\lambda_P}+\frac{1}{1-\lambda_P^{-1}}\Big)=2-|\operatorname{Fix}(h)|\in\mathbb Z.$$
So the $\mathfrak B[H]$-module $\mathfrak M_{\mathfrak X}$ is rational by the equivalence $(1)\Leftrightarrow (4)$ of Lemma \ref{L3}. 

As in the proof of Corollary \ref{C4} we get a $\mathfrak B[H]$-linear isomorphism
$$\mathfrak M_{\mathfrak X}\cong \bigl(2\textup{g}_{Y_0}\mathcal B_{0,\mathbb Q}\otimes_{\mathbb Q}\mathfrak B\bigr)\bigoplus_{i=1}^s \Big(\frac{2\textup{g}_{Y_i}-2\textup{g}_{Y_{i-1}}}{p^i-p^{i-1}}\mathcal B_{i,\mathbb Q}\otimes_{\mathbb Q}\mathfrak B\Big).$$ 

Lemma \ref{L4} applied over $\mathfrak B$ gives that $\mathfrak V_{\mathfrak X}$
and $\mathfrak W_{\mathfrak X}$ are rational and isomorphic as $\mathfrak B[H]$-modules.
As the short exact sequence splits, we also have
$\mathfrak M_{\mathfrak X}\cong 2\mathfrak V_{\mathfrak X}$. From the last two isomorphisms and $\textup{g}_{Y_0}=0$ by Lemma \ref{L8}, we get first that $\mathfrak V_{\mathfrak X}\cong\oplus_{i=1}^s \Big(\frac{\textup{g}_{Y_i}-\textup{g}_{Y_{i-1}}}{p^i-p^{i-1}}\mathcal B_{i,\mathbb Q}\otimes_{\mathbb Q}\mathfrak B\Big)$ and second that part (2) holds by the identities of Theorem \ref{C6}(2).\end{proof}

\begin{remark}\normalfont\label{R6}
{\bf (1)} Recall that $G=C_{p^s}\rtimes_{\chi} C_m$. The existence of an HKG cover $X\to X/G$ does not, in general, force
$\chi$ to be injective: the compatibility conditions for the lower
jumps in Equation (\ref{EQ61}) allow $q>1$. However, when $G$ is
non-abelian and the KGB obstruction vanishes, $\chi$ is injective by
Obus' result of Theorem \ref{C7}(1). Thus Example \ref{EX11} below is not in the applicability range of Obus Conjecture.

\smallskip
{\bf (2)} If $\chi$ is injective, then for each HKG cover $X\rightarrow X/G$, the $k[G]$-module $V_X$ has a lift $\mathcal V$ to $R_s$ with $r_{\mathcal V}=0$ by Theorem \ref{T6}. So the module-theoretic obstruction \cite{KT2}, Crit.\ 1.6 in the applicability range of Obus Conjecture is removed and replaced by additional evidences for it.

\smallskip
{\bf (3)} Suppose that $p$ is odd, $\chi$ is injective, $m$ is even, and the morphism $X\rightarrow X/G$ is an HKG cover. With the notation of Corollary \ref{C4}, the pullback of global $1$-forms gives an increasing, separated, and exhaustive filtration $\{0\}=V_{Y_0}\subset V_{Y_1}\subset\cdots\subset V_{Y_{s-1}}\subset V_{Y_s}=V_X$ for which condition (1) (resp.\ (2)) of Lemma \ref{L5} holds by Theorem \ref{C6}(2) (resp.\ can be deduced easily from Proposition \ref{P4}(1) and Theorem \ref{C6}(2)).
\end{remark}

\section{Examples}\label{S14}

In this section we include a few examples as follows. 

\begin{ex}\normalfont\label{EX10}
Suppose that $(p,s,m)=(5,2,4)$ and that $\chi:C_4\rightarrow \Aut(H)$ is injective. The upper jumps $u_1=3$ and $u_2=19$ satisfy the conditions of Lemma \ref{L7}: $5\nmid 3$, $19> 5\cdot 3$, and $5\nmid 19$. As $u_1\equiv u_2\equiv -1\pmod 4$, Proposition \ref{P10} gives an HKG cover $X\rightarrow X/G$ whose wild cyclic subgroup $H=C_{25}$ has these upper jumps.
The lower jumps are
$$l_1=u_1=3\;\;\;\textup{and}\;\;\;l_2=l_1+5(u_2-u_1)=3+5\cdot 16=83.$$
So $\l_1=2$, $\l_2=82$, and $\textup{g}_X=2(2\cdot 5+82)=184$ by Theorem \ref{C6}(1) applied to $i=s=2$. 

For $b=a_{1,b}+a_{2,b}5$ with $(a_{1,b},a_{2,b})\in\llbracket0,4\rrbracket^2$, Equation (\ref{EQ51}) gives
$$d_b=\Bigl\lfloor \frac{5\bigl(4+3(4-a_{1,b})\bigr)+4+83(4-a_{2,b})}{25}\Bigr\rfloor.$$
Thus the tuple $(d_0,\ldots,d_{24})$ is equal to
$$(16, 16, 15, 14, 14, 13, 12, 12, 11, 10, 10, 9, 8, 8, 7, 6, 6, 5, 4, 4, 3, 2, 2, 1, 0).$$
Thus from Equation (\ref{EQ57}) or Equations (\ref{EQ53}) and (\ref{EQ55}) we get that
\begin{align*}
V_X=H^0(X,\Omega_X)\cong &\ U_{[3]_4,2}\oplus U_{[2]_4,3}\oplus U_{[1]_4,5}\oplus U_{[0]_4,6}\oplus U_{[3]_4,8}\\
&\oplus U_{[2]_4,9}\oplus U_{[1]_4,11}\oplus U_{[0]_4,12}\oplus U_{[3]_4,14}\oplus U_{[2]_4,15}\\
&\oplus U_{[1]_4,17}\oplus U_{[0]_4,18}\oplus U_{[3]_4,20}\oplus U_{[2]_4,21}\oplus U_{[1]_4,23}.
\end{align*}

As $\chi$ is injective, $\epsilon\in\mathbb Z/4\mathbb Z$ has order $4$. Recall that the composition factors of $U_{a,b}$ are
$$U_a,\;U_{a-\epsilon},\;U_{a-2\epsilon},\;\ldots,\;U_{a-(b-1)\epsilon}.$$
Summing these composition factors over the above $15$ indecomposable direct summands of $V_X$ gives that the multiplicity function $\ell_{V_X}$ is constant, which is a particular case of Theorem \ref{T5}, of constant value $46=\frac{\textup{g}_X}{4}$.
\end{ex}

\begin{ex}\normalfont\label{EX11}
Suppose that $(p,s,m,q)=(5,2,4,2)$; so $\chi$ is non-injective, $d=2$, $G=C_{25}\rtimes_{\chi} C_4$, $|G|=100$, and $G/Q(G)\cong D_{25}$. We show that in such a case there exists an HKG cover $X\rightarrow X/G$. Let $z:=x^{-1}$, $K:=k((z))$, 
$$\underline{z}:=(z^{-2},z^{-14})\in W_2(K).$$
and $L:=K(y_1,y_2)$ be defined by $\wp(y_1,y_2)=\underline{z}$. So $z^{-2}\notin\{\beta^p-\beta|\beta\in K\}$ and the Witt vector $\underline{z}$ of length $2$ is in standard form in the sense of Lemma \ref{L10}, applied with local parameter $z$, as the pole orders $2$ and $14$ at infinity are not divisible by $5$. Hence $K\rightarrow L$ is a cyclic extension of degree $25$ with Galois group $H=C_{25}$, and Lemma \ref{L10} gives the upper jumps
$$u_1=2\qquad\textup{and}\qquad u_2=\max(5\cdot 2,14)=14.$$
Let $\sigma$ act $k$-linearly and continuously on $K$ by $\sigma(z):=\overline{\zeta}_4^{-1}z$.
As $\overline{\zeta}_4^2=-1$, we have
$\overline{\zeta}_4^{2}=\overline{\zeta}_4^{14}=-1$, and hence
$$\sigma(\underline{z})=((\overline{\zeta}_4^{-1}z)^{-2},(\overline{\zeta}_4^{-1}z)^{-14})=(-z^{-2},-z^{-14}).$$
This is also $(-1)\cdot\underline{z}$ in $W_2(K)$. Indeed, the unit
$-1\in W_2(\mathbb F_5)$ is the Teichm\"uller vector $(-1,0)$, and the
length-two Witt multiplication formula
\[
(x_0,x_1)(y_0,y_1)=(x_0y_0,x_0^5y_1+y_0^5x_1+5x_1y_1)
\]
shows that multiplication by a Teichm\"uller lift of $\gamma\in K$ satisfies
\[
(\gamma,0)(y_0,y_1)=(\gamma y_0,\gamma^5y_1);
\]
see also \cite{L}, Ch.\ 26, Sec.\ 3, Eq.\ (29). Thus
$\sigma(\underline{z})=(-1)\cdot\underline{z}$ in $W_2(K)$. 

By Lemma \ref{L9}, applied to the unit $[-1]_{25}\in (\mathbb Z/25\mathbb Z)^{\times}$, $\sigma$ extends to $L$ by $\sigma(y_1,y_2)=[-1]_{25}\cdot(y_1,y_2)$. This extension has order $4$ and acts on the translations in $H$ by $T_{[b]_{25}}\mapsto T_{-[b]_{25}}$, i.e., by additive inversion. Therefore $G=C_{25}\rtimes_{\chi} C_4$ is a subgroup of $\Aut(L)$, where $\chi:C_4\rightarrow \Aut(C_{25})\rightarrow \mathbb F_5^{\times}$ is non-injective of image $\{[1]_5,[4]_5\}\cong C_2$. The resulting local faithful action $G\rightarrow\Aut(L)$ is realized by an HKG cover $X\rightarrow X/G$ by Theorem \ref{T4}.

The corresponding lower jumps are
$$l_1=u_1=2\;\;\;\textup{and}\;\;\;l_2=l_1+5(u_2-u_1)=2+5(12)=62.$$
For $b=a_{1,b}+a_{2,b}5$ with $(a_{1,b},a_{2,b})\in\llbracket0,4\rrbracket^2$, Equation (\ref{EQ51}) gives
$$d_b=\Bigl\lfloor \frac{5\bigl(4+2(4-a_{1,b})\bigr)+4+62(4-a_{2,b})}{25}\Bigr\rfloor.$$
Thus the tuple $(d_0,\ldots,d_{24})$ is equal to
\[
\begin{aligned}
(12,12,11,11,10,10,9,9,8,8,7,7,6,6,5,5,4,4,3,3,2,2,1,1,0).
\end{aligned}
\]
Thus Equations (\ref{EQ53}) and (\ref{EQ55}) give that
\begin{align*}
V_X=H^0(X,\Omega_X)\cong &\ U_{[3]_4,2}\oplus U_{[2]_4,4}\oplus U_{[1]_4,6}\oplus U_{[0]_4,8}\oplus U_{[3]_4,10}\oplus U_{[2]_4,12}\\
&\oplus U_{[1]_4,14}\oplus U_{[0]_4,16}\oplus U_{[3]_4,18}\oplus U_{[2]_4,20}\oplus U_{[1]_4,22}.
\end{align*}

As $\chi(\sigma)=\overline{\zeta}_4^2=-1$, we check that we have an identity
\begin{equation}\label{EQ61.5}
\ell_{V_X}(a)=\sum_{b=0}^{24} \mu_{b,a+2[b]_4}.
\end{equation}
By Equation (\ref{EQ52}) and $\epsilon=[2]_4$, the multiplicity of $T_a$ in the quotient
$V_X^{(b+1)}/V_X^{(b)}$ is $\mu_{b,c}$ for the unique $c\in C_4$ such that
$c-2[b]_4=a$, namely $c=a+2[b]_4$. As we have a $k[\Sigma]$-linear isomorphism $V_X\cong\oplus_{b=0}^{24} V_X^{(b+1)}/V_X^{(b)}$, by computing the multiplicities we get that Equation (\ref{EQ61.5}) holds.

Evaluating these sums gives the values
\begin{equation}\label{EQ61.6}
\ell_{V_X}([0]_4)=\ell_{V_X}([2]_4)=30\;\;\;\textup{and}\;\;\;\ell_{V_X}([1]_4)=\ell_{V_X}([3]_4)=36.
\end{equation}
Thus $\ell_{V_X}$ is non-constant which also follows from Theorem \ref{T5}. 

For $j=1$ (resp.\ $j=2$), the $G/Q(G)$-module $M'_j$ as in Lemma \ref{L1} has a lift $\mathcal M'_j$ to $R_2$ with $r_{\mathcal M'_j}=0$ by the right-hand (resp.\ left-hand) side of Display (\ref{EQ61.6}) and Theorem \ref{T2}(2) applied to $G/Q(G)$. Thus $V_X$ has a lift $\mathcal V_X$ to $R_2$ with $r_{\mathcal V_X}=0$ by Lemma \ref{L1}.
\end{ex}

Next we exemplify the sets $\mathbb J_X=\sqcup_{j=0}^{s-1} \mathbb J_{X;j}$ introduced in Proposition \ref{P8}(5.d) and (6) and the multiplicity type of the $\mu_X$s introduced in Definition \ref{D6}(5) in a way that supplements Example \ref{EX9.5}(1) and (2).

\begin{ex}\normalfont\label{EX12}
We assume that $m\ge 2$, $X\rightarrow X/G$ is an $HKG$ cover, and $s\in\{1,2\}$. We use the notation of Proposition \ref{P8}. In particular, we have $\lfloor\frac{l_1}{p}\rfloor=m\l_{0,0}+\l_{0,-1}$ with $(\l_{0,0},\l_{0,-1})\in (\mathbb N\cup\{0\})\times\llbracket0,m-1\rrbracket$.

\medskip
{\bf (1)} Suppose that $s=1$. We have $\delta_1(p)=0$. For $t\in\llbracket0,p-1\rrbracket$ we have
\[
\delta_1(t)=
\left\lfloor\frac{p-1+(p-1-t)l_1}{p}\right\rfloor=l_1+1+\left\lfloor\frac{-1-l_1-tl_1}{p}\right\rfloor\in\llbracket1,l_1\rrbracket.
\]
In particular, $\delta_1(p-2)=0$, and, as we have $p\nmid l_1$ by Corollary \ref{C3}(2.a), $\delta_1(p-2)=\left\lfloor\frac{l_1}{p}\right\rfloor+1$ (also by Proposition \ref{P8}(2)). 
For each $b\in\llbracket1,p-1\rrbracket$, by taking $t\in\{b-1,b\}$, we get the following concrete form 
\begin{equation}\label{EQ61.9}
\delta_1(b-1)-\delta_1(b)=\left\lfloor\frac{-1-bl_1}{p}\right\rfloor-\left\lfloor\frac{-1-l_1-bl_1}{p}\right\rfloor\in\left\{\left\lfloor\frac{l_1}{p}\right\rfloor,\left\lfloor\frac{l_1}{p}\right\rfloor+1\right\}
\end{equation}
of Display (\ref{EQ60}).

We have $\mu_X([0]_m,p-1)=\l_{0,0}$ by Proposition \ref{P8}(2). By Proposition \ref{P8}(5d), let the subset $\mathbb J_X=\mathbb J_{X;0}$ of $C_m\times\llbracket1,p-1\rrbracket$ be such that
$$\mu_X(a,b)=\begin{cases}\l_{0,0}+1\quad\;\textup{if}\quad (a,b)\in\mathbb J_X\\
\l_{0,0}\quad\quad\quad\textup{if}\quad (a,b)\in(C_m\times\llbracket1,p-1\rrbracket)\setminus\mathbb J_X.\end{cases}$$

Proposition \ref{P8}(6) and Theorem \ref{C6}(1) applied to $i=1$ give that
\begin{equation}\label{EQ62}
\begin{aligned}
\sum_{(a,b)\in\mathbb J_X} b&=\textup{g}_X-\frac{mp(p-1)\l_{0,0}}{2}=\frac{(p-1)(l_1-1)}{2}-\frac{mp(p-1)\l_{0,0}}{2}\\
&=\frac{p-1}{2}(l_1-1-pm\l_{0,0}).
\end{aligned}
\end{equation}

As we have $([0]_m,p-1)\notin\mathbb J_X$ by Proposition \ref{P8}(2), we fix an arbitrary pair $(a,b)\in (C_m\times\llbracket1,p-1\rrbracket)\setminus\{([0]_m,p-1)\}$. To determine if $(a,b)$ does or does not belong to $\mathbb J_X$ we consider three disjoint cases as follows. 

{\bf Case 1: $\l_{0,-1}=0$.} Thus $\delta_1(b-1)-\delta_1(b)\in\{m\l_{0,0},m\l_{0,0}+1\}$ for each $b\in\llbracket1,p-1\rrbracket$ by Display (\ref{EQ61.9}) and, as $p\nmid l_1$, there exists $\l'_{0,-1}\in\llbracket1,p-1\rrbracket$ such that $l_1=pm\l_{0,0}+\l'_{0,-1}$. For $n\in\llbracket1,\l'_{0,-1}\rrbracket$, the natural number 
\[
b_n:=\Bigl\lfloor\frac{np-1}{\l'_{0,-1}}\Bigr\rfloor\in\llbracket1,p-1\rrbracket
\]
is the only one in the set $\llbracket1,p\rrbracket$ for which $\delta_1(b_n-1)-\delta_1(b_n)=m\l_{0,0}+1$ and $\delta_1(b_n)=\l'_{0,-1}-n+(p-1-b_n)m\l_{0,0}$. As for $n=\l'_{0,-1}$ we have $([\l'_{0,-1}-n]_m,b_n)=([0]_m,p-1)$, we have $(a,b)\in\mathbb J_X$ if and only if there exists $n\in\llbracket1,\l'_{0,-1}-1\rrbracket$ such that $b=b_n$ and $\r_{a}=\r_{\delta_1(b)}$, i.e., $a=[\l'_{0,-1}-n]_m$ by Proposition \ref{P8}(1). We conclude that
$$\mathbb J_X=\{([\l'_{0,-1}-n]_m,b_n)|n\in\llbracket1,\l'_{0,-1}-1\rrbracket\}.$$
Thus $|\mathbb J_X|=\l'_{0,-1}-1$ and for each $b\in\llbracket1,p-1\rrbracket$ we have 
$$0\le|\mathbb J_X\cap (C_m\times\{b\})|\le 1.$$
So $\mu_X$ has multiplicity type $(\l_{0,0},\l'_{0,-1}-1)=(\l_{0,0},l_1-pm\l_{0,0}-1)$ and jump type $1$ if $\l'_{0,-1}\ge 2$ and jump type $0$ if $\l'_{0,-1}=1$. Equation (\ref{EQ62}) becomes
$$\sum_{(a,b)\in\mathbb J_X} b=\frac{(p-1)(\l'_{0,-1}-1)}{2}.$$

{\bf Case 2: $\l_{0,-1}=m-1$.} By Display (\ref{EQ61.9}) we have the belonging relation $\delta_1(b-1)-\delta_1(b)\in\{m\l_{0,0}+m-1,m(\l_{0,0}+1)\}$ for each $b\in\llbracket1,p-1\rrbracket$. We have $(a,b)\notin\mathbb J_m$ if and only if $\delta_1(b-1)-\delta_1(b)=m\l_{0,0}+m-1$ and $\r_{a}=\r_{\delta_1(b)-1}$ by Proposition \ref{P8}(1). Thus for each $b\in\llbracket1,p-2\rrbracket$ we have 
$$0\le|(C_m\times\{b\})\setminus \mathbb J_X|\le 1;$$
moreover, as $\delta_1(p-2)-\delta_1(p-1)=m\l_{0,0}+m$, we have the identity of sets $(C_m\times\{p-1\})\setminus \mathbb J_X=\{([0]_m,p-1)\}$.
So 
$$(m-1)(p-1)\le|\mathbb J_X|\le m(p-1)-1=(m-1)(p-1)+(p-2).$$ The function $\mu_X$ has jump type $-1$ for $m\ge 3$ and jump type $1$ for $m=2$.  

{\bf Case 3: $\l_{0,-1}\in\llbracket1,m-2\rrbracket$.} By Display (\ref{EQ61.9}) we have strict inequalities $m\l_{0,0}<\delta_1(b-1)-\delta_1(b)<(m+1)\l_{0,0}$ for each $b\in\llbracket1,p-1\rrbracket$; thus $\r_{1,b}>0$ for each $b\in\llbracket1,p-1\rrbracket$. By Proposition \ref{P8}(1), with its notation we get that 
$$\mathbb J_X=\cup_{b=1}^{p-1} \bigl\{(a,b)|a\in [\delta_1(b)]_m+\{[i]_m|i\in\llbracket0,\r_{1,b}-1\rrbracket\}\bigr\}.$$
Thus for each $b\in\llbracket1,p-1\rrbracket$ we have 
$$1\le|(C_m\times\{b\})\setminus \mathbb J_X|\;\;\;\textup{and}\;\;\;1\le|(C_m\times\{b\})\cap \mathbb J_X|.$$
So $\mu_X$ has jump type $0$ and $|\mathbb J_X|=\sum_{b=1}^{p-1} \r_{1,b}$.

\smallskip
{\bf (2)} Suppose that $s=2$. For $b\in\llbracket1,p^2-1\rrbracket$, let $N_{2,b}$ be as in Proposition \ref{P3}(2); we have $N_{2,b-1}-N_{2,b}=pl_1$ if $p\nmid b$ and $N_{2,b-1}-N_{2,b}=l_2-(p-1)pl_1$ if $p\mid b$.  As $p\nmid l_1$ and $p\mid l_2-l_1$ by Corollary \ref{C3}(2.a) and (2.b), we have $l_2-(p-1)pl_1\neq pl_1$. Let $\l_{1,2}:=\frac{l_2-l_1-(p-1)pl_1}{p}$: so, either $\l_{1,2}=0$ or $\l_{1,2}\in\mathbb N\setminus p\mathbb N$ by Corollary \ref{C3}(2.b) and (2.c). Hence the differences
$$N_{2,b-1}-N_{2,b}=\begin{cases}pl_1\quad\quad\quad\quad\textup{if}\quad p\nmid b\\
l_1+p\l_{1,2}\quad\;\;\textup{if}\quad p\mid b\end{cases}$$
take two distinct values and either one can be the greater one. The pair $(\l_{1,0},\l_{1,-1})\in (\mathbb N\cup\{0\})\times\llbracket0,m-1\rrbracket$ is such that $\big\lfloor\frac{l_1}{p^2}+\frac{\l_{1,2}}{p}\big\rfloor=m\l_{1,0}+\l_{1,-1}$.

To provide a concrete example we assume that $m\ge 3$ and $l_1<p$. Thus $\frac{l_1}{p^2}<\frac{1}{p}$ and $\l_{0,0}=\l_{0-1}=0$. As $\frac{l_1}{p^2}<\frac{1}{p}$, we have $\big\lfloor\frac{\l_{1,2}}{p}\big\rfloor=m\l_{1,0}+\l_{1,-1}$. If $\l_{1,-1}=0$ (resp.\ $\l_{1,-1}=m-1$ or $\l_{1,-1}\in\llbracket1,m-2\rrbracket$), then as in the cases of part (1) we argue that the jump type of $\mu_X$ is $(1,1)$ (resp.\ $(1,-1)$ or $(1,0)$).\end{ex}

\begin{ex}\normalfont\label{EX13}
We illustrate the jump and multiplicity types when $p=5$, $m=2$, $q=1$, and $s\in\{1,2,3\}$. The upper jumps
$u_1=3$, $u_2=19$, and $u_3=107$ satisfy Lemma \ref{L7}, and the corresponding lower jumps are
\[
l_1=3,\;\;\; l_2=3+5(19-3)=83,\;\;\;\textup{and}\;\;\;
l_3=83+25(107-19)=2283.
\]
They are all odd, so Proposition \ref{P10} gives HKG covers $X\rightarrow X/G$ with $G=D_5$ if $s=1$, $G=D_{25}$ if $s=2$, and $G=D_{125}$ if $s=3$. We use the notation of Proposition \ref{P8} and Example \ref{EX12}(2). As $\lfloor\frac{3}{5}\rfloor=0$, we have $\l_{0,0}=\l_{0,-1}=0$. Moreover, $\l_{1,2}=4$ and hence $\l_{1,0}=\l_{1,-1}=0$.

For $s=1$, Equation (\ref{EQ51}) gives $(d_0,d_1,d_2,d_3,d_4)=(3,2,2,1,0)$. From the Case 1 of Example \ref{EX12} we get that $\mathbb J_X=\{([0]_2,1),([1]_2,3)\}$; so $\mu_X$ has jump type $1$ and multiplicity type $(0,2)$.

For $s=2$, using $(l_1,l_2)=(3,83)$, the tuple $(d_0,d_1,\ldots,d_{24})$ is equal to 
$$(16,16,15,14,14,13,12,12,11,10,10,9,8,8,7,6,6,5,4,4,3,2,2,1,0).$$
Using Equation (\ref{EQ54}) and then Equation (\ref{EQ55}), the non-zero
values of the $\mu_X(a,b)$s are $\mu_X([0]_2,3r+3)=1$ for $r\in\llbracket0,6\rrbracket$ and
$\mu_X([1]_2,3r+2)=1$ for $r\in\llbracket0,7\rrbracket$; so $\mu_X$ has jump type $(1,1)$ and multiplicity type $(0,12,0,3)$.

For $s=3$, using $(l_1,l_2,l_3)=(3,83,2283)$, we have
$$\Big\lfloor\frac{l_3}{p^3}-(p-1)\Big(\frac{l_1}{p}+\frac{l_2}{p^2}\Big)\Big\rfloor=\Big\lfloor\frac{323}{125}\Big\rfloor=2.$$ Hence $\l_{3,0}=1$ and $\l_{3,-1}=0$. The tuple $(d_0,\ldots,d_{124})$ is equal to
\[
\begin{aligned}
&(89,89,88,87,87,86,85,85,84,84,83,82,81,81,80,79,79,78,77,77,76,75,\\
&75,74,74,71,70,70,69,69,68,67,66,66,65,64,64,63,63,62,61,60,60,59,59,\\
&58,57,56,56,55,53,52,52,51,50,49,49,48,48,47,46,45,45,44,44,43,42,42,\\
&41,40,39,39,38,38,37,34,34,33,33,32,31,31,30,29,29,28,27,27,26,25,24,\\
&24,23,23,22,21,21,20,19,19,16,16,15,14,14,13,12,12,11,10,10,9,8,8,7,\\
&6,6,5,4,4,3,2,2,1,0).
\end{aligned}
\]
Using Equation (\ref{EQ54}) and then Equation (\ref{EQ55}), the non-zero values of the $\mu_X(a,b)$s are precisely the following. Let
\[
\begin{aligned}
S_0:=\{&2,5,8,11,14,17,20,23,25,26,30,32,35,39,41,45,\\
       &47,50,51,54,57,60,63,66,69,72,79,82,85,88,90,94,\\
       &97,103,106,109,112,115,118,121\}\\
   \end{aligned}
\]   
and
\[
\begin{aligned}       
S_1:=\{&3,6,10,12,15,18,21,28,31,34,37,40,43,46,49,50,\\
       &53,55,59,61,65,68,70,74,75,77,80,83,86,89,92,95,\\
       &98,100,102,105,108,111,114,117,120,123\}.
\end{aligned}
\]
Then $\mu_X([0]_2,b)=1$ for each $b\in S_0$, $\mu_X([1]_2,b)=1$ for each $b\in S_1$, and $\mu_X([1]_2,25)=\mu_X([0]_2,75)=\mu_X([0]_2,100)=2$; so $\mu_X$ has jump type $(1,1,1)$ and multiplicity type $(0,58,0,19,1,3)$. For instance, as $S_0$ and $S_1$ each has $29$ elements not divisible by $5$, the second entry is $58$.
\end{ex}

\begin{ex}\normalfont\label{EX14}
Suppose that $(p,s,m,q)=(5,3,2,1)$. So $G\cong D_{125}$. We define ten subsets of the set $\mathbb J_{2,125}=\mathbb Z/2\mathbb Z\times\llbracket1,125\rrbracket$ as follows. Let
$$J_1:=\mathbb Z/2\mathbb Z\times\{2,4,8,12,18,22,28,32,34,38,42,46,48,52,54,58,62,66,68\},$$
$$J_2:=\mathbb Z/2\mathbb Z\times\{72,78,82,88,92,96,98,102,104,108,112,116,118,122\},$$
$$J_3:=\{[1]_2\}\times\{6,24,26,44,64,84,86,106,124\},$$
$$J_4:=\{[0]_2\}\times\{14,16,36,56,74,76,94,114\},$$
$$J_5:=\mathbb Z/2\mathbb Z\times\{1,7,9,11,13,17,19,21,23,27,29,31,37,39,41,43,49,51,57,59\},$$
$$J_6:=\mathbb Z/2\mathbb Z\times\{61,63,69,71,73,77,79,81,83,87,89,91,93,99,101,107,109\},$$
$$J_7:=\mathbb Z/2\mathbb Z\times\{111,113,119,121,123\},$$
$$J_8:=\{[1]_2\}\times\{3,15,35,47,55,67,75,95,103,115\},$$
$$J_9:=\{[0]_2\}\times\{5,25,33,45,53,65,85,97,105,117\},$$
and $ J_0:=\cup_{i=1}^9 J_i$. For $i\in\llbracket0,9\rrbracket$, let $V_i:=\oplus_{(a,b)\in J_i} U_{a,b}$. So $V_0$ is the $k[G]$-module considered in \cite{KT2}, Sect.\ 5.1, and we have $V_0=\oplus_{i=1}^9 V_i$. From Example \ref{EX1} we get that $\ell_{V_i}([0]_2)-\ell_{V_i}([1]_2)=0$ for each $i\in\llbracket1,7\rrbracket$ and that $\ell_{V_8}([0]_2)-\ell_{V_8}([1]_2)=-10=-\ell_{V_9}([0]_2)+\ell_{V_9}([1]_2)$. Therefore we have $\ell_{V_0}([0]_2)-\ell_{V_0}([1]_2)=0$, i.e., the character function $\ell_{V_0}$ is constant. From this and the fact that $\mu_{V_0}([0]_2,125)=\mu_{V_0}([1]_2,125)=0$, we get that $V_0$ has a lift $\mathcal V$ to $R_3$ with $r_{\mathcal V}=0$ by Theorem \ref{T2}(2). Moreover, we get that $\mu_X$ has jump type $(1,1,1)$ and multiplicity type $(0,175,0,10,0,2)$.\end{ex}

\medskip\noindent
{\bf Acknowledgement.}
The second author thanks SUNY Binghamton for good working conditions and Ofer Gabber for outlining the proof of Proposition \ref{P4}(1), the suggestion to use the Krull--Schmidt--Remak--Azumaya Theorem for $R[H]$-modules instead of projective resolutions over $R[[y]]$, and the reference \cite{dJ}, Prop.\ 4.2. 

\bibliographystyle{alpha}
 
\bigskip

\hbox{Huy Dang,\;\;\;E-mail: hdang2@binghamton.edu}
\hbox{Address: Department of Mathematics and Statistics, Binghamton University,}
\hbox{P. O. Box 6000, Binghamton, New York 13902-6000, U.S.A.}

\bigskip
\hbox{Adrian Vasiu,\;\;\;E-mail: avasiu@binghamton.edu}
\hbox{Address: Department of Mathematics and Statistics, Binghamton University,}
\hbox{P. O. Box 6000, Binghamton, New York 13902-6000, U.S.A.}
\end{document}